\documentclass[11pt]{article}

\usepackage{stmaryrd}

\usepackage[ansinew]{inputenc}
\usepackage{amsfonts}
\usepackage{amssymb,amsmath}
\usepackage{latexsym}
\usepackage{graphics}
\usepackage{epic}
\usepackage{epsfig}
\usepackage{amsthm}
\usepackage{psfrag}
\usepackage{pict2e}
\usepackage{subfigure}
\usepackage{bbm}
\usepackage{dsfont}
\usepackage{color}
\definecolor{gray5}{gray}{0.8}
\definecolor{gray1}{gray}{0.4}
\definecolor{gray2}{gray}{0.6}
\definecolor{gray3}{gray}{0.7}
\definecolor{gray4}{gray}{0.3}

\numberwithin{equation}{section}

\renewcommand {\l}{\ell}
                
\newtheorem     {thm}{Theorem}[section]

\newtheorem     {lem}[thm]{Lemma}
\newtheorem     {prop}[thm]{Proposition}
\newtheorem     {cor}[thm]{Corollary}

\DeclareMathOperator{\cov}{\mathbb{C}ov}
\newtheorem     {rem}[thm]{Remark}

\renewcommand{\textcolor}[2]{#2}

\begin{document}

\title{Moments of the frequency spectrum of a splitting tree with neutral Poissonian mutations.} 

\author{\textsc{Nicolas Champagnat$^{1,2}$, Beno\^it Henry$^{1,2}$}}

\footnotetext[1]{TOSCA project-team, INRIA Nancy -- Grand Est, IECL -- UMR 7502,
  Nancy-Universit\'e, Campus scientifique, B.P.\ 70239, 54506 Vand\oe uvre-l\`es-Nancy Cedex,
  France}

\footnotetext[2]{IECL -- UMR 7502,
Nancy-Universit\'e, Campus scientifique, B.P.\ 70239, 54506 Vand\oe uvre-l\`es-Nancy Cedex,
  France, E-mail: \texttt{Nicolas.Champagnat@inria.fr}, \texttt{benoit.henry@univ-lorraine.fr}
}
\date{}
\maketitle

\begin{abstract}

We consider a branching population where individuals live and reproduce independently. Their lifetimes are i.i.d.\ and they give birth at a constant rate $b$. The genealogical tree spanned by this process is called a splitting tree, and the population counting process is a homogeneous, binary Crump-Mode-Jagers process.
We suppose that mutations affect individuals independently at a constant rate $\theta$ during their lifetimes, under the infinite-alleles assumption: each new mutation gives a new type, called allele, to his carrier. We study the allele frequency spectrum which is the numbers $A(k,t)$ of types represented by $k$ alive individuals in the population at time $t$.
Thanks to a new construction of the coalescent point process describing the genealogy of individuals in the splitting tree, we are able to compute recursively all joint factorial moments of $\left(A(k,t)\right)_{k\geq 1}$. These moments allow us to give an elementary proof of the almost sure convergence of the frequency spectrum in a supercritical splitting tree.

\end{abstract}          
\bigskip

\noindent {\it MSC 2000 subject classifications:} Primary 60J80; secondary 
 92D10, 60J85, 60G51, 60G57, 60F15.\\

\noindent \textit{Key words and phrases.}  branching process -- coalescent point process -- splitting tree -- Crump--Mode--Jagers process -- linear birth--death process -- allelic partition -- frequency spectrum -- infinite alleles model -- Lévy process -- scale function  -- random measure -- Palm measure -- Campbell's formula.

\section{Introduction}
\label{sec:intro}

In this work, we study a branching population in which every individual is supposed to have a lifetime independent from the other individuals in the population. Moreover, during their lifetimes, they give birth to new individuals at Poisson rate. The genealogical tree underlying the history of the population, the so called splitting tree, has been widely studied in the past \cite{L10, GJ, GK}.

%blablasur la comparaison de deux ind

In our model, individuals also experience mutations at Poisson rate. Each mutation leads to a totally new type replacing the previous type of the individual, this is the \emph{infinitely-many alleles} assumption. Every time an individual gives birth to a new individual, it transmits its type to his child. This mutation process \textcolor{red}{may model the occurrence of a new species in an area or a new phenotype in a given species.} Our study concerns the allelic partition of the living population at a fixed time $t$, which is characterized by the frequency spectrum $\left(A(k,t)\right)_{k\geq 1}$ of the population, where each integer $A(k,t)$ is the number of families represented by $k$ alive individuals at time $t$. A famous example is the Ewen's sampling formula which gives the distribution of the frequency spectrum when the genealogy is given by the Kingman coalescent model \cite{EV}.  Other works studied similar quantities in the case of Galton-Waston branching processes (see \cite{Ber} or \cite{Gri}). The purpose of this work is to obtain explicit formulas for the moments of the frequency spectrum.

The model with Poissonian mutations was studied in Champagnat and Lambert \cite{CL1,CL2}, where many properties of the frequency spectrum and the clonal family (the family who carries the type of the ancestral individual at time $0$) were obtained.
The population counting process $\left(N_{t}, \ t\in\mathbb{R}_{+}\right)$ and the frequency spectrum $\left(A(k,t)\right)_{k\geq 1}$ belong to the class of general branching processes counted by random characteristics. This class of processes has been deeply studied by Jagers and Nerman, who give, for instance, criteria for the long time convergence of such processes \cite{J,N,JNa,JNb,taib}. Using these tools, Richard and Lambert \cite{L10, rich} shown the almost sure convergence of $N_{t}$, properly renormalized, to an exponential random variable in the supercritical case. The almost sure convergence of the ratios $\frac{A(k,t)}{N_{t}}$ was proved in \cite{CL1} using similar tools. From this, one can easily deduce the a.s.\ convergence of $\frac{A(k,t)}{W(t)}$ \textcolor{red}{where $W(t)$ is the average number of individuals at time $t$ conditionally on $N_{t}>0$.} This result was stated without proof in \cite{CLR}.

An important tool is the so called \emph{coalescent point process} (CPP): given the individuals alive at a fixed time, the coalescent point process at time $t$ is the tree describing the relation between the lineages of all individuals alive at time $t$. Here, the term lineage of an individual refers to the succession of individuals, from child to parent, backward in time until the ancestor of the population. Roughly speaking, the CPP is the genealogical tree of the lineages of the individuals. This tool goes back to Aldous and Popovic \cite{AP} who introduced it for a Markovian model. Later in \cite{L10}, Lambert showed the general link between coaslescent point processes and splitting trees.

In this work, we use the representation of the CPP of a splitting tree as an i.i.d.\ sequence of random variables $\left(H_{i}\right)_{i\geq 1}$.
We introduce a new construction of the coalescent point process. Thanks to  a new formula for the expectation of an integral w.r.t.\ a random measure with specific independence structure, this allows us to obtain explicit recursive formulas for the moments of the frequency spectrum, valid for any parameter of the model. As an application, we prove the almost sure convergence of the frequency spectrum avoiding the use of the theory of general branching processes counted by random characteristics in the supercritical case. Of course, these moment formulas can also provide many valuable informations, for instance, on the error in the aforementioned convergence, which suggest CLT-type results. Indeed, such results can be proved \cite{H} but leads to many additional difficulties.

Section \ref{sec:models} is dedicated to the description of the models and the introduction of the classical tools (from \cite{L10}) used in the sequel. \textcolor{red}{  In Section \ref{sec:MainRes} we state our main results (Theorems \ref{thm: simple factorial moment} and \ref{thm: multiple factorial moment}) giving explicit formulas for the factorial moments of the frequency spectrum $\left(A(k,t) \right)_{k\geq 1}$ expressed in terms of the lower order moments.}
Section \ref{sec: ranMes} is devoted to the proof of an extension of the Campbell formula concerning the expectation of the integral of a random process with respect to a random measure when both objects present some local independence properties. Even if this result is used as a tool for this work, it is interesting by itself.
\textcolor{red}{Section \ref{sec:moments} is devoted to the proof of the moments formulas stated in Section \ref{sec:MainRes}}. We give the key decomposition of the CPP in Subsection \ref{ssec:construc}. The rest of Section \ref{sec:moments} is dedicated to the proofs of theorems \ref{thm: simple factorial moment} and \ref{thm: multiple factorial moment}. In particular, we provide a computation of the first moment much simpler than the one of \cite{CL1}. We give the asymptotic behaviour of higher moments in Section \ref{sec:asypmom}. Section \ref{sec:finalCV} is dedicated to the proof the following law of large numbers:
\[
\lim\limits_{t\to \infty}\frac{A(k,t)}{W(t)}=c_{k}\mathcal{E}, \qquad \text{almost surely}, \qquad k\geq1,
\]
where $\mathcal{E}$ is an exponential random variable with parameter 1 conditionally on non-extinction, and the constants $c_{k}$ are explicit.

\section{Splitting trees and the coalescent point process}
\label{sec:models}
We study a branching model of population dynamics called splitting tree where individuals live and reproduce independently from each other. Their lifetimes are i.i.d.\ following a given arbitrary distribution $\mathbb{P}_{V}$ on $(0,\infty]$. During this lifetime, an individual gives birth to new individuals, with binary reproduction (i.e. new individuals appear singly), at independent Poisson times with positive constant rate $b$ until his death. We also suppose that the population starts with a single individual called the {\it root} or {\it ancestor}. A graphical representation of a splitting tree is shown in Figure \ref{fig: tree}.

The finite measure $\Lambda:=b\mathbb{P}_{V}$ is called the \emph{lifespan measure}, and plays an important role in the study of the model.

Moreover, we assume that individuals undergo mutations at Poisson times with rate $\theta$ during their lifetimes independently from each other and from their reproduction processes. Each new mutation leads to a brand new type replacing the preceding type of the individual ({\it infinitely many alleles model}). Parents yield their current type to their children.

A family at a given time $t$ is a set of alive individuals carrying the same type at time $t$. Our purpose is to study the distribution of the sizes of families in the population at time $t$.
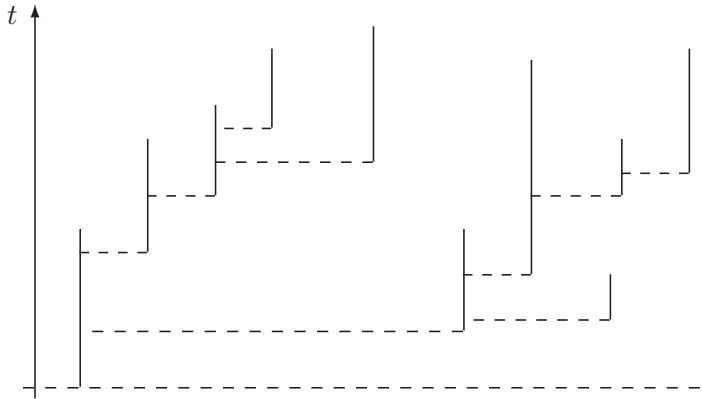
\begin{figure}[ht]
\unitlength 2mm % = 4.55pt
\linethickness{0.4pt}

%TeXCAD (http://texcad.sf.net/) Picture. File: [smalltree.tex]. Options on following lines.
%\grade{\on}
%\emlines{\off}
%\epic{\off}
%\beziermacro{\on}
%\reduce{\on}
%\snapping{\off}
%\quality{8.000}
%\graddiff{0.005}
%\snapasp{1}
%\zoom{8.0000}
\unitlength 1.5mm % = 2.845pt
\linethickness{0.6pt}
\ifx\plotpoint\undefined\newsavebox{\plotpoint}\fi % GNUPLOT compatibility
\begin{picture}(68,40)(-10,0)
\put(15,5){\vector(0,1){35}}
%\dashline{1}(14,6)(74,6)
\put(13.93,5.93){\line(1,0){.9836}}
\put(15.897,5.93){\line(1,0){.9836}}
\put(17.864,5.93){\line(1,0){.9836}}
\put(19.831,5.93){\line(1,0){.9836}}
\put(21.799,5.93){\line(1,0){.9836}}
\put(23.766,5.93){\line(1,0){.9836}}
\put(25.733,5.93){\line(1,0){.9836}}
\put(27.7,5.93){\line(1,0){.9836}}
\put(29.667,5.93){\line(1,0){.9836}}
\put(31.635,5.93){\line(1,0){.9836}}
\put(33.602,5.93){\line(1,0){.9836}}
\put(35.569,5.93){\line(1,0){.9836}}
\put(37.536,5.93){\line(1,0){.9836}}
\put(39.503,5.93){\line(1,0){.9836}}
\put(41.471,5.93){\line(1,0){.9836}}
\put(43.438,5.93){\line(1,0){.9836}}
\put(45.405,5.93){\line(1,0){.9836}}
\put(47.372,5.93){\line(1,0){.9836}}
\put(49.34,5.93){\line(1,0){.9836}}
\put(51.307,5.93){\line(1,0){.9836}}
\put(53.274,5.93){\line(1,0){.9836}}
\put(55.241,5.93){\line(1,0){.9836}}
\put(57.208,5.93){\line(1,0){.9836}}
\put(59.176,5.93){\line(1,0){.9836}}
\put(61.143,5.93){\line(1,0){.9836}}
\put(63.11,5.93){\line(1,0){.9836}}
\put(65.077,5.93){\line(1,0){.9836}}
\put(67.044,5.93){\line(1,0){.9836}}
\put(69.012,5.93){\line(1,0){.9836}}
\put(70.979,5.93){\line(1,0){.9836}}
\put(72.946,5.93){\line(1,0){.9836}}
%\end
\put(19,20){\line(0,-1){14}}
\put(25,28){\line(0,-1){10}}
\put(31,23){\line(0,1){8}}
\put(36,29){\line(0,1){7}}
\put(45,26){\line(0,1){12}}
\put(53,11){\line(0,1){9}}
\put(59,16){\line(0,1){19}}
%\dashline{1}(25,18)(19,18)
\put(24.93,17.93){\line(-1,0){.8571}}
\put(23.215,17.93){\line(-1,0){.8571}}
\put(21.501,17.93){\line(-1,0){.8571}}
\put(19.787,17.93){\line(-1,0){.8571}}
%\end
%\dashline{1}(31,23)(25,23)
\put(30.93,22.93){\line(-1,0){.8571}}
\put(29.215,22.93){\line(-1,0){.8571}}
\put(27.501,22.93){\line(-1,0){.8571}}
\put(25.787,22.93){\line(-1,0){.8571}}
%\end
%\dashline{1}(36,29)(31,29)
\put(35.93,28.93){\line(-1,0){.8333}}
\put(34.263,28.93){\line(-1,0){.8333}}
\put(32.596,28.93){\line(-1,0){.8333}}
%\end
%\dashline{1}(45,26)(31,26)
\put(44.93,25.93){\line(-1,0){.9333}}
\put(43.063,25.93){\line(-1,0){.9333}}
\put(41.196,25.93){\line(-1,0){.9333}}
\put(39.33,25.93){\line(-1,0){.9333}}
\put(37.463,25.93){\line(-1,0){.9333}}
\put(35.596,25.93){\line(-1,0){.9333}}
\put(33.73,25.93){\line(-1,0){.9333}}
\put(31.863,25.93){\line(-1,0){.9333}}
%\end
%\dashline{1}(59,16)(53,16)
\put(58.93,15.93){\line(-1,0){.8571}}
\put(57.215,15.93){\line(-1,0){.8571}}
\put(55.501,15.93){\line(-1,0){.8571}}
\put(53.787,15.93){\line(-1,0){.8571}}
%\end
%\dashline{1}(53,11)(19.13,11)
\put(52.93,10.93){\line(-1,0){.9962}}
\put(50.937,10.93){\line(-1,0){.9962}}
\put(48.945,10.93){\line(-1,0){.9962}}
\put(46.953,10.93){\line(-1,0){.9962}}
\put(44.96,10.93){\line(-1,0){.9962}}
\put(42.968,10.93){\line(-1,0){.9962}}
\put(40.976,10.93){\line(-1,0){.9962}}
\put(38.983,10.93){\line(-1,0){.9962}}
\put(36.991,10.93){\line(-1,0){.9962}}
\put(34.999,10.93){\line(-1,0){.9962}}
\put(33.006,10.93){\line(-1,0){.9962}}
\put(31.014,10.93){\line(-1,0){.9962}}
\put(29.021,10.93){\line(-1,0){.9962}}
\put(27.029,10.93){\line(-1,0){.9962}}
\put(25.037,10.93){\line(-1,0){.9962}}
\put(23.044,10.93){\line(-1,0){.9962}}
\put(21.052,10.93){\line(-1,0){.9962}}
%\end
\put(67,23){\line(0,1){5}}
\put(73,25){\line(0,1){11}}
\put(66,12){\line(0,1){4}}
%\dashline{1}(66,12)(53,12)
\put(65.93,11.93){\line(-1,0){.9286}}
\put(64.073,11.93){\line(-1,0){.9286}}
\put(62.215,11.93){\line(-1,0){.9286}}
\put(60.358,11.93){\line(-1,0){.9286}}
\put(58.501,11.93){\line(-1,0){.9286}}
\put(56.644,11.93){\line(-1,0){.9286}}
\put(54.787,11.93){\line(-1,0){.9286}}
%\end
%\dashline{1}(67,23)(59,23)
\put(66.93,22.93){\line(-1,0){.8889}}
\put(65.152,22.93){\line(-1,0){.8889}}
\put(63.374,22.93){\line(-1,0){.8889}}
\put(61.596,22.93){\line(-1,0){.8889}}
\put(59.819,22.93){\line(-1,0){.8889}}
%\end
%\dashline{1}(73,25)(67,25)
\put(72.93,24.93){\line(-1,0){.8571}}
\put(71.215,24.93){\line(-1,0){.8571}}
\put(69.501,24.93){\line(-1,0){.8571}}
\put(67.787,24.93){\line(-1,0){.8571}}
%\end
\put(13,39){\makebox(0,0)[cc]{$t$}}

\end{picture}
\caption{ Graphical representation of a splitting tree. The vertical axis represents the biological time for the population. The horizontal axis has no biological meaning. The vertical segments represent the lifetimes of the individuals: the lower bounds their birth-times and the upper bounds death-times. The dotted lines denote the filiations between individuals. }
\label{fig: tree}
\end{figure}

For our study, it is easier to work with the genealogical tree of the population alive a time $t$. Indeed, since mutations are Poissonian, the different types in the population only depend  of the coalescence times of the lineages of the alive population. In order to derive the law of that genealogical tree, we need to characterize the law of the \emph{times of coalescence} between pairs of individuals in the population, which are the times since their lineages have split.

In \cite{L10}, Lambert introduces a contour process $Y$, which codes for the tree, and hence its genealogy. Suppose we are given a tree $\mathbb{T}$, seen as a subset of $\mathbb{R}\times\left(\cup_{k\geq 0}\mathbb{N}^{k}\right)$ with some compatibility conditions (see \cite{L10}), where $\mathbb{N}$ refers to the set of non-negative integers. On this object, Lambert constructs a Lebesgue measure $\lambda$ and a total order relation  $\preceq $ which can be roughly described as follows: let $x,y$ in $\mathbb{T}$, the point of birth of the lineage of $x$ during the lifetime of the root splits the tree in two connected components, then $y\preceq x$ if $y$ belong to the same component as $x$ but is not an ancestor of $x$ (see Figure \ref{fig:order}).

\textcolor{red}{If we assume that $\lambda(\mathbb{T})$ is finite}, then the application,
\[
\begin{array}{cccc}
\varphi:&\mathbb{T} & \to & [0,\lambda\left(\mathbb{T}\right)], \\
& x & \mapsto & \lambda\left(\{y  \mid  \ y\preceq x\} \right), \\
\end{array}
\]
is a one-to-one correspondence. In a graphical sense (see Figure \ref{fig:order}), $\varphi(x)$ measures the length of the part of the tree which is above the lineage of $x$.
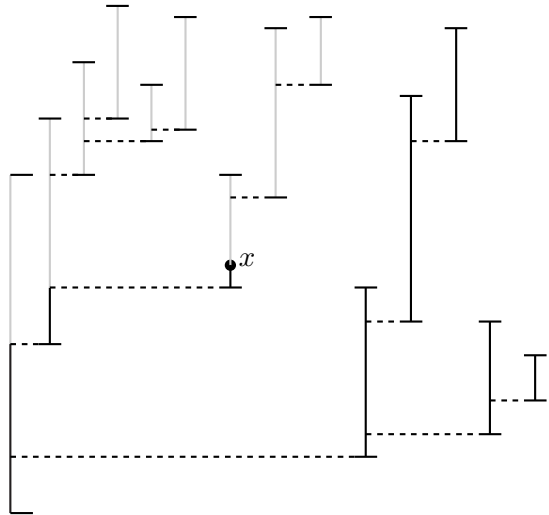
\begin{figure}[ht]
\unitlength 2mm % = 4.55pt
\linethickness{0.4pt}

%TeXCAD (http://texcad.sf.net/) Picture. File: [smalltree.tex]. Options on following lines.
%\grade{\on}
%\emlines{\off}
%\epic{\off}
%\beziermacro{\on}
%\reduce{\on}
%\snapping{\off}
%\quality{8.000}
%\graddiff{0.005}
%\snapasp{1}
%\zoom{8.0000}
\unitlength 1.5mm % = 2.845pt
\linethickness{0.8pt}
\begin{picture}(15,45)(-30,0)

\put(0,0){\line(0,1){15}}%branche racine
\color{gray5}
\put(0,15){\line(0,1){15}}
\color{black}
\put(0,0){\line(1,0){2}}%pied racine
\put(0,30){\line(1,0){2}}%tete racine
%R2
\multiput(0,15)(1,0){3}{\line(1,0){0.5}}%r to r2  2.5  

\put(3.5,15){\line(0,1){5}}% r2 15>35
\color{gray5}
\put(3.5,20){\line(0,1){15}}
\color{black}
\put(2.5,15){\line(1,0){2}}%pied r2
\put(2.5,35){\line(1,0){2}}%tete r2
%R22
\multiput(3.5,30)(1,0){3}{\line(1,0){0.5}}%r2 to r22  l=3.5 -> 6.5
\color{gray5}
\put(6.5,30){\line(0,1){10}}% r22 30>40
\color{black}
\put(5.5,30){\line(1,0){2}}%pied r22
\put(5.5,40){\line(1,0){2}}%tete r22
%R222
\multiput(6.5,35)(1,0){3}{\line(1,0){0.5}}%r22 to r222  l=6.5 -> 8.5
\color{gray5}
\put(9.5,35){\line(0,1){10}}%r222 35>45
\color{black}
\put(8.5,35){\line(1,0){2}}%pied r22
\put(8.5,45){\line(1,0){2}}%tete r22
%R221
\multiput(6.5,33)(1,0){6}{\line(1,0){0.5}}%r22 to r221  l=6.5 -> 12.5
\color{gray5}
\put(12.5,33){\line(0,1){5}}%r221 33>38
\color{black}
\put(11.5,33){\line(1,0){2}}%pied r22
\put(11.5,38){\line(1,0){2}}%tete r22
%R2211
\multiput(12.5,34)(1,0){3}{\line(1,0){0.5}}%r221 to r2212  l=12.5 -> 14.5
\color{gray5}
\put(15.5,34){\line(0,1){10}}%r221 33>38
\color{black}
\put(14.5,34){\line(1,0){2}}%pied r22
\put(14.5,44){\line(1,0){2}}%tete r22
%R21
\put(19.5,22){\circle*{1}}
\put(20.5,22){\makebox(1,1){$x$}}
\multiput(3.5,20)(1,0){15}{\line(1,0){0.5}}%r2 to r21  l=3.5 -> 21.5
\color{gray5}
\put(19.5,22){\line(0,1){8}}%r21 20>30
\color{black}
\put(19.5,20){\line(0,1){2}}%r21 20>30

\put(18.5,20){\line(1,0){2}}%pied r21
\put(18.5,30){\line(1,0){2}}%tete r21
%R211
\multiput(19.5,28)(1,0){3}{\line(1,0){0.5}}%r21 to r211  l=3.5 -> 18.5
\color{gray5}
\put(23.5,28){\line(0,1){15}}%r211 28>43
\color{black}
\put(22.5,28){\line(1,0){2}}%pied r21
\put(22.5,43){\line(1,0){2}}%tete r21
%R2111
\multiput(23.5,38)(1,0){3}{\line(1,0){0.5}}%r211 to r2111  l=3.5 -> 18.5
\color{gray5}
\put(27.5,38){\line(0,1){6}}%r21 33>38
\color{black}
\put(26.5,38){\line(1,0){2}}%pied r21
\put(26.5,44){\line(1,0){2}}%tete r21
%R1
\multiput(0,5)(1,0){31}{\line(1,0){0.5}}%r to r1  l=3.5 -> 18.5
\put(31.5,5){\line(0,1){15}}%r1 33>38
\put(30.5,5){\line(1,0){2}}%pied r1
\put(30.5,20){\line(1,0){2}}%tete r1
%R12
\multiput(31.5,17)(1,0){3}{\line(1,0){0.5}}%r1 to r12  l=3.5 -> 18.5
\put(35.5,17){\line(0,1){20}}%r21 17>37
\put(34.5,17){\line(1,0){2}}%pied r21
\put(34.5,37){\line(1,0){2}}%tete r21
%R121
\multiput(35.5,33)(1,0){3}{\line(1,0){0.5}}%r211 to r2111  l=3.5 -> 18.5
\put(39.5,33){\line(0,1){10}}%r21 5>20
\put(38.5,33){\line(1,0){2}}%pied r21
\put(38.5,43){\line(1,0){2}}%tete r21
%R11
\multiput(31.5,7)(1,0){10}{\line(1,0){0.5}}%r1 to r11  l=3.5 -> 18.5
\put(42.5,7){\line(0,1){10}}%r21 5>20
\put(41.5,7){\line(1,0){2}}%pied r21
\put(41.5,17){\line(1,0){2}}%tete r21
%R111
\multiput(42.5,10)(1,0){3}{\line(1,0){0.5}}%r1 to r11  l=3.5 -> 18.5
\put(46.5,10){\line(0,1){4}}%r21 5>20
\put(45.5,10){\line(1,0){2}}%pied r21
\put(45.5,14){\line(1,0){2}}%tete r21
\end{picture}
\caption{\textcolor{red}{Graphical representation of the set $\left\{y\in\mathbb{T}\mid y\preceq x \right\}$ (in grey).} }
\label{fig:order}
\end{figure}
The contour process is then defined, for all $s$, by
\[
Y_{s}:=\Pi_{\mathbb{R}}\left(\varphi^{-1}\left(s\right)\right),
\]
where $\Pi_{\mathbb{R}}$ is the projection from $\mathbb{R}\times\left(\cup_{k\geq 0}\mathbb{N}^{k}\right)$ to $\mathbb{R}$. \textcolor{red}{In the case where $\lambda(\mathbb{T})$ is infinite, one has to consider truncations of the tree above fixed levels in order to define contours (see \cite{L10} for more details)}.

In a more graphical way, the contour process can be seen as the graph of an exploration process of the tree:
it begins at the top of the root and decreases with slope $-1$ while running back along the life of the root until it meets a birth. The contour process then jumps at the top of the life interval of the child born at this time and continues its exploration as before. If the exploration process does not encounter a birth when exploring the life interval of an individual, it goes back to its parent and continues the exploration from the birth-date of the just left individual (see Figure \ref{fig: cpp}).
It is then readily seen that the intersections of the contour process with the line of ordinate $t$ are in one-to-one correspondence with the individuals in the tree alive at time $t$.

In \cite{L10}, Lambert shows that the contour process of the splitting tree which has been pruned from every part above $t$ (called \emph{truncated tree above $t$}), has the law of a spectrally positive Lévy process reflected below $t$ and killed at $0$ with Laplace exponent
\[
\psi(x)=x-\int_{(0,\infty]}\left(1-e^{-rx}\right)\Lambda(dr), \ \ x\in\mathbb{R}_{+}.
\]
The largest root of $\psi$, denoted $\alpha$, is called the Malthusian parameter and, as soon as $\alpha >0$, gives the rate of growth of the population on the survival event.

The time of coalescence of two individuals alive at time $t$ corresponds to the amount of time one needs to go back in the past along their lineages to get their first common ancestor.
The time of coalescence between an individual alive at time $t$ and the next one visited by the contour is exactly the depth of the excursion of the contour process below $t$ between this two successive individuals (see Figure \ref{fig: cpp}). We are interested in the sequence of coalescence times shown in Figure \ref{fig: cpp}, which contain the minimal information needed to reconstruct the genealogy at time $t$.

%The genealogy of the population is then the genealogical tree formed by the sequence of branches of lengths given by the depths of the excursions of the contour below $t$, where lineages coalesce with the first deeper branch on their left (see Figure \ref{fig: contour}).

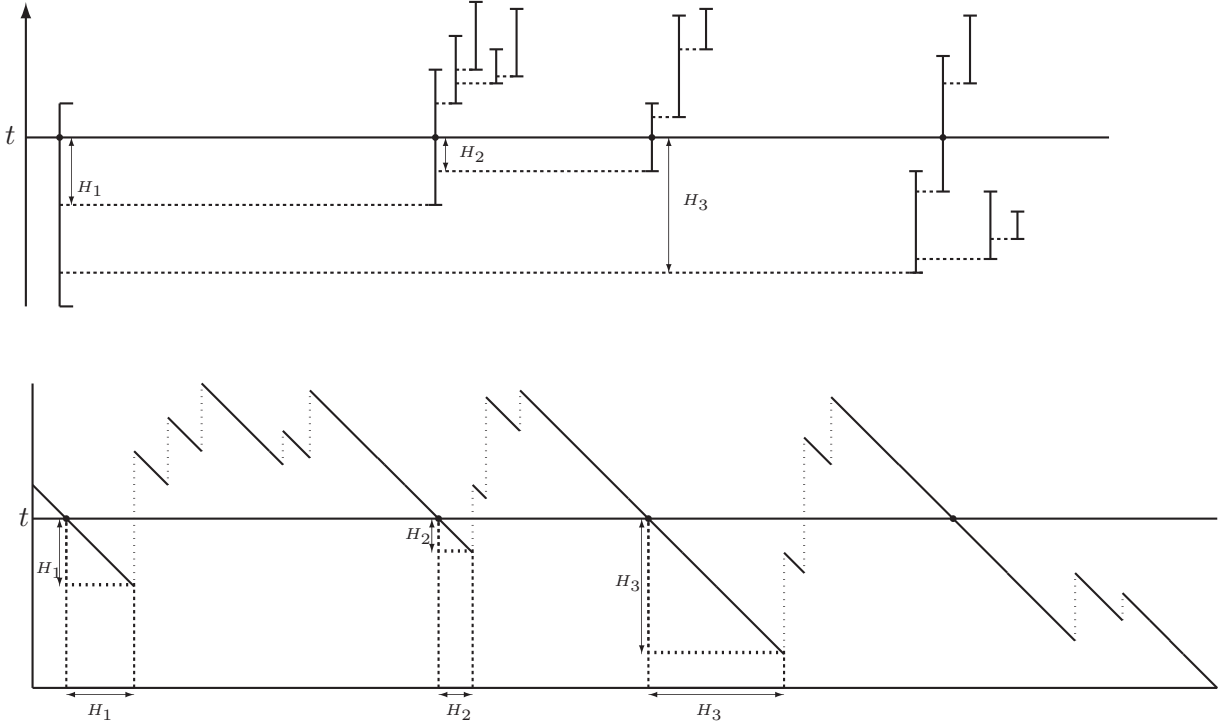
\begin{figure}[ht]
\unitlength 0.9mm % = 5.69pt
\linethickness{0.8pt}

\begin{picture}(160,60)(0,0)

\put(-2,24){\makebox{$t$}}
\put(1,0){
%\color{blue}
\put(0,25){\line(1,0){160}}
%\color{black}
\put(0,0){\vector(0,1){45}}
\put(5,0){

\put(0,25){\circle*{1}}
\put(0,0){\line(0,1){30}}%branche racine
\put(0,0){\line(1,0){2}}%pied racine
\put(0,30){\line(1,0){2}}%tete racine
\linethickness{0.2pt}
\put(1.75,25){\vector(0,-1){10}}
\put(1.75,15){\vector(0,1){10}}
\put(2.5,17){\makebox{$\scriptscriptstyle{H_{1}}$}}
%R2
\linethickness{0.8pt}
\put(52,0){
\multiput(-52,15)(1,0){55}{\line(1,0){0.5}}%r to r2  2.5  
\put(3.5,25){\circle*{1}}
\put(3.5,15){\line(0,1){20}}% r2 15>35
\put(2.5,15){\line(1,0){2}}%pied r2
\put(2.5,35){\line(1,0){2}}%tete r2
%R22
\multiput(3.5,30)(1,0){3}{\line(1,0){0.5}}%r2 to r22  l=3.5 -> 6.5
\put(6.5,30){\line(0,1){10}}% r22 30>40
\put(5.5,30){\line(1,0){2}}%pied r22
\put(5.5,40){\line(1,0){2}}%tete r22
%R222
\multiput(6.5,35)(1,0){3}{\line(1,0){0.5}}%r22 to r222  l=6.5 -> 8.5
\put(9.5,35){\line(0,1){10}}%r222 35>45
\put(8.5,35){\line(1,0){2}}%pied r22
\put(8.5,45){\line(1,0){2}}%tete r22
%R221
\multiput(6.5,33)(1,0){6}{\line(1,0){0.5}}%r22 to r221  l=6.5 -> 12.5
\put(12.5,33){\line(0,1){5}}%r221 33>38
\put(11.5,33){\line(1,0){2}}%pied r22
\put(11.5,38){\line(1,0){2}}%tete r22
%R2211

\multiput(12.5,34)(1,0){3}{\line(1,0){0.5}}%r221 to r2212  l=12.5 -> 14.5
\put(15.5,34){\line(0,1){10}}%r221 33>38
\put(14.5,34){\line(1,0){2}}%pied r22
\put(14.5,44){\line(1,0){2}}%tete r22
%R21
\put(16,0){
\multiput(-12,20)(1,0){31}{\line(1,0){0.5}}%r2 to r21  l=3.5 -> 21.5
\put(19.5,25){\circle*{1}}
\put(19.5,20){\line(0,1){10}}%r21 20>30
\put(18.5,20){\line(1,0){2}}%pied r21
\put(18.5,30){\line(1,0){2}}%tete r21
%R211
\multiput(19.5,28)(1,0){3}{\line(1,0){0.5}}%r21 to r211  l=3.5 -> 18.5
\put(23.5,28){\line(0,1){15}}%r211 28>43
\put(22.5,28){\line(1,0){2}}%pied r21
\put(22.5,43){\line(1,0){2}}%tete r21
%R2111

\multiput(23.5,38)(1,0){3}{\line(1,0){0.5}}%r211 to r2111  l=3.5 -> 18.5
\put(27.5,38){\line(0,1){6}}%r21 33>38
\put(26.5,38){\line(1,0){2}}%pied r21
\put(26.5,44){\line(1,0){2}}%tete r21
\put(27,0){
%R1
\multiput(-95,5)(1,0){128}{\line(1,0){0.5}}%r to r1  l=3.5 -> 18.5
\put(31.5,5){\line(0,1){15}}%r1 33>38
\put(30.5,5){\line(1,0){2}}%pied r1
\put(30.5,20){\line(1,0){2}}%tete r1
%R12

\multiput(31.5,17)(1,0){3}{\line(1,0){0.5}}%r1 to r12  l=3.5 -> 18.5
\put(35.5,25){\circle*{1}}
\put(35.5,17){\line(0,1){20}}%r21 17>37
\put(34.5,17){\line(1,0){2}}%pied r21
\put(34.5,37){\line(1,0){2}}%tete r21
%R121
\multiput(35.5,33)(1,0){3}{\line(1,0){0.5}}%r211 to r2111  l=3.5 -> 18.5
\put(39.5,33){\line(0,1){10}}%r21 5>20
\put(38.5,33){\line(1,0){2}}%pied r21
\put(38.5,43){\line(1,0){2}}%tete r21
%R11
\multiput(31.5,7)(1,0){10}{\line(1,0){0.5}}%r1 to r11  l=3.5 -> 18.5
\put(42.5,7){\line(0,1){10}}%r21 5>20
\put(41.5,7){\line(1,0){2}}%pied r21
\put(41.5,17){\line(1,0){2}}%tete r21
%R111
\multiput(42.5,10)(1,0){3}{\line(1,0){0.5}}%r1 to r11  l=3.5 -> 18.5
\put(46.5,10){\line(0,1){4}}%r21 5>20
\put(45.5,10){\line(1,0){2}}%pied r21
\put(45.5,14){\line(1,0){2}}%tete r21
%\color{green}
%\put(0,15){\line(0,1){10}}
%\put(3.5,15){\line(0,1){10}}
%\put(19.5,20){\line(0,1){5}}
%\put(35.5,17){\line(0,1){8}}
%\color{black}
\linethickness{0.2pt}
\put(-38,25){\vector(0,-1){5}}
\put(-38,20){\vector(0,1){5}}
\put(-36,22){\makebox{$\scriptscriptstyle{H_{2}}$}}
\put(-5,25){\vector(0,-1){20}}
\put(-5,5){\vector(0,1){20}}
\put(-3,15){\makebox{$\scriptscriptstyle{H_{3}}$}}
}}}}}
\end{picture}

\begin{picture}(160,60)(0,0)

%%%%%%%%%%%%%AXES%%%%%%%%%%%%%%%%%%%
\linethickness{0.2pt}
\put(7,3){\vector(1,0){10}}
\put(17,3){\vector(-1,0){10}}
\put(62,3){\vector(1,0){5}}
\put(67,3){\vector(-1,0){5}}
\put(93,3){\vector(1,0){20}}
\put(113,3){\vector(-1,0){20}}
\put(10,0){\makebox{$\scriptscriptstyle{H_{1}}$}}
\put(63,0){\makebox{$\scriptscriptstyle{H_{2}}$}}
\put(100,0){\makebox{$\scriptscriptstyle{H_{3}}$}}
\linethickness{0.8pt}
\put(0,4){
\put(0,24){\makebox{$t$}}
\put(2,0){
\put(0,0){\line(1,0){175}}
\put(0,25){\line(1,0){175}}
\put(0,0){\line(0,1){45}}
%%%%%%%%%%%%%%%%%%%%
%R AVANT NAISSANCE
\put(5,25){\circle*{1}}
\put(60,25){\circle*{1}}
\put(91,25){\circle*{1}}
\put(136,25){\circle*{1}}
\multiput(5,0)(0,1){25}{\line(0,1){0.5}}
\multiput(60,0)(0,1){25}{\line(0,1){0.5}}
\multiput(91,0)(0,1){25}{\line(0,1){0.5}}
%\multiput(136,0)(0,1){25}{\line(0,1){0.5}}
%
%
\multiput(15,0)(0,1){15}{\line(0,1){0.5}}
\multiput(65,0)(0,1){20}{\line(0,1){0.5}}
\multiput(111,0)(0,1){5}{\line(0,1){0.5}}
%\multiput(154,0)(0,1){7}{\line(0,1){0.5}}

\multiput(15,15)(-1,0){10}{\line(0,1){0.5}}
\multiput(65,20)(-1,0){5}{\line(0,1){0.5}}
\multiput(111,5)(-1,0){20}{\line(0,1){0.5}}
%\multiput(154,7)(-1,0){18}{\line(0,1){0.5}}
%\color{green}
\multiput(5,15)(0,1){10}{\line(0,1){0.5}}
\multiput(60,20)(0,1){5}{\line(0,1){0.5}}
\multiput(91,5)(0,1){20}{\line(0,1){0.5}}
%\multiput(136,7)(0,1){18}{\line(0,1){0.5}}
%
\put(5,0){\line(1,0){10}}
\put(60,0){\line(1,0){5}}
\put(91,0){\line(1,0){20}}
%\put(136,0){\line(1,0){18}}
%\color{black}
%
%
\put(0.5,17){\makebox{$\scriptscriptstyle{H_{1}}$}}
\put(55,22){\makebox{$\scriptscriptstyle{H_{2}}$}}
\put(86,15){\makebox{$\scriptscriptstyle{H_{3}}$}}
\put(0,30){\line(1,-1){15}}%15,15
\multiput(15,15)(0,1){20}{\line(0,1){0.1}}%15,35
\put(15,35){\line(1,-1){5}}%20,30
\multiput(20,30)(0,1){10}{\line(0,1){0.1}}%20,40
\put(20,40){\line(1,-1){5}}%25,35
\multiput(25,35)(0,1){10}{\line(0,1){0.1}}%25,45
\put(25,45){\line(1,-1){10}}%35,35
\put(35,35){\line(1,-1){2}}%37,33
%\color{black}
\multiput(37,33)(0,1){5}{\line(0,1){0.1}}%37,38
\put(37,38){\line(1,-1){4}}%41,34
\multiput(41,34)(0,1){10}{\line(0,1){0.1}}%41,44
\put(41,44){\line(1,-1){10}}%51,34
\put(51,34){\line(1,-1){1}}%52,33
\put(52,33){\line(1,-1){3}}%55,30
\put(55,30){\line(1,-1){10}}%65,20
\multiput(65,20)(0,1){10}{\line(0,1){0.1}}%65,30
\put(65,30){\line(1,-1){2}}%67,28
\multiput(67,28)(0,1){15}{\line(0,1){0.1}}%67,43
\put(67,43){\line(1,-1){5}}%72,38
\multiput(72,38)(0,1){6}{\line(0,1){0.1}}%72,44
\put(72,44){\line(1,-1){6}}%78,38
\put(78,38){\line(1,-1){10}}%88,28
\put(88,28){\line(1,-1){8}}%96,20
\put(96,20){\line(1,-1){5}}%101,15
\put(101,15){\line(1,-1){10}}%111,5
\multiput(111,5)(0,1){15}{\line(0,1){0.1}}%111,20
\put(111,20){\line(1,-1){3}}%114,17
\multiput(114,17)(0,1){20}{\line(0,1){0.1}}%114,37
\put(114,37){\line(1,-1){4}}%118,33
\multiput(118,33)(0,1){10}{\line(0,1){0.1}}%118,43
\put(118,43){\line(1,-1){10}}%128,33
\put(128,33){\line(1,-1){16}}%144,17
\put(144,17){\line(1,-1){10}}%154,7
\multiput(154,7)(0,1){10}{\line(0,1){0.1}}%154,17
\put(154,17){\line(1,-1){7}}%161,10
\multiput(161,10)(0,1){4}{\line(0,1){0.1}}%161,14
\put(161,14){\line(1,-1){4}}%165,10
\put(165,10){\line(1,-1){3}}%168,7
\put(168,7){\line(1,-1){2}}%170,5
\put(170,5){\line(1,-1){5}}%175,0
\linethickness{0.2pt}

\put(4,15){\vector(0,1){10}}
\put(4,25){\vector(0,-1){10}}
\put(59,20){\vector(0,1){5}}
\put(59,25){\vector(0,-1){5}}
\put(90,5){\vector(0,1){20}}
\put(90,25){\vector(0,-1){20}}
}}
\end{picture}
\caption{Construction of the contour process and link between the excursions of the contour process and the times of coalescence in the tree.}
\label{fig: cpp}
\end{figure}
More precisely, it follows from well known fluctuation properties of \textcolor{red}{spectrally positive} L\'evy processes (see \cite{Kyp}, Theorem 8.1 \textcolor{red}{for spectrally negative L\'evy processes}) that the law of the depth $H$ of an excursions below $t$ is given by
%It is well known that such a spectrally positive Levy process has nice fluctuations properties (see \cite{Kyp}, theorem 8.1). In particular, the law of the depth $H$ of an excursions below $t$ is given by,
\[
\mathbb{P}\left(H>s\right)=\frac{1}{W(s)}, \quad s\in\mathbb{R}_{+},
\]
where $W$ is the scale function of the L\'evy process characterized by its Laplace transform
\begin{equation}
\label{eq:laplace}
\int_{(0,\infty)}e^{-rt}W(r)dr=\frac{1}{\psi(t)}, \quad t>\alpha.
\end{equation}
Since the contour process is strong Markov, the sequence of excursion depths is \emph{i.i.d.}

To summarize, given the population is still alive at time $t$, one can forget the splitting tree and code the genealogy of the living individuals alive at time $t$ by a new object called the {\it coalescent point process} (CPP) \textcolor{red}{at time $t$} shown in Figure \ref{fig: contour}. Its law is the law of a sequence $\left(H_{i}\right)_{0\leq i\leq N_{t}-1}$, where the family $\left(H_{i}\right)_{i\geq 1}$ is \emph{i.i.d.} with the same law as $H$, stopped at its first value $H_{N_{t}}$ greater than $t$, and $H_{0}$ is deterministic equal to $t$ (see Figure \ref{fig: contour}). The heights $H_{1},\dots,H_{N_{t}-1}$ are called \emph{branch lengths} of the CPP.
\begin{rem}
\label{rem:agoodrem}
Let $N$ be an integer valued random variable. In the sequel we say that a random vector with random size $\left(X_{i}\right)_{1\leq i\leq N}$ form an i.i.d.\ family of random variables independent of $N$, if and only if
\[
\left(X_{1},\dots,X_{N}\right)\overset{d}{=}\left(\tilde{X_{1}},\dots,\tilde{X}_{N}\right),
\] 
where $\left(\tilde{X}_{i}\right)_{i\geq 1}$ is a sequence of i.i.d.\ random variable distributed as $X_{1}$ independent of $N$.
\end{rem}
From the CPP \textcolor{red}{at time $t$}, the genealogical tree of alive individuals at time $t$ is obtained considering that the $i$th branch coalesces with the first branch on its left \textcolor{red}{such that $H_{j}>H_{i}$ (for $j<i$) (see Figure \ref{fig: contour})}.
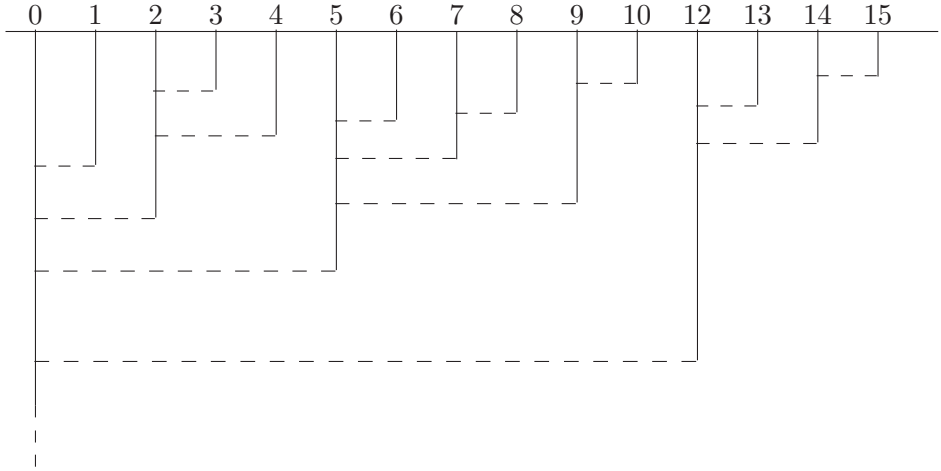
\begin{figure}[ht]

\unitlength 2mm % = 4.55pt
\linethickness{0.4pt}

\begin{picture}(66,33)(-5,10)
\put(4,39.875){\line(1,0){62}}
\put(10,40){\line(0,-1){9}}
\put(14,40){\line(0,-1){12.5}}
\put(18,40){\line(0,-1){4}}
\put(22,40){\line(0,-1){7}}
\put(26,40){\line(0,-1){16}}
\put(30,40){\line(0,-1){6}}
\put(34,39.875){\line(0,-1){8.5}}
\put(38,40){\line(0,-1){5.5}}
\put(42,40){\line(0,-1){11.5}}
\put(46,40){\line(0,-1){3.625}}
\put(50,40){\line(0,-1){22}}
\put(54,40){\line(0,-1){5}}
\put(58,40){\line(0,-1){7.5}}
\put(62,39.875){\line(0,-1){3}}
\put(6,40){\line(0,-1){25}}
%\dashline{1}(6,15)(6,11)
% \put(5.93,14.93){\line(0,-1){.8}}
% \put(5.93,13.33){\line(0,-1){.8}}
% \put(5.93,11.73){\line(0,-1){.8}}
\put(6,15){\line(0,-1){.8}}
\put(6,13.33){\line(0,-1){.8}}
\put(6,11.73){\line(0,-1){.8}}
%\end
%\dashline{1}(10,31)(6,31)
\put(9.93,30.93){\line(-1,0){.8}}
\put(8.33,30.93){\line(-1,0){.8}}
\put(6.73,30.93){\line(-1,0){.8}}
%\end
%\dashline{1}(14,28.5)(6,28.5)
\put(13.93,27.43){\line(-1,0){.8889}}
\put(12.152,27.43){\line(-1,0){.8889}}
\put(10.374,27.43){\line(-1,0){.8889}}
\put(8.596,27.43){\line(-1,0){.8889}}
\put(6.819,27.43){\line(-1,0){.8889}}
%\end
%\dashline{1}(17.875,36)(13.875,36)
\put(17.805,35.93){\line(-1,0){.8}}
\put(16.205,35.93){\line(-1,0){.8}}
\put(14.605,35.93){\line(-1,0){.8}}
%\end
%\dashline{1}(22,33)(14,33)
\put(21.93,32.93){\line(-1,0){.8889}}
\put(20.152,32.93){\line(-1,0){.8889}}
\put(18.374,32.93){\line(-1,0){.8889}}
\put(16.596,32.93){\line(-1,0){.8889}}
\put(14.819,32.93){\line(-1,0){.8889}}
%\end
%\dashline{1}(26,24)(6,24)
\put(25.93,23.93){\line(-1,0){.9524}}
\put(24.025,23.93){\line(-1,0){.9524}}
\put(22.12,23.93){\line(-1,0){.9524}}
\put(20.215,23.93){\line(-1,0){.9524}}
\put(18.311,23.93){\line(-1,0){.9524}}
\put(16.406,23.93){\line(-1,0){.9524}}
\put(14.501,23.93){\line(-1,0){.9524}}
\put(12.596,23.93){\line(-1,0){.9524}}
\put(10.692,23.93){\line(-1,0){.9524}}
\put(8.787,23.93){\line(-1,0){.9524}}
\put(6.882,23.93){\line(-1,0){.9524}}
%\end
%\dashline{1}(30,34)(26,34)
\put(29.93,33.93){\line(-1,0){.8}}
\put(28.33,33.93){\line(-1,0){.8}}
\put(26.73,33.93){\line(-1,0){.8}}
%\end
%\dashline{1}(34,31.5)(26,31.5)
\put(33.93,31.43){\line(-1,0){.8889}}
\put(32.152,31.43){\line(-1,0){.8889}}
\put(30.374,31.43){\line(-1,0){.8889}}
\put(28.596,31.43){\line(-1,0){.8889}}
\put(26.819,31.43){\line(-1,0){.8889}}
%\end
%\dashline{1}(38,34.5)(34,34.5)
\put(37.93,34.43){\line(-1,0){.8}}
\put(36.33,34.43){\line(-1,0){.8}}
\put(34.73,34.43){\line(-1,0){.8}}
%\end
%\dashline{1}(42,28.5)(26,28.5)
\put(41.93,28.43){\line(-1,0){.9412}}
\put(40.047,28.43){\line(-1,0){.9412}}
\put(38.165,28.43){\line(-1,0){.9412}}
\put(36.283,28.43){\line(-1,0){.9412}}
\put(34.4,28.43){\line(-1,0){.9412}}
\put(32.518,28.43){\line(-1,0){.9412}}
\put(30.636,28.43){\line(-1,0){.9412}}
\put(28.753,28.43){\line(-1,0){.9412}}
\put(26.871,28.43){\line(-1,0){.9412}}
%\end
%\dashline{1}(46,36.5)(42,36.5)
\put(45.93,36.43){\line(-1,0){.8}}
\put(44.33,36.43){\line(-1,0){.8}}
\put(42.73,36.43){\line(-1,0){.8}}
%\end
%\dashline{1}(50,18)(6,18)
\put(49.93,17.93){\line(-1,0){.9778}}
\put(47.974,17.93){\line(-1,0){.9778}}
\put(46.019,17.93){\line(-1,0){.9778}}
\put(44.063,17.93){\line(-1,0){.9778}}
\put(42.107,17.93){\line(-1,0){.9778}}
\put(40.152,17.93){\line(-1,0){.9778}}
\put(38.196,17.93){\line(-1,0){.9778}}
\put(36.241,17.93){\line(-1,0){.9778}}
\put(34.285,17.93){\line(-1,0){.9778}}
\put(32.33,17.93){\line(-1,0){.9778}}
\put(30.374,17.93){\line(-1,0){.9778}}
\put(28.419,17.93){\line(-1,0){.9778}}
\put(26.463,17.93){\line(-1,0){.9778}}
\put(24.507,17.93){\line(-1,0){.9778}}
\put(22.552,17.93){\line(-1,0){.9778}}
\put(20.596,17.93){\line(-1,0){.9778}}
\put(18.641,17.93){\line(-1,0){.9778}}
\put(16.685,17.93){\line(-1,0){.9778}}
\put(14.73,17.93){\line(-1,0){.9778}}
\put(12.774,17.93){\line(-1,0){.9778}}
\put(10.819,17.93){\line(-1,0){.9778}}
\put(8.863,17.93){\line(-1,0){.9778}}
\put(6.907,17.93){\line(-1,0){.9778}}
%\end
%\dashline{1}(54,35)(50,35)
\put(53.93,34.93){\line(-1,0){.8}}
\put(52.33,34.93){\line(-1,0){.8}}
\put(50.73,34.93){\line(-1,0){.8}}
%\end
%\dashline{1}(58,32.5)(50,32.5)
\put(57.93,32.43){\line(-1,0){.8889}}
\put(56.152,32.43){\line(-1,0){.8889}}
\put(54.374,32.43){\line(-1,0){.8889}}
\put(52.596,32.43){\line(-1,0){.8889}}
\put(50.819,32.43){\line(-1,0){.8889}}
%\end
%\dashline{1}(62,35)(58,35)

\put(61.93,36.93){\line(-1,0){.8}}
\put(60.33,36.93){\line(-1,0){.8}}
\put(58.73,36.93){\line(-1,0){.8}}
%\end
\put(6,41){\makebox(0,0)[cc]{$0$}}
\put(10,41){\makebox(0,0)[cc]{$1$}}
\put(14,41){\makebox(0,0)[cc]{$2$}}
\put(18,41){\makebox(0,0)[cc]{$3$}}
\put(22,41){\makebox(0,0)[cc]{$4$}}

\put(26,41){\makebox(0,0)[cc]{$5$}}
\put(30,41){\makebox(0,0)[cc]{$6$}}
\put(34,41){\makebox(0,0)[cc]{$7$}}
\put(38,41){\makebox(0,0)[cc]{$8$}}
\put(42,41){\makebox(0,0)[cc]{$9$}}
\put(46,41){\makebox(0,0)[cc]{$10$}}

\put(50,41){\makebox(0,0)[cc]{$12$}}
\put(54,41){\makebox(0,0)[cc]{$13$}}
\put(58,41){\makebox(0,0)[cc]{$14$}}
\put(62,41){\makebox(0,0)[cc]{$15$}}
\end{picture}

\caption{ A coalescent point process for $16$ individuals, hence $15$ branches. The filiation relation between lineages is indicated by horizontal dashed lines.}
\label{fig: contour}
\end{figure}

 The number $N_{t}$ of alive individuals at time $t$ in the splitting tree is then given by
\[
N_{t}=\inf\{i\geq 1 \mid H_{i}>t\}.
\]
From the comments above, $N_{t}$ is a geometric random variable given $N_{t}>0$. More precisely,
\[
\mathbb{P}\left(N_{t}=k\mid N_{t}>0\right)=\frac{1}{W(t)}\left(1-\frac{1}{W(t)}\right)^{k-1}, \quad \forall k\geq 1.
\]
Finally, we can define the occurrence of mutations directly on the CPP as the atoms of a random measure. Let $\mathcal{P}$ be a Poisson random measure on $(0,t)\times\mathbb{N}$ with intensity measure $\theta \lambda\otimes C$ where $\lambda$ is the Lebesgue measure on $(0,t)$ and $C$ is the counting measure on $\mathbb{N}$. The mutation random measure on the CPP \textcolor{red}{at time $t$} is then defined by
\begin{equation}
\label{eq:randommeasure}
\mathcal{N}\left(da,di\right)=\mathds{1}_{H_{i}>t-a}\mathds{1}_{i<N_{t}}\mathcal{P}\left(di,da\right),
\end{equation}
where an atom at $(a,i)$ means that the $i$th branch of the CPP experiences a mutation at time $t-a$.
Note that, when one looks at the allele distribution at time $t$, this construction is equivalent to the construction of Poissonian mutations on the original splitting tree \cite{CL2}.

We assume that each mutation gives a totally new type to its holder (infinitly-many alleles model) and that the types are transmitted to offspring. This rule yields a partition of the population by type at a given time $t$. The distribution of the frequency of types in the population is called the frequency spectrum and is defined as the sequence $\left(A(k,t)\right)_{k\geq1}$ where $A(k,t)$ is the number of types carried by exactly $k$ individuals in the alive population at time $t$ (or, for short, the number of families of size $k$ at this time) excluding the family holding the original type of the root.

% Hereafter, we call $k$-mutation a mutation held by family of size $k$ (such family should be called a $k$-family or $k$-type) at time $t$.

In the study of the frequency spectrum, an important role is played by the family carrying the type of the root. \textcolor{red}{The type of the ancestor individual at time $0$ is said \emph{clonal}}. Moreover, at any time $t$, the set of individuals carrying this type is called the clonal family at time $t$. We denote by $Z_{0}(t)$ the size of the clonal family at time $t$.

To study this family it is easier to consider the clonal splitting tree constructed from the original splitting tree by cutting every branches beyond mutations. This clonal splitting tree is a standard splitting tree without mutations where individuals are killed as soon as they die or experience a mutation.   The new lifespan law $\mathbb{P}_{V_{\theta}}$ is then the minimum between an exponential random variable of parameter $\theta$ and an independent copy of $V$.
As a splitting tree, one can study its contour process whose Laplace exponent is given, using simple manipulations on Laplace transforms, by
\[
\psi_{\theta}(x)=x-\int_{(0,\infty]}\left(1-e^{-rx}\right)\Lambda_{\theta}(dr)=\frac{x\psi(x+\theta)}{x+\theta}.
\]
In the case where $\alpha-\theta>0$ (resp. $\alpha-\theta<0$, $\alpha-\theta=0$) the clonal population is supercritical (resp. sub-critical, critical), and we talk about clonal supercritical (resp. sub-critical, critical) case.

We denote by $W_{\theta}$ the scale function of the L\'evy process induced by this new tree, related to $\psi_{\theta}$ as in \eqref{eq:laplace}.
This leads to
\[
\mathbb{P}\left(Z_{0}(t)=k\mid Z_{0}(t)>0\right)=\frac{1}{W_{\theta}(t)}\left(1-\frac{1}{W_{\theta}(t)} \right)^{k-1}.
\]
Moreover, $\mathbb{E}\left[N_{t}\right]$ satisfies the renewal equation
\[
f(t)=\mathbb{P}\left(V>t\right)+b\int_{0}^{t}f(t-s)\mathbb{P}\left(V>s\right)ds,
\]
which, applied to the clonal splitting tree, allows obtaining after some easy calculations,
\[
\frac{\mathbb{P}\left(Z_{0}(t)>0\right)}{\mathbb{P}\left(N_{t}>0\right)}=\frac{e^{-\theta t}W(t)}{W_{\theta}(t)},
\]
from which one can deduce
\begin{equation}
\mathbb{P}\left(Z_{0}(t)=k\mid N_{t}>0\right)=\frac{e^{-\theta t}W(t)}{W_{\theta}(t)^{2}}\left(1-\frac{1}{W_{\theta}(t)} \right)^{k-1},\ \forall k\geq1 \label{eq: loizzero},
\end{equation}\textcolor{red}{
and
\[
\mathbb{P}\left(Z_{0}(t)=0\mid N_{t}>0\right)=1-\frac{e^{-\theta t}W(t)}{W_{\theta}(t)}.
\]}
The main idea underlying our study is that the behaviour of any family in the CPP is the same as the clonal one but on a smaller time scale. 

For the rest of this paper, unless otherwise stated, the notation $\mathbb{P}_{t}$ refers to $\mathbb{P}\left(\cdot\mid N_{t}>0 \right)$ and $\mathbb{P}_{\infty}$ refers to the probability measure conditioned on the non-extinction event, denoted \emph{Non-Ex} in the sequel.

Finally, we recall the asymptotic behavior of the scale functions $W(t)$ and $W_{\theta}(t)$, which is widely used in the sequel.
\begin{lem}(Champagnat-Lambert \cite{CL2})
\label{lem: asyComp}
Assume $\alpha>0$, there exists a positive constant $\gamma$ such that
\[
e^{-\alpha t}\psi'(\alpha)W(t)-1=\mathcal{O}\left(e^{-\gamma t} \right).
\]
In the case that $\theta<\alpha$ (clonal supercritical case), 
\[
W_{\theta}(t)\underset{t\to\infty}{\sim}\psi'_{\theta}(\alpha-\theta)^{-1}e^{\left(\alpha-\theta \right)t}.
\]
In the case that $\theta>\alpha$ (clonal sub-critical case),
\[
W_{\theta}(t)=\frac{\theta}{\psi(\theta)}+\mathcal{O}\left(e^{-\left(\theta-\alpha\right)t} \right).
\]
In the case where $\theta=\alpha$ (clonal critical case),\textcolor{red}{
\[
W_{\alpha}(t)\underset{t\to\infty}{\sim}\frac{\alpha t}{\psi'(\alpha)}.
\]}
\end{lem}
From this lemma, \textcolor{red}{one can obtain that the probability that the clonal family reaches a fixed size at time $t$ decreases exponentially fast with $t$.}
\begin{cor}
\label{cor:estime}
In the supercritical case ($\alpha >0$), for any positive integer $k$,
\[
\textcolor{red}{\mathbb{P}_{t}}\left(Z_{0}(t)=k \right)=\mathcal{O}\left(e^{-\delta t}\right),
\]
where $\delta$ is equal to $\theta$ (resp. $2\alpha-\theta$) in the clonal critical and sub-critical cases (resp. supercritical case).
\end{cor}

\begin{rem}\textcolor{red}{
\label{rem:reviewed}
Note that Lemma \ref{lem: asyComp} implies in particular that, for any positive integer $k$,
\[
tW(t)^{k-1}=o\left(W(t)^{k} \right).
\]}
\end{rem}
\section{Statement of main results}
\label{sec:MainRes}
\textcolor{red}{In this section are stated the main results of the paper. In particular, the formulas for the moments of the frequency spectrum are given in Theorems \ref{thm: simple factorial moment} and \ref{thm: multiple factorial moment}.}

For two positive real numbers $a<t$, we denote by $N^{(t)}_{t-a}$ the number of individuals alive at time $t-a$ who have descent alive at time $t$. In the CPP at time $t$, $N^{(t)}_{t-a}$ corresponds to the number of branches higher than $t-a$, \textcolor{red}{that is $\text{Card}\left\{H_{i}\mid i\in\{0,\dots,N_{t}-1 \},\ H_{i}>t-a \right\}$}.

In the sequel, we use the following notation for multi-indexed sums: let $K,N$ be two positive integers and $\l_{1},\dots,\l_{K}$ some non-negative integers, then the notation
\[
\sum_{n^{1:K}_{1}+\dots+n^{1:K}_{N}=\l_{1:K}}
\]
refers to the sum
\[
\sum_{\underset{n_{1}^{K}+\dots+n_{N}^{K}=\l_{K}}{\underset{\dots}{n_{1}^{1}+\dots+n_{N}^{1}=\l_{1}}}}.
\]
In order to lighten notation, we also use the convention that for any integer $n$ and any negative integer $k$,
\[
\dbinom{n}{k}=0.
\]
\textcolor{red}{Recalling that $\mathbb{P}_{t}$ is the conditional probability on the event $\left\{N_{t}>0 \right\}$ and that $\mathbb{E}_{t}$ is the corresponding expectation, we now state our main results.}
\begin{thm}
\label{thm: simple factorial moment}
For any positive integers $n$ and $k$, we have,
\[
\mathbb{E}_{t}\left[\dbinom{A(k,t)}{n} \right]=\mathbb{E}_{t}\left\{\int_{0}^{t}\theta N^{(t)}_{t-a}\sum_{n_{1}+\dots+n_{N^{(t)}_{t-a}}=n-1}\mathbb{E}_{a}\left[\dbinom{A(k,a)}{n_{1}}\mathds{1}_{Z_{0}(a)=k}\right]\prod_{m=2}^{N^{(t)}_{t-a}}\mathbb{E}_{a}\left[\dbinom{A(k,a)}{n_{m}}\right]da\right\}.
\]

\end{thm}
\begin{thm}
\label{thm: multiple factorial moment}
Let $n_{1},\dots,n_{N}$ and $k_{1},\dots,k_{N}$ be positive integers. We have
\begin{multline}
\label{eq: momRep}
\mathbb{E}_{t}\left[\prod_{i=1}^{N}\dbinom{A(k_{i},t)}{n_{i}} \right]\\=\sum_{\ell=1}^{N}\mathbb{E}_{t}\left\{\int_{0}^{t}\theta N^{(t)}_{t-a}\sum_{n^{1:N}_{1}+\dots+n^{1:N}_{N^{(t)}_{t-a}}=n_{1:N}-\delta_{1:N,\ell}}\mathbb{E}_{a}\left[\prod_{i=1}^{N}\dbinom{A(k_{i},a)}{n^{i}_{1}}\mathds{1}_{Z_{0}(a)=k_{\ell}}\right]\prod_{m=2}^{N^{(t)}_{t-a}}\mathbb{E}_{a}\left[\prod_{i=1}^{N}\dbinom{A(k_{i},a)}{n_{m}^{i}}\right]da \right\},
\end{multline}
where $\delta$ refers to the Kronecker symbol \textcolor{red}{and $\delta_{1:N,l}=(\delta_{1,l},\dots,\delta_{N,l})$.}

\end{thm}
In Subsection \ref{ssec:jointzzero}, we also give formulas for the moments $\mathbb{E}_{t}[\prod_{i}\binom{A(k_{i},t)}{n_{i}}\mathds{1}_{Z_{0}(t)=\l}]$. \textcolor{red}{
These formulas are explicit in the sense that any moments can be computed recursively from the lower order moments. As an application, these formulas we obtain an elementary proof of the following law of large numbers.}
\begin{thm} We have,
\[
e^{-\alpha t}\left(A\left(k,t\right) \right)_{k\geq 1}\underset{t\to\infty}{\longrightarrow}\frac{\mathcal{E}}{\psi'(\alpha)}\left(c_{k} \right)_{k\geq 1}, \quad a.s.\text{ and in }L^{2},
\]
where $\mathcal{E}$ is an exponential random variable with parameter 1 conditionally on non-extinction, and $c_{k}$ is given by
\[
c_{k}:=\int_{0}^{\infty}\frac{\theta e^{-\theta s}}{W_{\theta}(s)^{2}}\left(1-\frac{1}{W_{\theta}(s)} \right)^{k-1}ds, \ \forall k\in\textcolor{red}{\mathbb{N}\backslash\{0\}}.
\]
\end{thm}
\textcolor{red}{ But before proving such Theorems \ref{thm: simple factorial moment} and \ref{thm: multiple factorial moment}, we need to introduce an important tool allowing to compute expectations of integrals w.r.t.\ random measures presenting particular independence structures. This is the purpose of the next section. }
\section{Expected stochastic integral using Palm theory.}
\label{sec: ranMes}
In this section, we use notation and vocabulary from \cite{DV}.
 
Let $\mathcal{X}$ a be Polish space.
We recall that a random measure is a measurable mapping from a probability space to the space $\mathcal{M}_{b}\left(\mathcal{X} \right)$ of all boundedly finite measures on $\mathcal{X}$, i.e.\ such that each bounded set has finite mass.

The purpose of this section is to prove an extension of the Campbell formula (see Proposition 13.1.IV in \cite{DV}), giving the expectation of an integral with respect to a random measure when the integrand has specific ``local" independence properties w.r.t. to the measure.

For this purpose, we need to introduce the notion of Palm measure related to a random measure $\mathcal{N}$. The presentation is borrowed from \cite{DV}. So let $\mathcal{N}$ be a random measure on $\mathcal{X}$ with intensity measure $\mu$, and $(X_{x},\ x\in\mathcal{X})$ be a continuous random process with value in $\mathbb{R}_{+}$.
Since this section is devoted to prove relations concerning only the distributions of $\mathcal{N}$ and $X$, we can assume without loss of generality that our random elements $X$ and $\mathcal{N}$ are defined (in the canonical way) on the space
\[
\mathcal{C}\left(\mathcal{X} \right)\times\mathcal{M}_{b}\left(\mathcal{X} \right),
\]
%endowed respectively with the (à completer) and the weak convergence topologies,
where $\mathcal{C}(\mathcal{X})$ denotes the space of continuous function on $\mathcal{X}$.
This space is Polish as a product of Polish spaces. We denote by $\mathcal{F}$ the corresponding product Borel $\sigma$-field.

  For the random measure $\mathcal{N}$, the corresponding Campbell measure $\mathcal{C}_{\mathcal{N}}$ is the measure defined on $\sigma\left(\mathcal{F}\times\mathcal{B}\left(\mathcal{X}\right)\right)$ by extension of the following relation on the semi-ring $\mathcal{F}\times \mathcal{B}\left(\mathcal{\mathcal{X}} \right)$,
\[
\mathcal{C}_{\mathcal{N}}\left(F\times B\right)=\mathbb{E}\left[\mathds{1}_{F}\mathcal{N}\left(B\right) \right], \quad F\in\mathcal{F}, \quad B\in\mathcal{B}\left( \mathcal{X}\right).
\]

It is straightforward to see that $\mathcal{C}_{\mathcal{N}} $ is $\sigma$-finite and for each $F$ in $\mathcal{F}$ the measure $\mathcal{C}_{\mathcal{N}}\left(F\times\cdot \right)$ is absolutely continuous with respect to $\mu$. Then, from Radon-Nikodym's theorem, for each $F\in\mathcal{F}$, there exist $y\in\mathcal{X}\mapsto P_{y}\left(F\right)$ in $L^{1}\left(\mu\right)$ such that,
\[
\mathcal{C}_{\mathcal{N}}\left(F\times B\right)=\int_{B}P_{y}\left(F\right)\ \mu\left(dy\right),
\]
uniquely defined up to its values on $\mu$-null sets.

Since our probability space is Polish, $P$ can be chosen to be a probabilistic kernel, i.e.  for all $F$ in $\mathcal{F}$,
\[
y\in\mathcal{X}\mapsto P_{y}\left(F\right) \text{ is mesurable},
\]
and for all $y$ in $\mathcal{X}$,
\[
F\in\mathcal{F}\mapsto P_{y}\left(F\right) \text{ is a probability measure.}
\]
The probability measure $P_{y}$ is called the Palm measure of $\mathcal{N}$ at point $y$.
Since $X$ is continuous, it is $\mathcal{B}\left(\mathcal{X}\right)\otimes\mathcal{F}$ measurable, and it is easily deduced from this point that
\begin{equation}
\label{eq:campbell}
\mathbb{E}\left[\int_{\mathcal{X}}X_{x}\ \mathcal{N}(dx)\right]=\int_{\mathcal{X}}\mathbb{E}_{P_{x}}\left[X_{x}\right]\ \mu(dx),
\end{equation}
where $\mathbb{E}_{P_{x}}$ denotes the expectation w.r.t.\ $P_{x}$. Formula \eqref{eq:campbell} is the so-called Campbell formula.

We can now state, the main results of this section which are the aforementioned extensions of the above formula.
\begin{thm}
\label{lem:locIndep}
Let $X$ be a continuous process from $\mathcal{X}$ to $\mathbb{R}_{+}$. Let $\mathcal{N}$ be a random measure on $\mathcal{X}$ with finite intensity measure $\mu$. Assume that $X$ is locally independent from $\mathcal{N}$, that is, for all $x\in\mathcal{X}$, there exists a neighbourhood $V_{x}$ of $x$ such that $X_{x}$ is independent from $\mathcal{N}\left(V_{x}\cap\cdot\right)$. Suppose moreover that there exists an integrable random variable $Y$ such that
\[
|X_{x}|\leq Y,\ \forall x\in\mathcal{X}, \ a.s.
\]and
\[
\mathbb{E}\left[Y\mathcal{N}\left(\mathcal{X} \right) \right]<\infty.
\]
Then we have
\begin{equation}
\label{eq:expectInt1}
\mathbb{E}\left[\int_{\mathcal{X}}X_{x}\ \mathcal{N}\left(dx\right)\right]=\int_{\mathcal{X}}\mathbb{E}\left[X_{x}\right]\ \mu\left(dx\right).
\end{equation}
\end{thm}
However, the continuity condition of the preceding theorem prevent the application of this result to our model. We need a more specific result.
\begin{thm}
\label{lem: rmexpc}
Let $X$ be a process from $[0,T]\times \mathcal{X}$ to $\mathbb{R}_{+}$ such that $X_{.,x}$ is càdlàg for all $x$ and $X_{s,.}$ is continuous for all $s$. Let $\mathcal{N}$ be a random measure on $[0,T]\times\mathcal{X}$ with finite intensity measure $\mu$. Assume that, for each $s$ in $[0,T]$, the family $\left(X_{s,x}, x\in\mathcal{X}\right)$ is independent from the restriction of $\mathcal{N}$ on $[0,s]$, that there exists an integrable random variable $Y$ such that
\[
|X_{s,x}|\leq Y,\ \forall x\in\mathcal{X},\ \forall s\in[0,t], \ a.s.
\]
and that
\[
\mathbb{E}\left[Y\mathcal{N}\left(\mathcal{X} \right) \right]<\infty.
\]
Then we have
\begin{equation}
\label{eq:expectInt}
\mathbb{E}\left[\int_{[0,T]\times\mathcal{X}}X_{s,x}\ \mathcal{N}\left(ds,dx\right)\right]=\int_{[0,T]\times\mathcal{X}}\mathbb{E}\left[X_{s,x}\right]\ \mu\left(ds,dx\right).
\end{equation}
\end{thm}
\medskip
Let $\llbracket 1,n\rrbracket$ denotes the set $\mathbb{N}\cap[1,n]$.
Before going further, we recall that a dissecting system is a sequence $\left\{A_{n,j}, j\in\llbracket1,K_{n}\rrbracket\right\}_{n\geq 0}$ of nested partitions of $\mathcal{X}$, where $\left(K_{n}\right)_{n\geq 0}$ is an increasing sequence of integers, such that 
\[
\lim\limits_{n\to\infty}\underset{j\in\llbracket1,K_{n}\rrbracket}{\max}\text{diam}\ A_{n,j}=0.
\]
In the spirit of the works of Kallenberg on the approximation of simple point processes, the proof of Theorems \ref{lem:locIndep} is based on the following Theorem which can be found in \cite{Kall} \textcolor{red}{or in \cite{PAM} (Section WIII.9)}.
\begin{thm}[Kallenberg \cite{Kall}]
\label{thm: kallen}
Let $\mu$ and $\nu$ be two finite measures on the Polish space $\mathcal{X}$, such that $\mu$ is absolutely continuous with respect to $\nu$. Let $f$ be the Radon-Nikodym derivative of $\mu$ w.r.t.\ $\nu$. Then, for any dissecting system $\left\{A_{n,j}, j\in\llbracket1,K_{n}\rrbracket\right\}_{n\geq 0}$ of $\mathcal{X}$, we have
\[
\lim\limits_{n\to\infty}\sum_{j=1}^{K_{n}}\frac{\mu\left(A_{n,j}\right)}{\nu\left(A_{n,j}\right)}\mathds{1}_{s\in A_{n,j}}=f\textcolor{red}{(s),\quad \text{for } \mu\text{-almost all } s\in\mathcal{X}}.
\]
\end{thm}
\begin{proof}[Proof of Theorem \ref{lem:locIndep}]
Let $\left\{A_{n,j}, j\in\llbracket1,K_{n}\rrbracket\right\}_{n\geq 0}$ be a dissecting system of $\mathcal{X}$. We denote by $A_{n}(x)$ the element of the partition $\left(A_{n,j}\right)_{1\leq j\leq K_{n}}$ which contain $x$. Let also $T$ be a denumerable dense subset of $\mathcal{X}$.
We use lower and upper approximations of $X$. More precisely, let for all positive integer $k$ and for all $a$ un $\mathcal{X}$,
\begin{align*}
\underline{X}^{(k)}_{x}:&=\inf\left\{X_{s}| s\in T\cap A_{k}(x) \right\}=\sum_{j=1}^{K_{k}}\underline{\chi}^{(k)}_{j}\mathds{1}_{x\in A_{j,k}},\\
\overline{X}^{(k)}_{x}:&=\sup\left\{X_{s}| s\in T\cap  A_{k}(x) \right\}=\sum_{j=1}^{K_{k}}\overline{\chi}^{(k)}_{j}\mathds{1}_{x\in A_{j,k}},
\end{align*}
with
\[
\overline{\chi}^{(k)}_{j}=\sup\left\{X_{s}| s\in A_{j,k}\cap T \right\}
\text{ and } \underline{\chi}^{(k)}_{j}=\inf\left\{X_{s}| s\in A_{j,k}\cap T \right\}.
\]
Note that the supremum and infinimum  are taken on $T\cap A_{k}(a)$ to ensure that $\underbar{X}^{(k)}_{j}$  and $\overline{X}^{(k)}_{j}$ are measurable, but the set $T$ could be removed by continuity of $X$.
We remark that, for any $j$, $k$, the measure
\[
\mathbb{E}\left[\overline{\chi}^{(k)}_{j}\mathcal{N}\left(\bullet \right) \right]
\]
is absolutely continuous with respect to $\mu$ and it follows from Campbell's formula \eqref{eq:campbell} that the Radon-Nikodym derivative is
\[
\mathbb{E}_{P_{x}}\left[\overline{\chi}^{(k)}_{j}\right].
\]
Thus, it follows from Theorem \ref{thm: kallen} that, $\mu$-a.e.,
\[
\mathbb{E}_{P_{x}}\left[\overline{\chi}_{j}^{(k)}\right]=\lim\limits_{n\to\infty}\frac{\mathbb{E}\left[\overline{\chi}^{(k)}_{j}\mathcal{N}\left(A_{n}(x) \right) \right]}{\mu\left(A_{n}(x)\right)}.
\]
Then, since $\underline{X}^{(k)}$ and $\overline{X}^{(k)}$ are finite sums of such random variables,
\[
\mathbb{E}_{P_{x}}\left[\overline{X}^{(k)}_{x}\right]=\lim\limits_{n\to\infty}\frac{\mathbb{E}\left[\overline{X}^{(k)}_{x}\mathcal{N}\left(A_{n}(x) \right) \right]}{\mu\left(A_{n}(x)\right)},
\]
and
\[
\mathbb{E}_{P_{x}}\left[\underline{X}^{(k)}_{x}\right]=\lim\limits_{n\to\infty}\frac{\mathbb{E}\left[\underline{X}^{(k)}_{x}\mathcal{N}\left(A_{n}(x) \right) \right]}{\mu\left(A_{n}(x)\right)},
\]
outside a $\mu$-null set which can be chosen independent of $k$ by countability.
Now, since
\[
\underline{X}^{(k)}_{x}\leq X_{x}\leq \overline{X}^{(k)}_{x},
\]
%  \[
%  \frac{ \mathbb{E}\left[\underbar{X}^{(k)}_{x}\mathcal{N}(A_{n}(x)) \right]}{\mathbb{E}\left[ \mathcal{N}(A_{n}(x))\right]}\leq  \frac{ \mathbb{E}\left[X_{x}\mathcal{N}(A_{n}(x)) \right]}{\mathbb{E}\left[ \mathcal{N}(A_{n}(x))\right]}\leq  \frac{ \mathbb{E}\left[\bar{X}^{(k)}_{x}\mathcal{N}(A_{n}(x)) \right]}{\mathbb{E}\left[ \mathcal{N}(A_{n}(x))\right]},
%  \]
it follows that
\[
 \mathbb{E}_{P_{x}}\left[\underline{X}^{(k)}_{x} \right]\leq  \liminf_{n\to\infty} \frac{ \mathbb{E}\left[X_{x}\mathcal{N}(A_{n}(x)) \right]}{\mathbb{E}\left[ \mathcal{N}(A_{n}(x))\right]}\leq  \limsup_{n\to\infty} \frac{ \mathbb{E}\left[X_{x}\mathcal{N}(A_{n}(x)) \right]}{\mathbb{E}\left[ \mathcal{N}(A_{n}(x))\right]}   \leq \mathbb{E}_{P_{x}}\left[\overline{X}^{(k)}_{x} \right], \ \mu-\text{almost everywhere}.
\]
Now, since $X$ is continuous,
\[
\overline{X}^{(k)}_{x}\underset{k\to\infty}{\longrightarrow} X_{x}\qquad \text{and}\qquad \underline{X}^{(k)}_{x}\underset{k\to\infty}{\longrightarrow} X_{x},
\]
it follows, from Lebesgue's Theorem, that
 \[
\mathbb{E}_{P_{x}}\left[X_{x} \right]=\lim\limits_{n\to\infty}\frac{ \mathbb{E}\left[X_{x}\mathcal{N}(A_{n}(x)) \right]}{\mathbb{E}\left[ \mathcal{N}(A_{n}(x))\right]}, \quad \mu-\text{almost everywhere}.
 \]
Now, since $A_{n,j}$ is a dissecting system, there exists an integer $N$ such that, for all $n>N$, $A_{n}(x)\subset V_{x}$. That is, for $n$ large enough,
 \[
 \frac{ \mathbb{E}\left[X_{x}\mathcal{N}(A_{n}(x)) \right]}{\mathbb{E}\left[ \mathcal{N}(A_{n}(x))\right]}=\mathbb{E}\left[X_{x}\right].
 \]
Finally,
\[
\mathbb{E}_{P_{x}}\left[X_{x}\right]=\mathbb{E}\left[X_{x}\right], \quad \mu-\text{almost everywhere}.
\]
And the conclusion comes from \eqref{eq:campbell}.
\end{proof}
\begin{proof}[Proof of Theorem \ref{lem: rmexpc}]
Clearly, we may assume without loss of generality that $T=1$. Define, for all integer $M$,
\[
X^{M}_{s,x}=\sum_{k=0}^{M-1}X_{\frac{k+1}{M},x}\mathds{1}_{s\in\left[\frac{k}{M},\frac{k+1}{M}\right)}.
\]
Since $X_{.,x}$ is c\`adl\`ag, this sequence of processes converges pointwise to $\left(X_{s,x},\ s\in[0,1]\right)$ for all $\omega$.
%outside a $\mathbb{P}$-null set which does not depend on $x$ by the separability of $\mathcal{X}$. 
Then, by Lebesgue's theorem,
\begin{align*}
\mathbb{E}\left[\int_{[0,1]\times\mathcal{X}}X_{s,x} \ \mathcal{N}(ds,dx)\right]&=\int_{[0,1]\times\mathcal{X}}\mathbb{E}_{P_{s,x}}\left[X_{s,x}\right]\ \mu(ds,dx),\\
&=\lim\limits_{M\to\infty}\sum_{k=0}^{M-1}\int_{[0,1]}\mathds{1}_{s\in\left[\frac{k}{M},\frac{k+1}{M}\right)\times\mathcal{X}}\ \mathbb{E}_{P_{s,x}}\left[X_{\frac{k+1}{M},x}\right]\mu(ds,dx).
\end{align*}
Clearly, for fixed $k$,
$
\left(s,x\right)\mapsto X_{\frac{k+1}{M},x}
$
is continuous on $[\frac{k}{M},\frac{k+1}{M}]\times\mathcal{X}$. 
%Hence Theorem \ref{thm: kallen} implies that, $\mu$-a.e.,
%\[
%\mathbb{E}_{P_{s,x}}\left[X_{\frac{k+1}{M},x}\right]=\lim\limits_{n\to\infty}\mathbb{E}\left[X_{\frac{k+1}{M},x}\,\Big|\,\mathcal{N}\left(A_{n}(s)\right)>0\right].
%\] 
%Finally, since $s<\frac{k+1}{M}$, there is an integer $n^{\star}$ such that, for all $n>n^{\star}$, $A_{n}(s)\subset[0,\frac{k+1}{M}]$ which leads by the independence property between $\mathcal{N}$ and $X$ that,
%\[
%\mathbb{E}\left[X_{\frac{k+1}{M},x}\,\Big|\,\mathcal{N}\left(A_{n}(s)\right)>0\right]=\mathbb{E}\left[X_{\frac{k+1}{M},x}\right].
%\]
%The result follows.
Hence, Theorem \ref{lem:locIndep} can be applied to
 \[\begin{array}{ccccc}
&  & \left[\frac{k}{M},\frac{k+1}{M}\right]\times \mathcal{X} & \to &\mathbb{R}_{+}, \\
& & (s,x) & \mapsto & X_{\frac{k+1}{M},x}, \\
\end{array}\]
to conclude the proof.
\end{proof}
\section{Proofs of the moments formulas}
\label{sec:moments}
\textcolor{red}{
The main goal of this section is to prove Theorems \ref{thm: simple factorial moment} and \ref{thm: multiple factorial moment}. Their proofs are given in Subsection \ref{ssec:proofs}. Subsection \ref{ssec:jointzzero} is devoted to the computation of the joint moments of the frequency spectrum with $\mathds{1}_{Z_{0}(t)=\ell}$. Subsection \ref{ssec:convariance} shows an application of our theorems to the computation of the covariances of the frequency spectrum. The next subsection gives the key decomposition of the CPP.}
\subsection{Recursive construction of the CPP}
\label{ssec:construc}
\textcolor{red}{ Here we describe the general idea of the proof of Theorems \ref{thm: simple factorial moment} and \ref{thm: multiple factorial moment} and} give an alternative construction of the CPP. We consider the CPP at some time $t$.
Suppose that a mutation occurs on branch $i$ at a time $a$. Then, by construction of the CPP, the future of this family depends only on what happens on the branches $\left(H_{j}, i\leq j < \tau \right) $ (see Figure \ref{fig : coalpointprocmut}), where 
\[
\tau=\inf\left\{j>i \mid \ H_{j}\geq a\right\}.
\]

 In fact, this set of branches is also a CPP with scale function $W$ stopped at $a$ (we talk about sub-CPP), and the number of individuals carrying the mutation at time $t$ is the number of clonal individuals in this sub-CPP.
\begin{figure}[ht]

\unitlength 2mm % = 4.55pt
\linethickness{0.4pt}

\begin{picture}(66,33)(-5,10)
\put(3,39.5){\makebox{\small{$0$}}}
\put(4.5,14.5){\makebox{\small{$t$}}}
\put(5.5,15){\line(1,0){1}}
\color{gray5}

\put(4,39.875){\line(1,0){62}}
\put(10,40){\line(0,-1){9}}
\put(14,40){\line(0,-1){11.5}}
\put(18,40){\line(0,-1){4}}
\put(22,40){\line(0,-1){7}}

\color{black}

\multiput(26,27)(-1,0){21}{\line(-1,0){0.5}}
\put(4,26.5){\text{$a$}}
\put(26,27){\circle*{1.061}}
\put(26,40){\line(0,-1){13}}
\color{gray5}
\put(26,27){\line(0,-1){3}}
\color{black}
\put(30,40){\line(0,-1){6}}

\put(34,39.875){\line(0,-1){8.5}}
\put(38,40){\line(0,-1){5.5}}

\put(42,40){\line(0,-1){11.5}}
\put(46,40){\line(0,-1){3.625}}
\color{gray5}

\put(50,40){\line(0,-1){22}}
\put(54,40){\line(0,-1){5}}
\put(58,40){\line(0,-1){7.5}}
\put(62,39.875){\line(0,-1){5}}
\put(6,40){\line(0,-1){25}}
%\dashline{1}(6,15)(6,11)
% \put(5.93,14.93){\line(0,-1){.8}}
% \put(5.93,13.33){\line(0,-1){.8}}
% \put(5.93,11.73){\line(0,-1){.8}}

%\put(6,15){\line(0,-1){.8}}
%\put(6,13.33){\line(0,-1){.8}}
%\put(6,11.73){\line(0,-1){.8}}
%\end
%\dashline{1}(10,31)(6,31)

\put(9.93,30.93){\line(-1,0){.8}}
\put(8.33,30.93){\line(-1,0){.8}}
\put(6.73,30.93){\line(-1,0){.8}}
%\end

%\dashline{1}(14,28.5)(6,28.5)

\put(13.93,28.43){\line(-1,0){.8889}}
\put(12.152,28.43){\line(-1,0){.8889}}
\put(10.374,28.43){\line(-1,0){.8889}}
\put(8.596,28.43){\line(-1,0){.8889}}
\put(6.819,28.43){\line(-1,0){.8889}}
%\end

%\dashline{1}(17.875,36)(13.875,36)
\put(17.805,35.93){\line(-1,0){.8}}
\put(16.205,35.93){\line(-1,0){.8}}
\put(14.605,35.93){\line(-1,0){.8}}
%\end
%\dashline{1}(22,33)(14,33)
\put(21.93,32.93){\line(-1,0){.8889}}
\put(20.152,32.93){\line(-1,0){.8889}}
\put(18.374,32.93){\line(-1,0){.8889}}
\put(16.596,32.93){\line(-1,0){.8889}}
\put(14.819,32.93){\line(-1,0){.8889}}
%\end
%\dashline{1}(26,24)(6,24)
\put(25.93,23.93){\line(-1,0){.9524}}
\put(24.025,23.93){\line(-1,0){.9524}}
\put(22.12,23.93){\line(-1,0){.9524}}
\put(20.215,23.93){\line(-1,0){.9524}}
\put(18.311,23.93){\line(-1,0){.9524}}
\put(16.406,23.93){\line(-1,0){.9524}}
\put(14.501,23.93){\line(-1,0){.9524}}
\put(12.596,23.93){\line(-1,0){.9524}}
\put(10.692,23.93){\line(-1,0){.9524}}
\put(8.787,23.93){\line(-1,0){.9524}}
\put(6.882,23.93){\line(-1,0){.9524}}
%\end
%\dashline{1}(30,34)(26,34)
\color{black}
\put(29.93,33.93){\line(-1,0){.8}}
\put(28.33,33.93){\line(-1,0){.8}}
\put(26.73,33.93){\line(-1,0){.8}}
\color{gray5}
%\end
\color{black}

%\dashline{1}(34,31.5)(26,31.5)
\put(33.93,31.43){\line(-1,0){.8889}}
\put(32.152,31.43){\line(-1,0){.8889}}
\put(30.374,31.43){\line(-1,0){.8889}}
\put(28.596,31.43){\line(-1,0){.8889}}
\put(26.819,31.43){\line(-1,0){.8889}}
\color{gray5}
%\end
%\dashline{1}(38,34.5)(34,34.5)
\color{black}
\put(37.93,34.43){\line(-1,0){.8}}
\put(36.33,34.43){\line(-1,0){.8}}
\put(34.73,34.43){\line(-1,0){.8}}
\color{gray1}
%\end
%\dashline{1}(42,28.5)(26,28.5)
\color{black}
\put(41.93,28.43){\line(-1,0){.9412}}
\put(40.047,28.43){\line(-1,0){.9412}}
\put(38.165,28.43){\line(-1,0){.9412}}
\put(36.283,28.43){\line(-1,0){.9412}}
\put(34.4,28.43){\line(-1,0){.9412}}
\put(32.518,28.43){\line(-1,0){.9412}}
\put(30.636,28.43){\line(-1,0){.9412}}
\put(28.753,28.43){\line(-1,0){.9412}}
\put(26.871,28.43){\line(-1,0){.9412}}
\color{gray5}
%\end
%\dashline{1}(46,36.5)(42,36.5)
\color{black}
\put(45.93,36.43){\line(-1,0){.8}}
\put(44.33,36.43){\line(-1,0){.8}}
\put(42.73,36.43){\line(-1,0){.8}}
\color{gray1}
%\end
%\dashline{1}(50,18)(6,18)
\put(49.93,17.93){\line(-1,0){.9778}}
\put(47.974,17.93){\line(-1,0){.9778}}
\put(46.019,17.93){\line(-1,0){.9778}}
\put(44.063,17.93){\line(-1,0){.9778}}
\put(42.107,17.93){\line(-1,0){.9778}}
\put(40.152,17.93){\line(-1,0){.9778}}
\put(38.196,17.93){\line(-1,0){.9778}}
\put(36.241,17.93){\line(-1,0){.9778}}
\put(34.285,17.93){\line(-1,0){.9778}}
\put(32.33,17.93){\line(-1,0){.9778}}
\put(30.374,17.93){\line(-1,0){.9778}}
\put(28.419,17.93){\line(-1,0){.9778}}
\put(26.463,17.93){\line(-1,0){.9778}}
\put(24.507,17.93){\line(-1,0){.9778}}
\put(22.552,17.93){\line(-1,0){.9778}}
\put(20.596,17.93){\line(-1,0){.9778}}
\put(18.641,17.93){\line(-1,0){.9778}}
\put(16.685,17.93){\line(-1,0){.9778}}
\put(14.73,17.93){\line(-1,0){.9778}}
\put(12.774,17.93){\line(-1,0){.9778}}
\put(10.819,17.93){\line(-1,0){.9778}}
\put(8.863,17.93){\line(-1,0){.9778}}
\put(6.907,17.93){\line(-1,0){.9778}}

%\end
%\dashline{1}(54,35)(50,35)
\put(53.93,34.93){\line(-1,0){.8}}
\put(52.33,34.93){\line(-1,0){.8}}
\put(50.73,34.93){\line(-1,0){.8}}
%\end
%\dashline{1}(58,32.5)(50,32.5)
\put(57.93,32.43){\line(-1,0){.8889}}
\put(56.152,32.43){\line(-1,0){.8889}}
\put(54.374,32.43){\line(-1,0){.8889}}
\put(52.596,32.43){\line(-1,0){.8889}}
\put(50.819,32.43){\line(-1,0){.8889}}
%\end
%\dashline{1}(62,35)(58,35)
\put(61.93,34.93){\line(-1,0){.8}}
\put(60.33,34.93){\line(-1,0){.8}}
\put(58.73,34.93){\line(-1,0){.8}}
%\end
\put(6,41){\makebox(0,0)[cc]{$0$}}
\put(10,41){\makebox(0,0)[cc]{$1$}}
\put(14,41){\makebox(0,0)[cc]{$2$}}
\put(18,41){\makebox(0,0)[cc]{$3$}}
\put(22,41){\makebox(0,0)[cc]{$4$}}
\color{black}
\put(26,41){\makebox(0,0)[cc]{$5$}}
\put(30,41){\makebox(0,0)[cc]{$6$}}
\put(34,41){\makebox(0,0)[cc]{$7$}}
\put(38,41){\makebox(0,0)[cc]{$8$}}
\put(42,41){\makebox(0,0)[cc]{$9$}}
\put(46,41){\makebox(0,0)[cc]{$10$}}
\color{gray5}
\put(50,41){\makebox(0,0)[cc]{$12$}}
\put(54,41){\makebox(0,0)[cc]{$13$}}
\put(58,41){\makebox(0,0)[cc]{$14$}}
\put(62,41){\makebox(0,0)[cc]{$15$}}
\end{picture}

\caption{The future of a mutation only depends on a sub-tree of the genealogical tree.}
\label{fig : coalpointprocmut}
\end{figure}
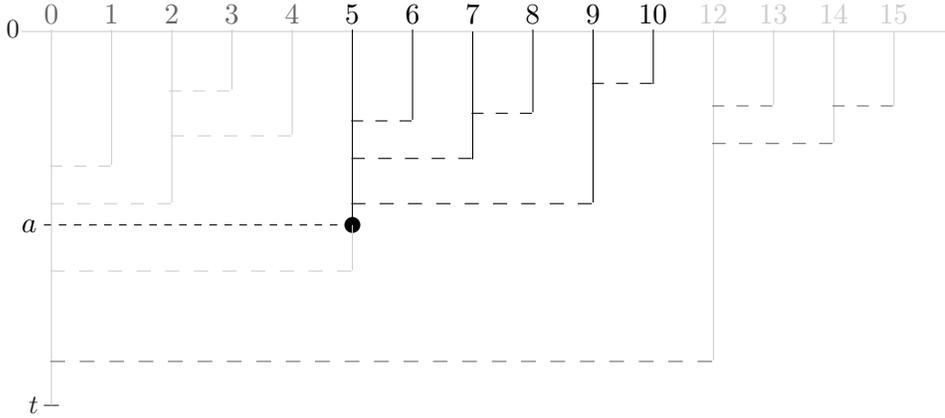

To capitalize on this fact, we introduce a construction of the CPP which underlines this independence. 
Suppose we are given a sequence $\left(\mathcal{P}^{(i)}\right)_{i\geq 1}$  of coalescent point processes stopped at time $a$ with scale function $W$. 
Then, take an independent CPP $\hat{\mathcal{P}}$, where the law of the branches corresponds to the excess over $a$ of a branch with scale function $W$ conditioned to be higher than $a$. As stated in the next proposition, the tree build from the grafting of the $\mathcal{P}^{(i)}$ above each branch of $\hat{\mathcal{P}}$ is also a CPP with scale function $W$ stopped at time $t$ (see Figure \ref{fig : coalpointproc}).
\begin{figure}[ht]
\unitlength 1.6mm % = 4.55pt
\linethickness{0.9pt}
\begin{picture}(100,30)(0,0)
\put(1,19){\makebox{\small{$0$}}}
\put(3,20){\line(1,0){1}}
\put(1,-1){\makebox{\small{$t$}}}
%\put(-1.5,9.5){\makebox{\small{$t-a$}}}
\put(1,9.5){\makebox{\small{$a$}}}
\put(3,10){\line(1,0){1}}
\put(3.2,0){\line(1,0){1}}
\put(3,0){
\color{black}
\put(0,10){\line(1,0){100}}
\put(0,10){\line(0,-1){10}}

\put(50,10){\line(0,-1){7}}
\put(25,10){\line(0,-1){3}}
\put(75,10){\line(0,-1){5}}
\multiput(25,7)(-1,0){25}{\line(-1,0){0.5}}
\multiput(50,3)(-1,0){50}{\line(-1,0){0.5}}
\multiput(75,5)(-1,0){25}{\line(-1,0){0.5}}
\setlength{\unitlength}{0.8mm}
\linethickness{0.62pt}
\put(0,20){
\put(13,22){\makebox{$\mathcal{P}^{(1)}$}}
\color{gray4}
\put(0,0){\line(0,1){20}}
\color{gray4}
\put(0,20){\line(1,0){10}}
\color{gray1}
\put(10,20){\line(1,0){10}}
\color{gray2}
\put(20,20){\line(1,0){10}}
\color{gray3}
\multiput(30,20)(1,0){10}{\line(-1,0){0.5}}
\color{black}
\color{gray1}
\put(20,20){\line(0,-1){5}}
\color{gray1}
\put(10,20){\line(0,-1){10}}
\color{gray3}
\put(30,20){\line(0,-1){15}}
\multiput(30,5)(-1,0){30}{\line(-1,0){0.5}}
\multiput(20,15)(-1,0){10}{\line(-1,0){0.5}}
\multiput(10,10)(-1,0){10}{\line(-1,0){0.5}}
}
\put(50,20){
\put(13,22){\makebox{$\mathcal{P}^{(2)}$}}
\color{gray4}

\put(0,0){\line(0,1){20}}
\color{gray4}
\put(0,20){\line(1,0){10}}
\color{gray1}
\put(10,20){\line(1,0){10}}
\color{gray2}
\put(20,20){\line(1,0){10}}
\color{gray3}
\multiput(30,20)(1,0){10}{\line(-1,0){0.5}}
\color{black}
\color{gray1}
\put(20,20){\line(0,-1){5}}
\color{gray1}
\put(10,20){\line(0,-1){10}}
\color{gray3}
\put(30,20){\line(0,-1){15}}
\multiput(30,5)(-1,0){30}{\line(-1,0){0.5}}
\multiput(20,15)(-1,0){10}{\line(-1,0){0.5}}
\multiput(10,10)(-1,0){10}{\line(-1,0){0.5}}}
\put(100,20){\color{gray4}
\put(0,0){\line(0,1){20}}
\put(13,22){\makebox{$\mathcal{P}^{(3)}$}}
\color{gray4}
\put(0,20){\line(1,0){10}}
\color{gray1}
\put(10,20){\line(1,0){10}}
\color{gray2}
\put(20,20){\line(1,0){10}}
\color{gray3}
\multiput(30,20)(1,0){10}{\line(-1,0){0.5}}
\color{black}
\color{gray1}
\put(20,20){\line(0,-1){5}}
\color{gray1}
\put(10,20){\line(0,-1){10}}
\color{gray3}
\put(30,20){\line(0,-1){15}}
\multiput(30,5)(-1,0){30}{\line(-1,0){0.5}}
\multiput(20,15)(-1,0){10}{\line(-1,0){0.5}}
\multiput(10,10)(-1,0){10}{\line(-1,0){0.5}}}
\put(150,20){\color{gray4}
\put(0,0){\line(0,1){20}}
\put(13,22){\makebox{$\mathcal{P}^{(4)}$}}
\color{gray4}
\put(0,20){\line(1,0){10}}
\color{gray1}
\put(10,20){\line(1,0){10}}
\color{gray2}
\put(20,20){\line(1,0){10}}
\color{gray3}
\multiput(30,20)(1,0){10}{\line(-1,0){0.5}}
\color{black}
\color{gray1}
\put(20,20){\line(0,-1){5}}
\color{gray1}
\put(10,20){\line(0,-1){10}}
\color{gray3}
\put(30,20){\line(0,-1){15}}
\multiput(30,5)(-1,0){30}{\line(-1,0){0.5}}
\multiput(20,15)(-1,0){10}{\line(-1,0){0.5}}
\multiput(10,10)(-1,0){10}{\line(-1,0){0.5}}}
}
\end{picture}
\caption{Grafting of trees.}
\label{fig : coalpointproc}
\end{figure}
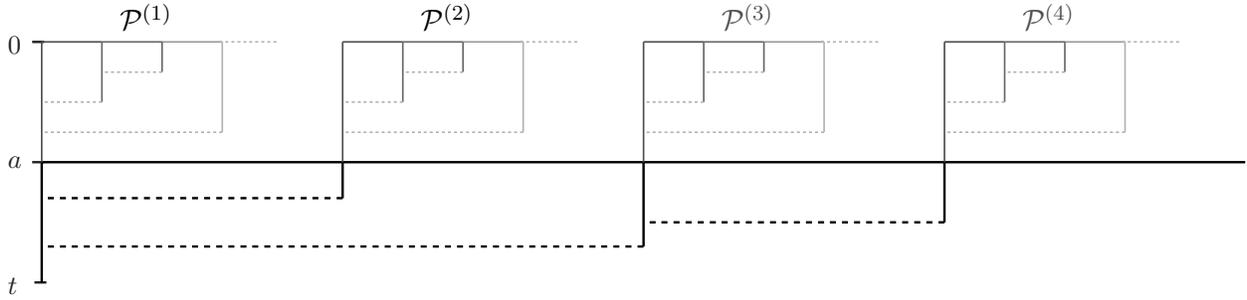

\begin{prop}
\label{lem: construction}
Let $\left(\mathcal{P}^{(i)}\right)_{i\geq 1}$ be an i.i.d.\ sequence of coalescent point processes with scale function $W$ at time $a$, and let $\left(N^{i}_{a} \right)_{i\geq 1}$ be their respective population sizes.
Let $\hat{\mathcal{P}}$ be a coalescent point process, independent of the previous family, with scale function
\[\hat{W}(t):=\frac{W(t+a)}{W(a)},\]
at time $t-a$, and let  $\hat{N}_{t-a}$ denotes its population size.
Let $S_{0}:=0$ and
\[
S_{i}:=\sum_{j=1}^{i}N^{j}_{a},\quad \forall i\geq 1.
\]
Then the random vector $\left(H_{k},\ 0\leq k\leq S_{\hat{N}_{a}-1}\right)$ defined, for all $k\geq0$, by 

\[H_{k}=
\left\{
\begin{array}{lll}
\mathcal{P}^{(i+1)}_{k-S_{i}}& \text{if } S_{i}<k<S_{i+1}, &\quad\text{for some }i\geq 0,\\
\hat{\mathcal{P}}_{i}+a& \text{if } k=S_{i}, &\quad\text{for some }i\geq 0,
\end{array}\
\right.
\]
is a CPP with scale function $W$ at time $t$, for which $N^{(t)}_{t-a}=\hat{N}_{t-a}$ a.s.
\end{prop}
\begin{proof}
Note that $H_{0}=\hat{\mathcal{P}}_{0}+a$.
To prove the result, it is enough to show that the sequence $\left(H_{k}\right)_{k\geq1}$ is an i.i.d.\ sequence \textcolor{red}{with the same law as $H$, given by
\[
\mathbb{P}\left(H>s \right)=\frac{1}{W(s)},\quad \forall s>0.
\]}
\textcolor{red}{The independence follows from the construction. We details the computation for the joint law of $(H_{l},H_{k})$ and leave the easy extension to the general case to the reader.} Let $k>l$ be two positive integers, and let also $s_{1},s_{2}$ be two positive real numbers.
We denote by $\mathcal{S}$ the random set $\left\{S_{i},\ i\geq1 \right\}$. Hence,
\begin{align*}
\mathbb{P}\left(H_{k}<s_{1},\ H_{l}<s_{2} \right)=&\mathbb{P}\left(H<s_{1}\mid H<a\right)\mathbb{P}\left(H<s_{2}\mid H<a\right)\mathbb{P}\left(l\notin \mathcal{S},\ k\notin \mathcal{S} \right)\\
&+\mathbb{P}\left(a+\hat{H}<s_{1}\right)\mathbb{P}\left(H<s_{2}\mid H<a\right)\mathbb{P}\left(l\notin \mathcal{S},\ k\in \mathcal{S} \right)\\
&+\mathbb{P}\left(H<s_{1}\mid H<a\right)\mathbb{P}\left(a+\hat{H}<s_{2}\right)\mathbb{P}\left(l\in \mathcal{S},\ k\notin \mathcal{S} \right)\\
&+\mathbb{P}\left(a+\hat{H}<s_{1}\right)\mathbb{P}\left(a+\hat{H}<s_{2}\right)\mathbb{P}\left(l\in \mathcal{S},\ k\in \mathcal{S} \right),\\
\end{align*}
where $\hat{H}$ denotes a random variable with the law of the branches of $\hat{\mathcal{P}}$, \textcolor{red}{ i.e.\ such that
\[
\mathbb{P}\left(\hat{H}>s \right)=\frac{W(s)}{W(s+a)},\quad \forall s>0.
\]}
Now, since the random variables $S_{i}$ are sums of geometric random variables, we get
\begin{multline*}
\mathbb{P}\left(H_{k}<s_{1},\ H_{l}<s_{2} \right)\\=\left(p\mathbb{P}\left(H<s_{1}\mid H<a\right)+(1-p)\mathbb{P}\left(a+\hat{H}<s_{1}\right) \right)\left(p\mathbb{P}\left(H<s_{2}\mid H<a\right)+(1-p)\mathbb{P}\left(a+\hat{H}<s_{2}\right)\right),
\end{multline*}
with $p=\mathbb{P}\left(k\in \mathcal{S} \right)$.
Moreover we have,
\begin{align*}
\mathbb{P}\left(H_{k}\leq s\right)=&\sum_{i\geq 1}\Big\{\mathbb{P}\left(H_{k}\leq s\mid k\in\rrbracket S_{i-1},S_{i}\llbracket \right)\mathbb{P}\left( k\in\rrbracket S_{i-1},S_{i}\llbracket\right)\\
& +\mathbb{P}\left(H_{k}\leq s\mid k=S_{i} \right)\mathbb{P}\left( k=S_{i}\right)\Big\}\\
=&\mathbb{P}\left(H\leq s\mid H<a\right)\mathbb{P}\left(\bigcup_{i\geq1} \left\{k\in\rrbracket S_{i-1},S_{i}\llbracket\right\}\right)\\
& +\mathbb{P}\left(H\leq s\mid H>a \right)\mathbb{P}\left(\bigcup_{i\geq1}\left\{k=S_{i}\right\}\right).
\end{align*}

Since the $S_{i}$'s are sums of geometric random variables of parameters $\hat{W}(t-a)^{-1}$, they follow binomial negative distributions with parameters $i$ and $\hat{W}(t-a)^{-1}$. \textcolor{red}{Hence, since
\[
\mathbb{P}\left(S_{i}=k\right)=
\left\{
\begin{array}{ll}
0,&\text{ if } k<i,\\
\dbinom{k-1}{i-1}\hat{W}(t-a)^{-i}\left(1-\hat{W}(t-a)^{-1} \right)^{k-i},& \text{ else,}
\end{array}
\right.
\]}
 some elementary calculus leads to 
 \[
\mathbb{P}\left(\bigcup_{i\geq1}\left\{k=S_{i}\right\}\right)=\mathbb{P}\left(H>a\right),\quad \forall k\in\mathbb{N}.
 \]
 which ends the proof.
\end{proof}

 Proposition \ref{lem: construction} shows that, under $\mathbb{P}_{t}$, $N^{(t)}_{t-a}$ is geometrically distributed with parameter $\frac{W(a)}{W(t)}$.
Moreover, although $N_{t-a}$ may depends on informations which do not appear in the CPP stopped at time $t$, $N^{(t)}_{t-a}$ only depends on the lineages of the population at time $t$. 
%More precisely, $N^{(t)}_{t-a}$ is the number of lineage alive at time $t$ which were alive a time $t-a$, or the number of individuals alive at time $t-a$ who have descendant alive at time $t$.

A very simple application of this construction is the derivation of the expectation of $A(k,t)$. 
%Moreover, it is a nice consequence of the property that every family in the CPP behave as the clonal family on a lower time scale.  
Recall that this expectation was first calculated in \cite{CL1}, with a much more complicated proof.

\begin{thm}(\cite[Cor. 3.4]{CL1})
\label{thm : expc}
For any positive integer $k$, we have
\[
\mathbb{E}_{t}\left[A(k,t)\right]=W(t)\int_{0}^{t}\frac{\theta e^{-\theta a}}{W_{\theta}(a)^{2}}\left(1-\frac{1}{W_{\theta}(a)}\right)^{k-1}da.
\]
\end{thm}

\begin{proof}
 Since $A(k,t)$ is the number of types represented at time $t$ by $k$ individuals, it is equivalent to enumerate all the mutations
and ask if they have exactly $k$ clonal children at time $t$. This remark leads to the following integral representation of $A(k,t)$:
\begin{equation}
\label{eq:aktrep}
A\left(k,t\right)=\int_{[0,t]\times \mathbb{N}}\mathds{1}_{Z_{0}^{i}(a)=k}\ \mathcal{N}\left(da,di\right),
\end{equation}
where $\mathcal{N}$ is defined in \eqref{eq:randommeasure}, and $Z_{0}^{i}(a)$ denotes the number of alive individuals at time $t$ carrying the same type as the type carried at time $t-a$ on the $i$th branch of the CPP of the individuals alive at time $t$ (the notation comes from the fact that $Z_{0}^{i}(a)$ corresponds to the size of the clonal family in the sub-CPP induced by the $i$th individual at time $t-a$, see Figure \ref{fig : coalpointprocmut}).
From Proposition \ref{lem: construction}, it follows that $\mathds{1}_{Z_{0}^{i}(a)=k}$ satisfies the conditions of Theorem \ref{lem: rmexpc}, so
\[
\mathbb{E}_{t}\left[A\left(k,t\right)\right]=\int_{0}^{t}\theta \ \mathbb{P}_{a}\left(Z_{0}(a)=k\right) \mathbb{E}_{t}N^{(t)}_{t-a}\ da=W(t)\int_{0}^{t}\frac{\theta e^{-\theta a}}{W_{\theta}(a)^{2}}\left(1-\frac{1}{W_{\theta}(a)}\right)^{k-1}da,
\]
using \eqref{eq: loizzero}.
\end{proof}

\subsection{Proof of Theorems \ref{thm: simple factorial moment} and \ref{thm: multiple factorial moment}}

\label{ssec:proofs}

Let $a$ and $t$ be two positive real numbers such that $a<t$, and $n$ a positive integer. We call $k$-mutation, a mutation represented by $k$ alive individuals at time $t$ in the splitting tree. %We consider the construction of lemma \ref{lem: construction}. 
Let $\left(A^{(i)}(k,a)\right)_{k\geq 1}$ be the frequency spectrum in the $i$-th subtree of construction provided by Proposition \ref{lem: construction}.

To count the number of $n$-tuples in the set of $k$-mutations, we look along the tree and seek for mutations in the CPP. For each $k$-mutation encountered, we count the number of $(n-1)$-tuples made of younger $k$-mutations. The $(n-1)$-tuples should be enumerated by decomposition in each subtree in order to exploit the independence property of the subtrees of Proposition \ref{lem: construction}.
%of  we look along the tree until meeting a $k$-represented mutation and then count the number of $n-1$-combination in the reste of the tree.Suppose that mutation appear at time $a$, then the number of $n-1$-combination in the rest of the tree is the number of n-1-mutation in each  subtree plus ect..;...
Suppose that a mutation is encountered at a time $a$, then the number of $(n-1)$-tuples made of younger mutations is given by
\[
\sum_{n_{1}+\dots+n_{N^{(t)}_{t-a}}=n-1}\prod_{m=1}^{N^{(t)}_{t-a}}\dbinom{A^{(m)}(k,a)}{n_{m}}.
\]

So the number $\dbinom{A(k,t)}{n}$ of $n$-tuples of $k$-mutations is given by
\begin{align}
\dbinom{A(k,t)}{n}&=\int_{[0,t]\times\mathbb{N}}\mathds{1}_{Z_{0}^{i}(a)=k}\sum_{n_{1}+\dots+n_{N^{(t)}_{t-a}}=n-1}\prod_{m=1}^{N^{(t)}_{t-a}}\dbinom{A^{(m)}(k,a)}{n_{m}}\mathcal{N}(da,di),\label{eq: intRep}\\
&=\sum_{\l\geq 1}\int_{[0,t]\times\mathbb{N}}\mathds{1}_{Z_{0}^{i}(a)=k}\sum_{n_{1}+\dots+n_{l}=n-1}\prod_{m=1}^{\l}\dbinom{A^{(m)}(k,a)}{n_{m}}\mathds{1}_{N^{(t)}_{t-a}=\l}\ \mathcal{N}(da,di)\notag,
\end{align}
where $Z_{0}^{i}(a)$ was defined in the proof of Theorem \ref{thm : expc}.
Finally, using the independence provided by Proposition \ref{lem: construction}, it follows from Theorem \ref{lem: rmexpc} applied to all the integrals with respect to the random measures $\mathds{1}_{N^{(t)}_{t-a}=k}\mathcal{N}(da,di)$, that
\begin{align*}
\mathbb{E}_{t}\left[\dbinom{A(k,t)}{n} \right]&=\mathbb{E}_{t}\int_{[0,t]\times\mathbb{N}} \sum_{n_{1}+\dots+n_{N^{(t)}_{t-a}}=n-1}\mathbb{E}_{a}\left[\dbinom{A(k,a)}{n_{1}}\mathds{1}_{Z_{0}(a)=k}\right]\prod_{m=2}^{N^{(t)}_{t-a}}\mathbb{E}_{a}\left[\dbinom{A(k,a)}{n_{m}}\right]\mathcal{N}\left(da,di\right).\\
\end{align*}
Finally, using that the $\mathcal{N}(da,di)=\mathds{1}_{H_{i}>t-a}\mathds{1}_{i<N_{t}}\mathcal{P}(di,da)$ where $\mathcal{P}$ independent from the CPP (and, hence, from $N^{(t)}_{t-a}$), it follows that
\begin{align}
\label{eq:petiteeq}
\mathbb{E}_{t}\left[\dbinom{A(k,t)}{n} \right]&=\mathbb{E}_{t}\int_{[0,t]} \sum_{n_{1}+\dots+n_{N^{(t)}_{t-a}}=n-1}\mathbb{E}_{a}\left[\dbinom{A(k,a)}{n_{1}}\mathds{1}_{Z_{0}(a)=k}\right]\notag\\
&\quad\times\prod_{m=2}^{N^{(t)}_{t-a}}\mathbb{E}_{a}\left[\dbinom{A(k,a)}{n_{m}}\right]\int_{\mathbb{N}}\mathds{1}_{H_{i}>t-a}\mathds{1}_{i<N_{t}}\ C(di)\theta da,
\notag\\&=\mathbb{E}_{t}\int_{[0,t]}\theta N^{(t)}_{t-a}\sum_{n_{1}+\dots+n_{N^{(t)}_{t-a}}=n-1}\mathbb{E}_{a}\left[\dbinom{A(k,a)}{n_{1}}\mathds{1}_{Z_{0}(a)=k}\right]\prod_{m=2}^{N^{(t)}_{t-a}}\mathbb{E}_{a}\left[\dbinom{A(k,a)}{n_{m}}\right]da,
\end{align}
which ends the proof of Theorem \ref{thm: simple factorial moment}.

The proof of Theorem \ref{thm: multiple factorial moment} follows exactly the same lines, and we leave it to the reader.
\subsection{Joint moments of the frequency spectrum and $\mathds{1}_{Z_{0}(t)=\l}$}
\label{ssec:jointzzero}
In order to compute the terms of the form
\[
\mathbb{E}_{t}\left[\prod_{i=1}^{N}\dbinom{A(k_{i},t)}{n_{i}}\mathds{1}_{Z_{0}(t)=\l} \right]
\]
involved in \eqref{eq: momRep}, we need to extend the representation \eqref{eq: intRep} of $\binom{A(k,t)}{n}$ to take into account the indicator function of $\left\{Z_{0}(t)=\l\right\}$. To do this, when integrating w.r.t. $\mathcal{N}(da,di)$, we need to ask that the sum of the number of clonal individuals in each subtree for which the type at time $t-a$ is the ancestral type, is equal to $k$.
 \textcolor{red}{We begin with the case
\[
\mathbb{E}\left[A(k,t)\mathds{1}_{Z_{0}(t)=\l}\right]
\]
in order to highlight the ideas. In this case, we have the following result.}
\begin{prop}
\label{prop:oneProp}
\begin{align}
\label{eq:indicat}
\mathbb{E}_{t}\left[A(k,t)\mathds{1}_{Z_{0}(t)=\l}\right]=&\mathbb{E}_{t}\int_{0}^{t}\left(N^{(t)}_{t-a}-Z^{(t)}_{0}(a)\right)\mathbb{P}_{a}\left(Z_{0}(a)=k \right)\sum_{\ell_{1}+\dots+\ell_{Z_{0}^{(t)}(a)}=\l}\prod_{i=1}^{Z_{0}^{(t)}(a)}\mathbb{P}_{a}\left(Z_{0}(a)=\ell_{i}\right)\ \theta da\notag\\
&+\mathbb{E}_{t}\int_{0}^{t}Z_{0}^{(t)}(a)\mathbb{P}_{a}\left(Z_{0}(a)=k \right)\sum_{\ell_{1}+\dots+\ell_{Z_{0}^{(t)}(a)-1}=\l}\prod_{i=1}^{Z_{0}^{(t)}(a)-1}\mathbb{P}_{a}\left(Z_{0}(a)=\ell_{i}\right)\ \theta da.
\end{align}
\end{prop}
\begin{proof}
Recalling that $N^{(t)}_{t-a}$ refers to the size whole population in the lower tree $\hat{\mathcal{P}}$ of the construction of Proposition \ref{lem: construction}, we similarly define $Z_{0}^{(t)}(a)$ as the size of the clonal population in the same tree (with the convention that mutations that occur at time $t-a$, i.e.\ on the leaves of the tree $\hat{\mathcal{P}}$, do not affect $Z_{0}^{(t)}(a)$).
 It follows that
\begin{equation}
\label{eq:momzzero}
A(k,t)\mathds{1}_{Z_{0}(t)=\l}=\int_{[0,t]\times \mathbb{N}}\frac{\mathds{1}_{Z^{j}_{0}(a)=k}}{\left(Z_{0}^{(t)}(a)-B_{j}\right)!}\sum_{\sigma\in I}\mathds{1}_{\sigma\text{ is ancestral}}\sum_{\l_{1}+\dots+\l_{Z^{(t)}_{0}(a)-B_{j}}=\l}\prod_{i=1}^{Z^{(t)}_{0}(a)-B_{j}}\mathds{1}_{Z^{\sigma_{i}}_{0}(a)=\l_{i}}\ \mathcal{N}(da,dj),\\
\end{equation}
where $I$ is the set of injections from $\left\{1,\dots,Z_{0}^{(t)}(a)-B_{j}\right\}$ to $\left\{1,\dots,N_{t-a}^{(t)}\right\}$, $B_{j}$ is the indicator function of the event
\[
\left\{\text{the }j\text{th individual at time $t-a$ is clonal} \right\},
\]
and \textcolor{red}{"$\sigma$ is ancestral" denotes the event that the individuals} $\sigma_{1},\dots,\sigma_{Z_{0}^{(t)}(a)-B_{j}}$ at time $t-a$ have the ancestral type.
Now, using the same method as in the proof of Theorem \ref{thm: simple factorial moment} leads to
\begin{align*}
\mathbb{E}_{t}&\left[A(k,t)\mathds{1}_{Z_{0}(t)=\l}\right]\\&=\mathbb{E}_{t}\int_{[0,t]\times \mathbb{N}}\ \mathbb{P}_{a}\left(Z_{0}(a)=k\right)\sum_{\sigma\in I}\mathds{1}_{\sigma\text{ is ancestral}}\sum_{\l_{1}+\dots+\l_{Z^{(t)}_{0}(a)-B_{j}}=\l}\prod_{i=1}^{Z^{(t)}_{0}(a)-B_{j}} \mathbb{P}_{a}\left(Z_{0}(a)=\l_{i}\right)\ \frac{\mathcal{N}(da,dj)}{\left(Z_{0}^{(t)}(a)-B_{j}\right)!}\\
&=\mathbb{E}_{t}\int_{[0,t]\times \mathbb{N}}\ \mathbb{P}_{a}\left(Z_{0}(a)=k\right)\sum_{\l_{1}+\dots+\l_{Z^{(t)}_{0}(a)-B_{j}}=\l}\prod_{i=1}^{Z^{(t)}_{0}(a)-B_{j}} \mathbb{P}_{a}\left(Z_{0}(a)=\l_{i}\right)\ \mathcal{N}(da,dj)\\
&=\mathbb{E}_{t}\int_{[0,t]\times \mathbb{N}}\ \mathbb{P}_{a}\left(Z_{0}(a)=k\right)\sum_{\l_{1}+\dots+\l_{Z^{(t)}_{0}(a)-B_{j}}=\l}\prod_{i=1}^{Z^{(t)}_{0}(a)-B_{j}} \mathbb{P}_{a}\left(Z_{0}(a)=\l_{i}\right)\ \mathds{1}_{H_{j}>t-a}\mathds{1}_{j<N_{t}}\mathcal{P}(da,dj).\\
\end{align*}
Now, $Z_{0}^{(t)}(a)$ is not independent from $\mathcal{P}$, but we have that $Z_{0}^{(t)}(a)$ is independent from $\mathcal{P}\left([a,T]\cap\cdot \right)$ for all $a<T$. Hence, Theorem \ref{lem: rmexpc} applies to $\widetilde{X}_{a}:=Z_{0}^{(t)}(t-a)$ and $\widetilde{\mathcal{P}}$ defined for all measurable set $A\subset[0,t]$ by
\[
\widetilde{\mathcal{P}}\left(A \right)=\mathcal{P}\left(t-A\right),
\]
and, as in \eqref{eq:petiteeq},
\begin{multline*}
\mathbb{E}_{t}\left[A(k,t)\mathds{1}_{Z_{0}(t)=\l}\right]\\=\mathbb{E}_{t}\int_{[0,t]\times \mathbb{N}}\ \mathbb{P}_{a}\left(Z_{0}(a)=k\right)\sum_{\l_{1}+\dots+\l_{Z^{(t)}_{0}(a)-B_{j}}=\l}\prod_{i=1}^{Z^{(t)}_{0}(a)-B_{j}} \mathbb{P}_{a}\left(Z_{0}(a)=\l_{i}\right)\ \mathds{1}_{H_{j}>t-a}\mathds{1}_{j<N_{t}} \theta da\ C(dj).
\end{multline*}

Finally, integrating with respect to $C(dj)$ leads to the result.
\end{proof}
\textcolor{red}{This last proposition in not exactly a closed formula since its involves the law of the couple $(N^{(t)}_{t-a},Z^{(t)}_{0}(a))$.}
\textcolor{red}{To close the formula}, we need an explicit formula for the joint generating function of $N^{(t)}_{t-a}$ and $Z^{(t)}_{0}(a)$. Let
\[
F(u,v)=\mathbb{E}_{t}\left[u^{N^{(t)}_{t-a}}v^{Z^{(t)}_{0}(a)} \right],\ u,v\in[0,1],
\]
which is given, thanks to Proposition 4.1 of \cite{CL1}, by
\begin{equation}
\label{eq:oneEq}
F(u,v)=u\frac{\hat{W}(t-a,u)}{\hat{W}(t-a)}\left(1-\frac{e^{-\theta(t-a)}\hat{W}(t-a,u)}{\frac{v}{1-v}+\hat{W}_{\theta}(t-a,u)}\right),
\end{equation}
where $\hat{W}$ is the scale function of the lower CPP, $\hat{\mathcal{P}}$, defined in Proposition \ref{lem: construction},
\[
\hat{W}(t,u):=\frac{\hat{W}(t)}{\hat{W}(t)-u\left(\hat{W}(t)-1\right)},
\]
and
\[
\hat{W}_{\theta}(t,u):=e^{-\theta t}\hat{W}(t,u)+\theta\int_{0}^{t}\hat{W}(s,u)e^{-\theta s}\ ds.
\]

\begin{prop}For all $k\geq1$ and $l\geq0$,
\begin{align*}
&\mathbb{E}_{t}\left[A(k,t)\mathds{1}_{Z_{0}(t)=\l}\right]\\=&\int_{0}^{t}\mathbb{P}_{a}\left(Z_{0}(a)=k \right)\sum_{j=1}^{ l}\dbinom{l-1}{j-1}\frac{1}{j!}\left(1-\frac{1}{W_{\theta}(a)}\right)^{l-j}\left(\frac{e^{-\theta a}W(a)}{W_{\theta}(a)^{2}\mathbb{P}\left(Z_{0}(a)=0\right)} \right)^{j}
 H_{j}\left(1,1-\frac{e^{-\theta a}W(a)}{W_{\theta}(a)} \right)\ \theta da\\
&+\int_{0}^{t}\mathbb{P}_{a}\left(Z_{0}(a)=k \right)\sum_{j=1}^{ l}\dbinom{l-1}{j-1}\frac{1}{j!}\left(1-\frac{1}{W_{\theta}(a)}\right)^{l-j}\left(\frac{e^{-\theta a}W(a)}{W_{\theta}(a)^{2}\mathbb{P}\left(Z_{0}(a)=0\right)} \right)^{j}
 G_{j}\left(1-\frac{e^{-\theta a}W(a)}{W_{\theta}(a)} \right) \theta da,
\end{align*}
where
\[
H_{j}(u,v):=v^{j}\partial^{j}_{v}\partial_{u}u F(u,v)-v^{j+1}\partial^{j+1}_{v}\left\{v \mathbb{E}\left[v^{Z_{0}^{(t)}(a)} \right]\right\},
\]
and
\[
G_{j}:=v^{j-1}\partial^{j}_{v}\mathbb{E}_{t}\left[v^{Z_{0}^{(t)}(a)}\right].
\]
\end{prop}
\begin{proof}
Let $A_{1}$ and $A_{2}$ denote the two terms of the r.h.s.\ of \eqref{eq:indicat}. We detail the computations of $A_{1}$. The case $A_{2}$ is similar.
\begin{multline*}
A_{1}=\mathbb{E}_{t}\int_{0}^{t}\left(N^{(t)}_{t-a}-Z^{(t)}_{0}(a)\right)\mathbb{P}_{a}\left(Z_{0}(a)=k \right)\\\times\sum_{j=1}^{Z_{0}^{(t)}(a)\wedge l}\dbinom{Z_{0}^{(t)}(a)}{j}\sum_{\underset{\l_{j}>0}{\ell_{1}+\dots+\ell_{j}=\l}}\prod_{i=1}^{j}\mathbb{P}_{a}\left(Z_{0}(a)=\ell_{i}\right)\mathbb{P}_{a}\left(Z_{0}(a)=0\right)^{Z_{0}(a)-j}\ \theta da.
\end{multline*}
%This last expression can be rewritten choosing which $l_{i}$s are positive, that is 
%\begin{multline*}
%\mathbb{E}A(k,t)\mathds{1}_{Z_{0}(t)=l}=\mathbb{E}\int_{0}^{t}\left(N^{(t)}_{t-a}-Z^{(t)}_{0}(a)\right)\mathbb{P}\left(Z_{0}(a)=k \right)\sum_{j=1}^{Z^{(t)}_{0}(a)\wedge l}\dbinom{Z^{(t)}_{0}(a)}{j}\sum_{\ell_{1}+\dots+\ell_{j}=l}\prod_{i=1}^{j}\mathbb{P}\left(Z_{0}(a)=\ell_{i}\right)\ \theta da\\
%+\mathbb{E}\int_{0}^{t}Z_{0}^{(t)}(a)\sum_{j=1}^{\left(Z^{(t)}_{0}(a)-1\right)\wedge l}\dbinom{Z^{(t)}_{0}(a)-1}{j}\sum_{\ell_{1}+\dots+\ell_{j}=l}\prod_{i=1}^{j}\mathbb{P}\left(Z_{0}(a)=\ell_{i}\right)\ \theta da.\\
%\end{multline*}
Since, from  \eqref{eq: loizzero},
\[
\prod_{i=1}^{j}\mathbb{P}_{a}\left(Z_{0}(a)=\ell_{i}\right)=\prod_{i=1}^{j}\frac{e^{-\theta a}W(a)}{W_{\theta}(a)^{2}}\left(1-\frac{1}{W_{\theta}(a)}\right)^{\l_{i}-1}=\left(\frac{e^{-\theta a}W(a)}{W_{\theta}(a)^{2}} \right)^{j}\left(1-\frac{1}{W_{\theta}(a)}\right)^{l-j},
\]
we get
\begin{align*}
A_{1}=&\mathbb{E}_{t}\int_{0}^{t}\left(N^{(t)}_{t-a}-Z^{(t)}_{0}(a)\right)\mathbb{P}_{a}\left(Z_{0}(a)=k \right)\sum_{j=1}^{Z_{0}^{(t)}(a)\wedge l}\dbinom{Z_{0}^{(t)}(a)}{j}\dbinom{l-1}{j-1}\\&\times\left(\frac{e^{-\theta a}W(a)}{W_{\theta}(a)^{2}} \right)^{j}\left(1-\frac{1}{W_{\theta}(a)}\right)^{l-j}\mathbb{P}_{a}\left(Z_{0}(a)=0\right)^{Z_{0}(a)-j}\ \theta da\\
=&\int_{0}^{t}\mathbb{P}_{a}\left(Z_{0}(a)=k \right)\sum_{j=1}^{ l}\dbinom{l-1}{j-1}\frac{1}{j!}\left(1-\frac{1}{W_{\theta}(a)}\right)^{l-j}\left(\frac{e^{-\theta a}W(a)}{W_{\theta}(a)^{2}\mathbb{P}_{a}\left(Z_{0}(a)=0\right)} \right)^{j}\\
&\times\mathbb{E}_{t}\left[\left(N^{(t)}_{t-a}-Z^{(t)}_{0}(a)\right)\left(Z_{0}^{(t)}(a)\right)_{(j)}\mathbb{P}_{a}\left(Z_{0}(a)=0\right)^{Z^{(t)}_{0}(a)}\right]\ \theta da
\end{align*}
%\begin{multline*}
%\mathbb{E}_{t}A(k,t)\mathds{1}_{Z_{0}(t)=l}\\=\mathbb{E}_{t}\int_{0}^{t}\left(N^{(t)}_{t-a}-Z^{(t)}_{0}(a)\right)\mathbb{P}_{a}\left(Z_{0}(a)=k \right)\dbinom{l+Z_{0}^{(t)}(a)-1}{l}\left(\frac{e^{-\theta a}W(a)}{W_{\theta}(a)^{2}} \right)^{Z_{0}^{(t)}(a)}\left(1-\frac{1}{W_{\theta}(a)}\right)^{l-Z_{0}^{(t)}(a)}\ \theta da\\
%+\mathbb{E}_{t}\int_{0}^{t}Z_{0}^{(t)}(a)\mathbb{P}_{a}\left(Z_{0}(a)=k \right)\dbinom{l+Z_{0}^{(t)}(a)-2}{l}\left(\frac{e^{-\theta a}W(a)}{W_{\theta}(a)^{2}} \right)^{Z_{0}^{(t)}(a)-1}\left(1-\frac{1}{W_{\theta}(a)}\right)^{l-Z_{0}^{(t)}(a)-1}\ \theta da.\\
%\end{multline*}
Finally, if we define, for all integer $j$,
\[
H_{j}(u,v):=v^{j}\partial^{j}_{v}\partial_{u}u F(u,v)-v^{j+1}\partial^{j+1}_{v}\left\{v \mathbb{E}\left[v^{Z_{0}^{(t)}(a)} \right]\right\},
\]
and
\[
G_{j}:=v^{j-1}\partial^{j}_{v}\mathbb{E}_{t}\left[v^{Z_{0}^{(t)}(a)}\right],
\]
we get
\begin{align*}
&\mathbb{E}_{t}\left[A(k,t)\mathds{1}_{Z_{0}(t)=\l}\right]\\=&\int_{0}^{t}\mathbb{P}_{a}\left(Z_{0}(a)=k \right)\sum_{j=1}^{ l}\dbinom{l-1}{j-1}\frac{1}{j!}\left(1-\frac{1}{W_{\theta}(a)}\right)^{l-j}\left(\frac{e^{-\theta a}W(a)}{W_{\theta}(a)^{2}\mathbb{P}\left(Z_{0}(a)=0\right)} \right)^{j}
 H_{j}\left(1,1-\frac{e^{-\theta a}W(a)}{W_{\theta}(a)} \right)\ \theta da\\
&+\int_{0}^{t}\mathbb{P}_{a}\left(Z_{0}(a)=k \right)\sum_{j=1}^{ l}\dbinom{l-1}{j-1}\frac{1}{j!}\left(1-\frac{1}{W_{\theta}(a)}\right)^{l-j}\left(\frac{e^{-\theta a}W(a)}{W_{\theta}(a)^{2}\mathbb{P}\left(Z_{0}(a)=0\right)} \right)^{j}
 G_{j}\left(1-\frac{e^{-\theta a}W(a)}{W_{\theta}(a)} \right) \theta da.
\end{align*}
\end{proof}

These ideas also lead to the following formula, which is proved similarly.

\begin{cor}
\label{co:finalFormula}
Let $n_{1},\dots,n_{N}$ and $k_{1},\dots,k_{N}$ be positive integers. Let $\l$ be a positive integer. We have

\begin{align*}
&\mathbb{E}_{t}\left[\prod_{i=1}^{N}\dbinom{A(k_{i},t)}{n_{i}}\mathds{1}_{Z_{0}(t)=\l}\right]\\=&\sum_{\kappa=1}^{N} \mathbb{E}_{t}\int_{[0,t]} \left(N^{(t)}_{t-a}-Z_{0}^{(t)}(a) \right) \sum_{\underset{\l_{2}+\dots+\l_{Z^{(t)}_{0}(a)+1}=\l}{n^{1:N}_{1}+\dots+n^{1:N}_{N^{(t)}_{t-a}}=n_{1:N}-\delta_{1:N,l}}} \prod_{m=Z^{(t)}_{0}(a)+2}^{N^{(t)}_{t-a}}\mathbb{E}_{a}\left[\prod_{i=1}^{N}\dbinom{A(k_{i},a)}{n^{i}_{m}}\right]\\&\times\prod_{{m=2}}^{Z^{(t)}_{0}(a)+1} \mathbb{E}_{a}\left[\prod_{i=1}^{N}\dbinom{A(k_{i},a)}{n^{i}_{m}}\mathds{1}_{Z_{0}(a)=\l_{m}}\right]\mathbb{E}_{a}\left[\prod_{i=1}^{N}\dbinom{A(k_{i},a)}{n_{1}^{i}}\mathds{1}_{Z_{0}(a)=k_{\kappa}} \right]\ \theta da\\
&+\sum_{\kappa=1}^{N} \mathbb{E}_{t}\int_{[0,t]}Z_{0}^{(t)}(a)  \sum_{\underset{\l_{2}+\dots+\l_{Z^{(t)}_{0}(a)+1}=\l}{n^{1:N}_{1}+\dots+n^{1:N}_{N^{(t)}_{t-a}}=n_{1:N}-\delta_{1:N,l}}} \prod_{m=Z^{(t)}_{0}(a)+1}^{N^{(t)}_{t-a}}\mathbb{E}_{a}\left[\prod_{i=1}^{N}\dbinom{A(k_{i},a)}{n^{i}_{m}}\right]\\&\times\prod_{{m=2}}^{Z^{(t)}_{0}(a)} \mathbb{E}_{a}\left[\prod_{i=1}^{N}\dbinom{A(k_{i},a)}{n^{i}_{m}}\mathds{1}_{Z_{0}(a)=\l_{m_{2}}}\right]\mathbb{E}_{a}\left[\prod_{i=1}^{N}\dbinom{A(k_{i},a)}{n_{1}^{i}}\mathds{1}_{Z_{0}(a)=k_{\kappa}} \right]\ \theta da.\\
\end{align*}
\end{cor}

\begin{proof}
According to Section \ref{ssec:proofs}, we have the following integral representation.

\begin{eqnarray*}
\prod_{i=1}^{N}\dbinom{A(k_{i},t)}{n_{i}}=\sum_{l=1}^{N} \int_{[0,t]\times\mathbb{N}} \mathds{1}_{Z^{j}_{0}(a)=k_{l}} \sum_{n^{1:N}_{1}+\dots+n^{1:N}_{N^{(t)}_{t-a}}=n_{1:N}-\delta_{1:N,l}} \prod_{m=1}^{N^{(t)}_{t-a}}\prod_{i=1}^{N}\dbinom{A(k_{i},a)}{n_{m}^{j}}\ \mathcal{N}(da,dj).
\end{eqnarray*}
Now, using this equation in conjunction with the decomposition of $\mathds{1}_{Z_{0}(t)=\l}$ used in Section \ref{ssec:jointzzero}, we have
%\begin{eqnarray*}
%\mathds{1}_{Z_{0}(t)=l}=\sum_{\sigma\in I}\frac{\mathds{1}_{\sigma\text{ is ancestral}}}{\left(Z_{0}^{(t)}(a)-B_{j}\right)!}\sum_{l_{1}+\dots+l_{Z^{(t)}_{0}(a)}=l}\prod_{i=1}^{Z^{(t)}_{0}(a)-B_{j}}\mathds{1}_{Z^{\sigma_{i}}_{0}(a)=l_{i}}, \quad \forall a\in[0,t].
%\end{eqnarray*}

\begin{multline*}
\prod_{i=1}^{N}\dbinom{A(k_{i},t)}{n_{i}}\mathds{1}_{Z_{0}(t)=\l}=\sum_{l=1}^{N} \int_{[0,t]\times\mathbb{N}} \mathds{1}_{Z^{j}_{0}(a)=k_{l}}\sum_{\sigma\in I}\mathds{1}_{\sigma\text{ is ancestral}} \\\times\sum_{\underset{\l_{1}+\dots+\l_{Z^{(t)}_{0}(a)-B_{j}}=\l}{n^{1:N}_{1}+\dots+n^{1:N}_{N^{(t)}_{t-a}}=n_{1:N}-\delta_{1:N,l}}} \prod_{m_{1}=1}^{N^{(t)}_{t-a}}\prod_{i=1}^{N}\dbinom{A^{m_{1}}(k_{i},a)}{n^{i}_{m_{1}}}\prod_{m_{2}=1}^{Z^{(t)}_{0}(a)-B_{j}} \mathds{1}_{Z^{\sigma_{m_{2}}}_{0}(a)=\l_{m_{2}}}\ \frac{\mathcal{N}(da,dj)}{\left(Z_{0}^{(t)}(a)-B_{j}\right)!}.\ 
\end{multline*}
We refer the reader to the proof of Proposition \ref{prop:oneProp} for the definitions of $I$,$B_{j}$, and the event $\{\sigma\text{ is ancestral}\}$.
The definitions of $A^{(m)}(k,a)$ and $Z^{(m)}_{0}(a)$ can be found in the beginning of this section.

\medskip

Now, we take the expectation in the last equality. Thanks to the method used in the proof of Proposition \ref{prop:oneProp}, we have
\begin{align*}
&\mathbb{E}_{t}\left[\prod_{i=1}^{N}\dbinom{A(k_{i},t)}{n_{i}}\mathds{1}_{Z_{0}(t)=\l}\right]\\&=\sum_{\kappa=1}^{N} \mathbb{E}_{t}\Bigg\{\int_{[0,t]\times\mathbb{N}} \sum_{\sigma\in I}\mathds{1}_{\sigma\text{ is ancestral}}\sum_{\underset{\l_{1}+\dots+\l_{Z^{(t)}_{0}(a)-B_{j}}=\l}{n^{1:N}_{1}+\dots+n^{1:N}_{N^{(t)}_{t-a}}=n_{1:N}-\delta_{1:N,l}}} \prod_{\underset{m_{1}\neq\sigma,m_{1}\neq i}{m_{1}=1}}^{N^{(t)}_{t-a}}\mathbb{E}_{a}\left[\prod_{i=1}^{N}\dbinom{A^{m_{1}}(k_{i},a)}{n^{i}_{m_{1}}}\right]\\&\times\prod_{{m_{2}=1}}^{Z^{(t)}_{0}(a)-B_{j}} \mathbb{E}_{a}\left[\prod_{i=1}^{N}\dbinom{A^{\sigma_{m_{2}}}(k_{i},a)}{n^{i}_{\sigma_{m_{2}}}}\mathds{1}_{Z^{\sigma_{m_{2}}}_{0}(a)=\l_{m_{2}}}\right]\mathbb{E}_{a}\left[\prod_{i=1}^{N}\dbinom{A^{i}(k_{j},a)}{n_{i}^{j}}\mathds{1}_{Z^{j}_{0}(a)=k_{\kappa}} \right]\ \frac{\mathcal{N}(da,dj)}{\left(Z_{0}^{(t)}(a)-B_{j}\right)!}\Bigg\},\\
\end{align*}
where $m_{1}\neq\sigma$ means that $m_{1}\notin\sigma\left(\left\{1,\dots,Z_{0}^{(t)(a)}-B_{j} \right\}\right)$.
Now, following, as above, we get
\begin{align*}
&\mathbb{E}_{t}\left[\prod_{i=1}^{N}\dbinom{A(k_{i},t)}{n_{i}}\mathds{1}_{Z_{0}(t)=\l}\right]\\&=\sum_{\kappa=1}^{N} \mathbb{E}_{t}\Bigg\{\int_{[0,t]\times\mathbb{N}} \sum_{\sigma\in I}\mathds{1}_{\sigma\text{ is ancestral}}\sum_{\underset{\l_{1}+\dots+\l_{Z^{(t)}_{0}(a)-B_{j}}=\l}{n^{1:N}_{1}+\dots+n^{1:N}_{N^{(t)}_{t-a}}=n_{1:N}-\delta_{1:N,l}}} \prod_{m_{1}=Z^{(t)}_{0}(a)-B_{i}+1}^{N^{(t)}_{t-a}}\mathbb{E}_{a}\left[\prod_{i=1}^{N}\dbinom{A(k_{i},a)}{n^{i}_{m_{1}}}\right]\\&\times\prod_{{m_{2}=2}}^{Z^{(t)}_{0}(a)-B_{i}+1} \mathbb{E}_{a}\left[\prod_{i=1}^{N}\dbinom{A(k_{i},a)}{n^{i}_{m_{2}}}\mathds{1}_{Z_{0}(a)=\l_{m_{2}}}\right]\mathbb{E}_{a}\left[\prod_{i=1}^{N}\dbinom{A(k_{1},a)}{n_{i}^{1}}\mathds{1}_{Z_{0}(a)=k_{\kappa}} \right]\ \frac{\mathds{1}_{H_{i}>t-a}\mathds{1}_{j<N_{t}}\theta da\ C(di)}{\left(Z_{0}^{(t)}(a)-B_{j}\right)!}\Bigg\}.\\
\end{align*}
Then, the sum with $\sigma$ can be removed since there is no term depending on $\sigma$. Finally, integrating with respect to $C(di)$ leads to the result.
\end{proof}
\color{black}
 
Together with Theorems \ref{thm: simple factorial moment} and \ref{thm: multiple factorial moment} and using the joint law of $N^{(t)}_{t-a}$ and $Z_{0}^{(t)}(a)$ given in \eqref{eq:oneEq}, these formulas give explicit recursion to compute each factorial moment of the frequency spectrum. 
\begin{rem}
Although, these formulas are quite heavy, an important interest lies in the method used to compute them. Indeed, this method should work to obtain the joint moments of $A(k,t)$ with any quantity which can be expressed, at any time $a$, as the sum of contributions of each subtrees. For instance, since
\[
N_{t}=\sum_{i=1}^{N^{(t)}_{t-a}}N^{i}_{a},\quad \forall a\in[0,t],
\]
where $N^{i}_{a}$ is the number of individuals of the $i$-th subtrees at time $a$, we are able to compute the joint moments of $N_{t}$ and $\left(A(k,t)\right)_{k\geq1}$. For example, 
using the integral representation \eqref{eq:aktrep} of $A(k,t)$ and following the proof of Theorem \ref{thm : expc} , we have that
%\begin{align*}
%\mathbb{E}_{t}\left[A(k,t)|N_{t}=n\right]&=\int_{0}^{t}\theta\ \mathbb{E}\left[ \sum_{i=0}^{n-1}\mathds{1}_{\tilde{H}_{i}>t-a}\right]\mathbb{P}_{a}\left(Z_{0}(a)=k\right)da,\\
%&=\int_{0}^{t}\theta \left(1+(n-1)\left(\frac{W(t)}{W(a)}-1\right)\left(W(t)-1 \right)^{-1} \right)\mathbb{P}\left(Z_{0}(a)=k \right)da.
%\end{align*}
%It follows,
%\begin{align*}
%\mathbb{E}_{t}\left[A(k,t)N_{t}\right]&=\int_{0}^{t}\theta \left(2\frac{W(t)^{2}}{W(a)}-W(t) \right)\mathbb{P}_{a}\left(Z_{0}(a)=k \right)da,\\
%&=2W(t)^{2}\int_{0}^{t}\frac{\theta e^{-\theta s}}{W_{\theta}(s)^{2}}\left(1-\frac{1}{W_{\theta}(s)} \right)^{k-1}da+\mathcal{O}\left(W(t)\right).
%\end{align*}
\begin{align}
\label{eq:momNtak}
&\mathbb{E}_{t}\left[A(k,t)N_{t}\right]\notag\\&=\mathbb{E}_{t}\int_{[0,t]\times \mathbb{N}}\sum_{j=1}^{N^{(t)}_{t-a}}N^{j}_{a}\mathds{1}_{Z_{0}^{(i)}(a)=k}\mathcal{N}\left(da,di\right)\notag\\
%&=\mathbb{E}\int_{[0,t]\times \mathbb{N}}\sum_{\underset{i\neq j}{j=1}}^{N^{(t)}_{t-a}}N^{j}_{a}\mathds{1}_{Z_{0}^{(i)}(a)=k}\mathcal{N}\left(da,di\right)\\
%+&\mathbb{E}\int_{[0,t]\times \mathbb{N}}N^{i}_{a}\mathds{1}_{Z_{0}^{(i)}(a)=k}\mathcal{N}\left(da,di\right)\\
&=\int_{[0,t]}\theta\mathbb{E}_{t}\left[N^{(t)}_{t-a}\left(N^{(t)}_{t-a}-1 \right)\right]\mathbb{E}_{a}\left[N_{a}\right]\mathbb{P}_{a}\left(Z_{0}(a)=k\right)da
+\int_{[0,t]}\theta\mathbb{E}_{t}\left[N^{(t)}_{t-a}\right]\mathbb{E}_{a}\left[N_{a}\mathds{1}_{Z_{0}(a)=k}\right]\theta da\notag\\
&=\int_{[0,t]}W(t)^{2}\left(1-\frac{W(a)}{W(t)}\right)\frac{\theta e^{-\theta a}}{W_{\theta}(a)^{2}}\left(1-\frac{1}{W_{\theta}(a)} \right)^{k-1}da
+W(t)\int_{[0,t]}\theta\frac{\mathbb{E}_{a}\left[N_{a}\mathds{1}_{Z_{0}(a)=k}\right]}{W(a)}\theta da.
\end{align}
\end{rem}

\subsection{Application to the computation of the covariances of the frequency spectrum}
\label{ssec:convariance}
\color{black}
A quantity of particular interest is the limit covariance between two terms of the frequency spectrum,
\begin{prop}
\label{lem: cov}
Suppose that $\alpha>0$. Let $k$ and $l$ two positive integers, then,

\[
\cov_{t} \left(A\left(k,t\right),A\left(l,t\right)\right)=W(t)^{2}c_{k}c_{l}+\textcolor{red}{o\left(W(t)^{2} \right)},
\]
where 
\[
c_{k}:=\int_{0}^{\infty}\frac{\theta e^{-\theta s}}{W_{\theta}(s)^{2}}\left(1-\frac{1}{W_{\theta}(s)} \right)^{k-1}ds, \ \forall k\in\textcolor{red}{\mathbb{N}\backslash\{0\}}.
\]

\end{prop}
\begin{proof}
In order to show how quantities in Theorem \ref{thm: multiple factorial moment} can be manipulated, we detail the proof.

Using Theorem \ref{thm: multiple factorial moment}, we obtain
\begin{multline*}
\mathbb{E}_{t}\left[A(k,t)A(l,t)\right]
%=\mathbb{E}\int_{0}^{t}\theta N^{(t)}_{t-a}&\mathbb{P}\left(Z_{0}(a)=k\right)\sum_{i=1}^{N^{(t)}_{t-a}-1}\mathbb{E}\left[A^{i}(l,a)\right]da,\\&+\mathbb{E}\int_{0}^{t}\theta N^{(t)}_{t-a}\mathbb{P}\left(Z_{0}(a)=l\right)\sum_{i=1}^{N^{(t)}_{t-a}-1}\mathbb{E}\left[A^{i}(k,a)\right]da,\\
%+\mathbb{E}\int_{0}^{t}\theta N^{(t)}_{t-a}&\mathbb{E}\left[A^{i}(l,a)\mathds{1}_{Z_{0}(a)=k}\right]da,\\
%&+\mathbb{E}\int_{0}^{t}\theta N^{(t)}_{t-a}\mathbb{E}\left[A^{i}(k,a)\mathds{1}_{Z_{0}(a)=l}\right]da.\\
=\int_{0}^{t}\theta\mathbb{E}_{t}\left[N^{(t)}_{t-a}\left(N^{(t)}_{t-a}-1\right) \right]\left(\mathbb{P}_{a}\left(Z_{0}(a)=k\right)\mathbb{E}_{a}\left[A(l,a)\right]+\mathbb{P}_{a}\left(Z_{0}(a)=\l\right)\mathbb{E}_{a}\left[A(k,a)\right]\right)da\\
+\int_{0}^{t}\theta \mathbb{E}_{t}N^{(t)}_{t-a}\left(\mathbb{E}_{a}\left[A(l,a)\mathds{1}_{Z_{0}(a)=k}\right]+\mathbb{E}_{a}\left[A(k,a)\mathds{1}_{Z_{0}(a)=\l}\right]\right)da.
\end{multline*}
Recalling, from Proposition \ref{lem: construction}, that $N^{(t)}_{t-a}$ is geometrically distributed with parameter $\frac{W(a)}{W(t)}$ under $\mathbb{P}_{t}$, 
\[
\mathbb{E}_{t}\left[N^{(t)}_{t-a}\right]=\frac{W(t)}{W(a)} \quad \text{and} \quad \mathbb{E}_{t}\left[N^{(t)}_{t-a}\left(N^{(t)}_{t-a}-1\right) \right]=2\frac{W(t)^{2}}{W(a)^{2}}\left(1-\frac{W(a)}{W(t)} \right).
\]

Since
\[
\mathbb{E}\left[A(k,a)\mathds{1}_{Z_{0}(a)=\l}\right]\leq \mathbb{E}\left[A(k,a)\right]=\mathcal{O}(W(a)),
\]
it follows by Lemma \ref{lem: asyComp} and Theorem \ref{thm : expc},
that
\begin{align*}
\mathbb{E}_{t}\left[A(k,t)A(l,t)\right]=2\int_{0}^{t}\theta\frac{W(t)^{2}}{W(a)^{2}}\Big\{\mathbb{P}_{a}(Z_{0}(a)=\l)\mathbb{E}_{a}\left[A(k,a) \right]+\mathbb{P}_{a}(Z_{0}(a)=k)\mathbb{E}_{a}\left[A(l,a) \right] \Big\}da+\mathcal{O}\left(tW(t)\right).
\end{align*}
By Theorem \ref{thm : expc} and \eqref{eq: loizzero}, the r.h.s. is equal to
\begin{align*}
2&W(t)^{2}\int_{0}^{t}\frac{\theta e^{-\theta a}}{W_{\theta}(a)^{2}}\Bigg(\left(1-\frac{1}{W_{\theta}(a)}\right)^{k-1}\int_{0}^{a}\frac{\theta e^{-\theta s}}{W_{\theta}(s)^{2}}\left(1-\frac{1}{W_{\theta}(s)}\right)^{l-1}ds\\&+\left(1-\frac{1}{W_{\theta}(a)}\right)^{l-1}\int_{0}^{a}\frac{\theta e^{-\theta s}}{W_{\theta}(s)^{2}}\left(1-\frac{1}{W_{\theta}(s)}\right)^{k-1}ds\Bigg) da+\mathcal{O}\left(tW(t)\right)\\
=&2W(t)^{2}\int_{0}^{t}\frac{\theta e^{-\theta a}}{W_{\theta}(a)^{2}}\left(1-\frac{1}{W_{\theta}(a)} \right)^{k-1}da\int_{0}^{t}\frac{\theta e^{-\theta a}}{W_{\theta}(a)^{2}}\left(1-\frac{1}{W_{\theta}(a)} \right)^{l-1}da+\mathcal{O}\left(tW(t)\right).\qedhere
\end{align*}
\textcolor{red}{The last equality follows from the identity
\[
\int_{0}^{t}g(s)ds\int_{0}^{t}f(s)ds=\int_{0}^{t}f(a)\int_{0}^{a}g(s)dsda+\int_{0}^{t}g(a)\int_{0}^{a}f(s)dsda.
\]
The proof ends thanks to Remark \ref{rem:reviewed}.}
\end{proof}
\section{Asymptotic behaviour of the moments of the frequency spectrum}
\label{sec:asypmom}
In this part, we study the long time behaviour of the moments of the frequency spectrum. From this point and until the end of this work, we suppose that the tree is supercritical, that is $\alpha >0$.

\begin{prop}
\label{lem: asymptMoment}
For any positive multi-integers $n$ and $k$ in $\mathbb{N}^{N}$, 
\begin{equation}
\label{eq:expo}
\mathbb{E}_{t}\left[\prod_{i=1}^{N}\dbinom{A\left(k_{i},t\right)}{n_{i}}\right]=\frac{W(t)^{|n|}|n|!}{\prod_{i=1}^{N}n_{i}!}\prod_{i=1}^{N}c_{k_{i}}^{n_{i}}+\mathcal{O}\left(tW(t)^{|n|-1}\right),
\end{equation}
where the $c_{k_{i}}$'s are as defined in Proposition \ref{lem: cov}.
\end{prop}

\begin{proof}

{\bf Step 1: Preliminaries and ideas.}

The proposition is proved by induction.
% From Theorem \ref{thm: multiple factorial moment} we have
%\begin{multline}
%\label{eq:expJoint}
%\mathbb{E}_{t}\left[\prod_{i=1}^{N}\dbinom{A(k_{i},t)}{n_{i}} \right]\\=\sum_{l=1}^{N}\mathbb{E}_{t}\int_{0}^{t}\theta N^{(t)}_{t-a}\sum_{n^{1:N}_{1}+\dots+n^{1:N}_{N^{(t)}_{t-a}}=n_{1:N}-\delta_{1:N,l}}\mathbb{E}_{a}\left[\prod_{i=1}^{N}\dbinom{A(k_{i},a)}{n^{i}_{1}}\mathds{1}_{Z_{0}(a)=k_{l}}\right]\prod_{m=2}^{N^{(t)}_{t-a}}\mathbb{E}_{a}\left[\prod_{i=1}^{N}\dbinom{A(k_{i},a)}{n_{m}^{i}}\right]da.
%\end{multline}

Using the symmetry of the formula provided by Theorem \ref{thm: multiple factorial moment}, we may restrict to the study of the term $l=1$ in \eqref{eq: momRep}. Hence, we want to study
%\begin{align*}
%AJOUTER E
%!!!
%!!!
\begin{equation}
\label{eq: jsaispas}
\mathbb{E}_{t}\int_{0}^{t}\theta N^{(t)}_{t-a}\sum_{n^{1:N}_{1}+\dots+n^{1:N}_{N^{(t)}_{t-a}}=n_{1:N}-\delta_{1:N,1}}\mathbb{E}_{a}\left[\prod_{i=1}^{N}\dbinom{A(k_{i},t)}{n_{1}^{i}}\mathds{1}_{Z_{0}(a)=k_{1}}\right]\prod_{m=2}^{N^{(t)}_{t-a}-1}\mathbb{E}_{a}\left[\prod_{i=1}^{N}\dbinom{A(k_{i},a)}{n_{m}^{i}}\right]da.\\
\end{equation}
We recall that the terms of the multi-sum in the above formula correspond to the ways of allocating the mutations in the subtrees. The analysis relies on the fact that the growth of each term depends on the repartition of the mutations. In particular, the main term correspond to the case where all mutations are allocated to different subtrees.

To capitalize on this fact, let $\mathcal{M}_{N^{(t)}_{t-a}}$ the subset of $\mathcal{M}_{\left(N^{t}_{t-a}-1\right)\times N}\left(\mathbb{N}\right)$ (the space of matrices of size $\left(N^{(t)}_{t-a}-1\right)\times N$ with coefficients in $\mathbb{N}$), such that each $\mathbf{n}$ in $\mathcal{M}_{N^{(t)}_{t-a}}$ satisfies the relation

 \begin{equation}
 \label{eq:condition}
 \sum_{m=1}^{N^{(t)}_{t-a}-1}n^{i}_{m}=n_{i}-\delta_{i,1}, \quad \forall i\in\mathbb{N}.\end{equation}

%Let $n^{1:N}_{1:N^{(t)}_{t-a}-1}$ be an element of $\mathcal{M}_{\left(N^{t}_{t-a}-1\right)\times N}\left(\mathbb{N}\right)$, 
%We denote by $\mathcal{M}_{N^{(t)}_{t-a}}$ this subset of $\mathcal{M}_{\left(N^{t}_{t-a}-1\right)\times N}\left(\mathbb{N}\right)$.
%
The notations $\mathbf{n}_{m}$ and $\mathbf{n}^{i}$ refer to the multi-integers $\left(n_{m}^{1},\dots,n_{m}^{N} \right)$ and $\left(n^{i}_{1},\dots,n_{N^{(t)}_{t-a}}^{i} \right)$ respectively.
 To simplify the analysis, we highlight three cases of interest:
\[
C_{1}:=\left\{\mathbf{n}\in \mathcal{M}_{N^{(t)}_{t-a}} \mid\forall i, n_{1}^{i}=0,\ \forall i\geq1,\forall m\geq 2, \ n_{m}^{i}\leq 1, \ \text{and} \left(n_{m}^{i}=1 \Rightarrow n_{m}^{k}=0, \forall k\neq i \right)  \right\}.
\]

This set corresponds to the case where all the mutations are taken in different subtrees and are not taken in the tree where a mutation just occurred. In fact, this corresponds to the dominant term of \eqref{eq: jsaispas} because  as $N^{(t)}_{t-a}$ tends to be large, the mutations tend to occur in different subtrees.
Let also
\[
C_{2}:=\left\{\mathbf{n}\in \mathcal{M}_{N^{(t)}_{t-a}} \mid\forall i, n_{1}^{i}=0\right\}\backslash C_{1}.
\]
Finally, let
\[
C_{3}:=\left\{\mathbf{n}\in \mathcal{M}_{N^{(t)}_{t-a}} \mid \sum_{i=1}^{N}n_{1}^{i}>0  \right\}.
\]

{\bf Step 2: Uniform bound on the number of tuple of mutations in the subtrees.}

Assuming that the relation of Lemma \ref{lem: asymptMoment} is true for any multi-integer $n^{\star}$ such that $|n^{\star}|=|n|-1$, we have
\begin{align}
\label{eq:eqeq}
\prod_{m=1}^{N^{(t)}_{t-a}}\mathbb{E}_{a}\left[\prod_{i=1}^{N}\dbinom{A(k_{i},a)}{n_{m}^{i}}\right]=\prod_{m=1}^{N^{(t)}_{t-a}}\left(\frac{W(a)^{|\mathbf{n}_{m}|}|\mathbf{n}_{m}|!}{\prod_{i=1}^{N}n_{m}^{i}!}\prod_{i=1}^{N}c_{k_{i}}^{n^{i}_{m}}+\mathcal{O}\left(aW(a)^{|\mathbf{n}_{m}|-1}\right)\right).\\
%&=W(a)^{|n|-1}\prod_{i=1}^{N}\left(\prod_{m=1}^{N^{(t)}_{t-a}-1}\frac{|\tilde{n}_{m}|!}{\prod_{j=1}^{N}n_{m}^{j}!}\right)\left(\int_{0}^{\infty}\frac{\theta e^{-\theta s}}{W_{\theta}(s)^{2}}\left(1-\frac{1}{W_{\theta}(s)} \right)^{k_{i}-1}ds\right)^{n_{i}}+\mathcal{E}(a),
\end{align}
%where $\mathcal{E}$ looks like,
%\[
%\mathcal{E}(a)=\sum_{i=1}^{N^{(t)}_{t-a}}\mathcal{C}_{i,N^{(t)}_{t-a}}\mathcal{O}\left(a^{i}W(a)^{|n|-i-1}\right),
%\]
%where $\mathcal{C}_{i,N^{(t)}_{t-a}}$ is the number of $i$-tuple of positive $\tilde{n}_{m}$.
%But, since $\sum_{i=1}^{N^{(t)}_{t-a}}\tilde{n}_{m}=n-1$, it follows that
%\[
%\mathcal{C}_{i,N^{(t)}_{t-a}}\leq\dbinom{n-1}{i},
%\]
%for all $i$. Hence,
%\[
%\mathcal{E}(a)=\mathcal{O}\left(aW(a)^{n-2} \right).
%\]
Since there are at most $|n|-1$ multi-integers $\mathbf{n}_{m}$ such that $|\mathbf{n}_{m}|>0$ (because of the condition \eqref{eq:condition}), we can assume without loss of generality, up to reordering the indices, that $n_{m}^{i}=0$, for all $m\geq|n|$, and so all the terms with $m>|n|$ in the product of \eqref{eq:eqeq} are equal to one. Hence,
\begin{equation}
\label{eq:prod}
\prod_{m=1}^{N^{(t)}_{t-a}-1}\mathbb{E}_{a}\left[\prod_{i=1}^{N}\dbinom{A(k_{i},a)}{n_{m}^{i}}\right]\leq \mathcal{C}_{\mathbf{n}} W(a)^{|n|-1},\\
\end{equation}
for some constant $ \mathcal{C}_{\mathbf{n}}$ depending only on the choice of $\mathbf{n}$ in $\mathcal{M}_{|n|}$.

Moreover, since $\mathcal{M}_{|n|}$ is finite, then

\begin{equation}
\prod_{m=1}^{N^{(t)}_{t-a}}\mathbb{E}_{a}\left[\prod_{i=1}^{N}\dbinom{A(k_{i},a)}{n_{m}^{i}}\right]\leq \mathcal{C}W(a)^{|n|-1}.\label{eq: majo}
\end{equation}

{\bf Step 3: Analysis of $C_{1}$.}

For $\mathbf{n}\in C_{1}$, and in this case only, the product
\[
\prod_{i=1}^{N}\dbinom{A(k_{i},a)}{n_{m}^{i}}
\]
has only one term different from $1$, and it follows from Theorem \ref{thm : expc}, that
\[
\prod_{m=1}^{N^{(t)}_{t-a}-1}\mathbb{E}\left[\prod_{i=1}^{N}\dbinom{A(k_{i},a)}{n_{m}^{i}}\right]=W(a)^{|n|-1}\prod_{i=1}^{N}\left(\int_{0}^{a}\frac{\theta e^{-\theta s}}{W_{\theta}(s)^{2}}\left(1-\frac{1}{W_{\theta}(s)} \right)^{k_{i}}ds\right)^{n_{i}-\delta_{i,1}}.
\]
The corresponding contribution in \eqref{eq: jsaispas} is
\[
I_{1}:=\int_{0}^{t}\theta W(a)^{|n|-1} \mathbb{P}_{a}\left(Z_{0}(t)=k_{1}\right)\prod_{i=1}^{N}\left(\int_{0}^{a}\frac{\theta e^{-\theta s}}{W_{\theta}(s)^{2}}\left(1-\frac{1}{W_{\theta}(s)} \right)^{k_{i}}ds\right)^{n_{i}-\delta_{i,1}} \mathbb{E}_{a}\left[N^{(t)}_{t-a}\text{Card}(C_{1})\right] \ da.
\]

Now, $\text{Card}(C_{1})$ is the number of ways we can choose $|n|-1$ subtrees among the $N^{(t)}_{t-a}-1$ possible subtrees and choosing a way to allocate to each chosen subtree a mutation sizes $k_{1},\dots,k_{N}$, i.e.
\[
\text{Card}(C_{1})=\dbinom{N^{(t)}_{t-a}-1}{|n|-1}\frac{\left(|n|-1\right)!}{\prod_{i=1}^{N}\left(n_{i}-\delta_{i,1} \right)!}.
\]
Finally,
\begin{align*}
I_{1}&=\int_{0}^{t}\theta W(a)^{|n|-1} \mathbb{P}_{a}\left(Z_{0}(t)=k_{1}\right)\prod_{i=1}^{N}\left(\int_{0}^{a}\frac{\theta e^{-\theta s}}{W_{\theta}(s)^{2}}\left(1-\frac{1}{W_{\theta}(s)} \right)^{k_{i}}ds\right)^{n_{i}-\delta_{i,1}} \frac{\mathbb{E}_{a}\left[\left(N^{(t)}_{t-a}\right)_{(|n|)}\right]}{\prod_{i=1}^{N}\left(n_{i}-\delta_{i,1}\right)!} da,
\end{align*}
where $(x)_{(|n|)}$ is the falling factorial of order $|n|$.
Since, $N^{(t)}_{t-a}$ is geometrically distributed under $\mathbb{P}_{t}$ with parameter $\frac{W(t)}{W(a)}$, it follows that
\begin{multline*}
I_{1}=\frac{|n|!W(t)^{|n|}}{\prod_{i=1}^{N}\left(n_{i}-\delta_{i,1}\right)!} \int_{0}^{t}\theta \frac{e^{-\theta a}}{W_{\theta}(a)^{2}}\left(1-\frac{1}{W_{\theta}(a)}\right)^{k_{1}-1}\prod_{m=1}^{N}\left(\int_{0}^{a}\frac{\theta e^{-\theta s}}{W_{\theta}(s)^{2}}\left(1-\frac{1}{W_{\theta}(s)} \right)^{k_{i}}ds\right)^{n_{i}-\delta_{i,1}} da\\+\mathcal{O}\left(tW(t)^{|n|-1}\right)
\end{multline*}
{\bf Step 4: Analysis of $C_{2}$.}

We denote
\begin{equation}
I_{2}:=\mathbb{E}_{t}\int_{0}^{t}N^{(t)}_{t-a}\sum_{\mathbf{n}\in C_{2}}\mathbb{P}_{a}\left(Z_{0}(a)=k_{1}\right)\prod_{m=1}^{N^{(t)}_{t-a}-1}\mathbb{E}_{a}\left[\prod_{i=1}^{N}\dbinom{A(k_{i},a)}{n_{m}^{i}}\right]da.\\
\end{equation}
Now, since
\[
\text{Card}(C_{2})=\mathcal{O} \left(\left(N^{(t)}_{t-a}\right)^{|n|-2} \right),
\]
we have using estimation \eqref{eq: majo},
\begin{align*}
I_{2}&\leq\int_{0}^{t}N^{(t)}_{t-a}\sum_{\mathbf{n}\in C_{2}}\mathcal{C}W(a)^{|n|-1}da\\&\leq\widetilde{\mathcal{C}} \int_{0}^{t}\left(N^{(t)}_{t-a}\right)^{|n|-1}W(a)^{|n|-1}da,
%&\leq  \int_{0}^{t}N^{(t)}_{t-a}\mathcal{C}W(a)^{|n|-1}\left(\text{Card}\left(\mathcal{M}_{N^{(t)}_{t-a}}\right)-\text{Card}\left(C_{1}\right)\right)da\\
%&\leq \int_{0}^{t}N^{(t)}_{t-a}\mathcal{C}W(a)^{|n|-1}\left(\frac{\prod_{i=1}^{N}\left(N^{(t)}_{t-a}-1+n_{i}-\delta_{i,1} \right)_{\left(n_{i}-\delta_{i,1}\right)}}{\prod_{i=1}^{N}\left(n_{i}-\delta_{i,1}\right)!}-\frac{\left(N^{(t)}_{t-a}-1\right)_{(|n|-1)}}{\prod_{i=1}^{N}\left(n_{i}-\delta_{i,1}\right)!}\right)da.
\end{align*}
for some positive real constant $\widetilde{\mathcal{C}}$.
%Hence, there exists a polynomial $\mathcal{P}$ with degree $|n|-1$ such that
%\[
%I_{2}\leq \int_{0}^{t}\mathcal{P}\left( N^{(t)}_{t-a}\right)\mathcal{C}W(a)^{|n|-1}da.
%\]
Using that $N^{(t)}_{t-a}$ is geometrically distributed with parameter $\frac{W(t)}{W(a)}$, it follows that there exists a positive real number $\hat{C}$ such that
\[
I_{2}\leq \hat{C}\int_{0}^{t}\left(\frac{W(t)}{W(a)} \right)^{|n|-1}W(a)^{|n|-1}da.
\]
Which imply that,
\[
I_{2}=\mathcal{O}\left(tW(t)^{|n|-1} \right).
\]
{\bf Step 5: Analysis of $C_{3}$.}

In the case where there is a positive $n^{i}_{1}$ ($C_{3}$ case), using that
\[
\mathbb{E}_{a}\left[\prod_{i=1}^{N}\dbinom{A(k_{i},a)}{n^{i}_{1}}\mathds{1}_{Z_{0}(a)=k_{l}} \right]\leq\mathbb{E}_{a}\left[\prod_{i=1}^{N}\dbinom{A(k_{i},a)}{n^{i}_{1}}\right],
\]
we have,
\begin{align*}
\int_{0}^{t}&N^{(t)}_{t-a}\sum_{\mathbf{n}\in C_{3}}\mathbb{E}_{a}\left[\prod_{i=1}^{N}\dbinom{A(k_{i},a)}{n^{i}_{1}}\mathds{1}_{Z_{0}(a)=k_{l}}\right]\prod_{m=2}^{N^{(t)}_{t-a}}\mathbb{E}_{a}\left[\prod_{i=1}^{N}\dbinom{A(k_{i},a)}{n_{m}^{i}}\right]da,
\\ &\leq\int_{0}^{t}N^{(t)}_{t-a}\sum_{\mathbf{n}\in C_{3}}\mathbb{E}_{a}\left[\prod_{i=1}^{N}\dbinom{A(k_{i},a)}{n^{i}_{1}}\right]\prod_{m=2}^{N^{(t)}_{t-a}}\mathbb{E}_{a}\left[\prod_{i=1}^{N}\dbinom{A(k_{i},a)}{n_{m}^{i}}\right]da,
\end{align*}
which is very similar to the the other steps. 
This term is $\mathcal{O}\left(tW(t)^{|n|-1}\right)$ because the condition $\sum_{i}n_{1}^{i}>0$ reduces the number of terms in the multi-sum.
Indeed,
%\begin{align*}
%\sum_{\underset{\sum_{i}n_{1}^{i}>0}{n^{i}_{1}+\dots+n^{i}_{N^{(t)}_{t-a}}=n-\delta_{i,l}}}= \sum_{j=1}^{|n|-1}\sum_{\underset{\sum_{i}n_{1}^{i}=j}{n^{i}_{1}+\dots+n^{i}_{N^{(t)}_{t-a}}=n-\delta_{i,l}}},\\
%\leq \left(|n|-1\right) \sum_{\underset{\sum_{i}n_{1}^{i}=1}{n^{i}_{1}+\dots+n^{i}_{N^{(t)}_{t-a}}=n-\delta_{i,l}}},\\
%\leq \left(|n|-1\right) \sum_{j=1}^{N}\sum_{n^{i}_{1}+\dots+n^{i}_{N^{(t)}_{t-a}-\delta_{i,j}}=n-\delta_{i,l}}.\\
%\end{align*}
\begin{align*}\text{Card}(C_{3})
%\sum_{n^{i}_{1}+\dots+n^{i}_{N^{(t)}_{t-a}}=n_{i}-\delta_{i,l}\ s.t. \ \sum_{i}n_{1}^{i}>0}1
=&\sum_{j_{1:N}=0\ s.t. \ \sum_{i} j_{i}>0}^{n_{1:N}-\delta_{1:N,1}} \quad \sum_{n^{i}_{2}+\dots+n^{i}_{N^{(t)}_{t-a}}=n_{i}-\delta_{i,l}-j_{i}}1\\
=&\sum_{j_{1:N}=0\ s.t. \ \sum_{i} j_{i}>0}^{n_{1:N}-\delta_{1:N,1}}\quad \prod_{i=1}^{N}\frac{\prod_{i=1}^{N}\left(N^{(t)}_{t-a}-1+n_{i}-\delta_{i,1}-j_{i} \right)_{\left(n_{i}-\delta_{i,1}-j_{i}\right)}}{\prod_{i=1}^{N}\left(n_{i}-\delta_{i,1}-j_{i}\right)!}\\
\leq& \mathcal{C}\sum_{j_{1:N}=0\ s.t. \ \sum_{i} j_{i}>0}^{n_{1:N}-\delta_{1:N,1}}\quad\left(N^{(t)}_{t-a} \right)^{|n|-1-\sum j_{i}}.\\
\end{align*}
Then, the expectation of the last quantity gives a polynomial of degree $|n|-1$ in $\frac{W(t)}{W(a)}$. Using the same study as $I_{2}$ shows that this part is of order $\mathcal{O}\left(t W(t)^{|n|-1}\right)$.

Finally, summing over $l$ ends the proof since the leading term is
\[
\sum_{l=1}^{N}\frac{|n|!W(t)^{|n|}}{\prod_{i=1}^{N}\left(n_{i}-\delta_{i,1}\right)!} \int_{0}^{t}\theta \frac{e^{-\theta a}}{W_{\theta}(a)^{2}}\left(1-\frac{1}{W_{\theta}(a)}\right)^{k_{l}-1}\prod_{m=1}^{N}\left(\int_{0}^{a}\frac{\theta e^{-\theta s}}{W_{\theta}(s)^{2}}\left(1-\frac{1}{W_{\theta}(s)} \right)^{k_{m}}ds\right)^{n_{m}-\delta_{m,1}} da,\\
\]
while the rest is a finite sum of $\mathcal{O}\left(t W(t)^{|n|-1}\right)$-terms.
By Lemma \ref{lem: asyComp},
\[
c_{k}=\int_{0}^{t}\frac{\theta e^{-\theta s}}{W_{\theta}(s)^{2}}\left(1-\frac{1}{W_{\theta}(s)} \right)^{k-1}ds+\mathcal{O}\left(e^{-\gamma t}\right),
\]
where $\gamma$ is equal to $\theta$ (resp. $2\alpha-\theta$) in the clonal critical and subcritical cases (resp. supercritical case). Hence, we deduce \eqref{eq:expo}.
\end{proof}
\begin{rem}
Taking the behavior of $\mathbb{P}\left(Z_{0}(a)=k\right)$ into account and using the Cauchy-Schwartz inequality for $\mathbb{E}\left[A(k,a)\mathds{1}_{Z_{0}(a)=\l}\right]$ one could actually prove that the error term in \eqref{eq:expo} is of order $\mathcal{O}\left(W(t)^{|n|-1}\right)$ in the clonal sub-critical and super-critical cases, and  $\mathcal{O}\left(\log t \  W(t)^{|n|-1}\right)$ in the clonal critical case.
\end{rem}
\begin{cor}
\label{thm: CinDspc}
We have, conditionally on non-extinction,
\[\lim\limits_{t\to \infty}\left(\frac{A(k,t)}{W(t)} \right)_{k\geq 1}= \mathcal{E}\left(c_{k}\right)_{k\geq 1} \qquad in \ distribution, \]
where $\mathcal{E}$ is an exponential random variable with parameter $1$.
\end{cor}
\begin{proof}
From Lemma \ref{lem: asymptMoment}, we have 
\[
\lim\limits_{t\to\infty}W(t)^{-|n|}\ \mathbb{E}_{t}\left[\prod_{i=1}^{K}A\left(k_{i},t\right)^{n_{i}}\right]=|n|!\prod_{i=1}^{N}c_{k_{i}}^{n_{i}}.
\]
Since the finite dimensional law of the process $\mathcal{E}\left(c_{k}\right)_{k\geq 1}$ is fully determined by its moments, it follows from the multidimensional moment problem (see \cite{KS}) and from the fact that the events $\left\{N_{t}>0 \right\}$ increase to the event of non-extinction, that we have the claimed convergence.

\end{proof}
\section{An elementary proof of the a.s.\ convergences of the frequency spectrum and the population counting process}
\label{sec:finalCV}
The goal of this section is to prove in Subsection \ref{subsec:cvFreq} the a.s.\ convergence of the frequency spectrum. \textcolor{red}{We begin by showing the law of large numbers for $N_{t}$}. We recall once again that we are in the supercritical case ($\alpha>0$).
\subsection{Convergence of the population counting process.}
Assume that $\alpha>0$, that is $W(t)\sim\frac{e^{\alpha t}}{\psi'(\alpha)}$.
The goal of this section is to prove the almost sure convergence of the population counting process. We first show that the convergence holds in probability, using the convergence of the process which counts at time $t$ the number $N^{\infty}_{t}$ of individuals having infinite descent. 
More formally, recalling that a splitting tree is a subset of $\mathbb{R}\times\left(\cup_{k\geq 0}\mathbb{N}^{k}\right)$ (see \cite{L10}), an individual $\left(u,t\right)$ in the tree $\mathbb{T}$ is said to have infinite descent at time $t$ if for any $T>t$ there exist $\tilde{u}$ in $\bigcup_{n\geq 0}\mathbb{N}^{n}$ such that $\left(T,u\tilde{u}\right)$ belong to $\mathbb{T}$. 

Finally, to obtain the almost sure convergence, we show in Theorem \ref{thm: convPop} that $N_{t}$ can not fluctuate faster than a Yule process.
\begin{prop}
\label{lem: cvYule}
Let $\left(N^{\infty}_{t}, \ t\in\mathbb{R}_{+}\right)$ be the number of alive individuals at time $t$ having alive descendant at infinity. Then, under $\mathbb{P}_{\infty}$, $N^{\infty}$ is a Yule process with parameter $\alpha$.
\end{prop}
\begin{proof}Let $T,t\in\mathbb{R}_{+}$. Recalling that, for $T<t$, $N^{(T)}_{t}$ is the number of individuals at time $t$ who have alive children at time $T$, we extend the notation to $t>T$ by setting $N^{(T)}_{t}=0$ in this case. Fix $S$ a positive real number, we consider the quantity,
\[
\sup_{t\leq S}\left|N^{(T)}_{t}-N^{\infty}_{t}\right|.
\]
\textcolor{red}{There exists a finite time $T^{S}$ such that $
N^{(T^{S})}_{S}=N^{\infty}_{S}.$ This means that the progeny of the all individuals alive at time $S$ who have finite descent are extinct at time $T^{S}$.
Moreover, $N^{(T^{S})}_{t}=N^{\infty}_{t}$ for all $t<S$, since, otherwise, there would exist
%suppose that there is a $t<S$ such that $
%N^{(T^{S})}_{t}>N^{\infty}_{t}.$
%Then there is 
an individual at time $t$ who has alive descent at time $T^{S}$ but which do not have an infinite descent. }

Hence,  for all $T>T^{S}$, $\sup_{t\leq S}\left|N^{(T)}_{t}-N^{\infty}_{t}\right|=0$. In particular, as $T\to\infty$, $N^{(T)}$ converges to $N^{\infty}$ a.s.\ for the Skorokhod topology of $\mathbb{D}\left[0,\infty \right)$ and $N^{\infty}$ is a.s.\ càdlàg.

Now, it remains to derive from $N^{(T)}$ the law of the process $N^{\infty}$. Let $0<s_{1}<s_{2}<\dots<s_{n}<T$. By a recursive use of Proposition \ref{lem: construction}, we see that, under $\mathbb{P}_{T}$, the process $\left(N^{(T)}_{s_{l}},\ 1\leq l\leq n \right)$ is a time inhomogeneous Markov chain with geometric initial distribution with parameter
 \[
 \mathbb{P}_{t}\left(H>T\mid H>T-s_{1} \right),
 \]
 and the law of $N^{(T)}_{s_{l}}$ given $N^{(T)}_{s_{l-1}}$ is the law of a sum of $N^{(T)}_{s_{l-1}}$ i.i.d.\ geometric random variable with parameter
 \[
 p_{l}=\mathbb{P}\left(H>T-s_{l-1}\mid H>T-s_{l} \right),
 \]
 i.e.\ a binomial negative with parameters $N^{(T)}_{s_{l-1}}$ and $1-p_{l}$.
 Hence,
 \begin{align*}
 \mathbb{P}_{t}&\left(N^{(T)}_{s_{1}}=m_{1},\dots,N^{(T)}_{s_{n}}=m_{n} \right)=p_{1}\left(1-p_{1} \right)^{m_{1}-1}\prod_{i=2}^{n}\dbinom{m_{i}+m_{i-1}-1}{m_{i}}p_{i}^{m_{i-1}}\left(1-p_{i} \right)^{m_{i}-1}.
 \end{align*}
 Moreover, we have, by Lemma \ref{lem: asyComp},
 \[
 p_{1}=\frac{W(T-s_{1})}{W(T)}\underset{t\to\infty}{\longrightarrow}e^{-\alpha s_{1}},
 \]
 and
 \[
 p_{l}=\frac{W(T-s_{l})}{W(T-s_{l-1})}\underset{t\to\infty}{\longrightarrow}e^{-\alpha\left(s_{l}-s_{l-1} \right)}.
 \]
 This leads to,
 \begin{align*}
 \mathbb{P}_{t}&\left(N^{(T)}_{s_{1}}=m_{1},\dots,N^{(T)}_{s_{n}}=m_{n} \right)\\&\quad\quad\quad\underset{t\to\infty}{\longrightarrow}e^{-\alpha s_{1}}\left(1-e^{-\alpha s_{1}} \right)^{m_{1}-1}\prod_{i=2}^{n}\dbinom{m_{i}+m_{i-1}-1}{m_{i}}e^{-\alpha m_{i-1}\left(s_{l}-s_{l-1} \right)}\left(1-e^{-\alpha\left(s_{l}-s_{l-1} \right)} \right)^{m_{i}-1}.
 \end{align*}
 Since the right hand side term corresponds to the finite dimensional distribution of a Yule process with parameter $\alpha$, this concludes the proof.
 \end{proof}

 As $N^{\infty}$ is a Yule process, $e^{-\alpha t}N^{\infty}_{t}$ converges a.s.\ to an exponential random variable of parameter $1$, denoted $\mathcal{E}$ hereafter, when $t$ goes to infinity (see for instance \cite{AN}).
 \begin{figure}[ht]
 \unitlength 1.3mm % = 5.69pt
 \linethickness{0.8pt}
 \center{
 \unitlength 1mm
 \begin{picture}(100,60)(0,0)
 %%%%%%%%%%%%%AXES%%%%%%%%%%%%%%%%%%%
 %\put(7,3){\vector(1,0){10}}
 %\put(17,3){\vector(-1,0){10}}
 %\put(62,3){\vector(1,0){5}}
 %\put(67,3){\vector(-1,0){5}}
 %\put(93,3){\vector(1,0){20}}
 %\put(113,3){\vector(-1,0){20}}
 %\put(10,0){\makebox{$\scriptscriptstyle{H_{1}}$}}
 %\put(63,0){\makebox{$\scriptscriptstyle{H_{2}}$}}
 %\put(100,0){\makebox{$\scriptscriptstyle{H_{3}}$}}
 \put(0,4){
 \put(0,24){\makebox{$t$}}
 \put(2,0){
 \put(0,0){\line(1,0){100}}
 \put(0,25){\line(1,0){100}}
 \put(0,0){\line(0,1){45}}
 %%%%%%%%%%%%%%%%%%%%
 %R AVANT NAISSANCE
 
 %\multiput(136,0)(0,1){25}{\line(0,1){0.5}}
 %
 %
 %
 %\multiput(15,0)(0,1){15}{\line(0,1){0.5}}
 %\multiput(65,0)(0,1){20}{\line(0,1){0.5}}
 %\multiput(111,0)(0,1){5}{\line(0,1){0.5}}
 %\multiput(154,0)(0,1){7}{\line(0,1){0.5}}

 %\multiput(154,7)(-1,0){18}{\line(0,1){0.5}}
 
 %\multiput(136,7)(0,1){18}{\line(0,1){0.5}}
 %
 
 \put(5,0){\line(1,0){10}}
 \put(60,0){\line(1,0){5}}
 \put(91,0){\line(1,0){20}}
 %\put(136,0){\line(1,0){18}}
 %\color{black}
 %
 %
 %\put(1,17){\makebox{$\scriptscriptstyle{H_{1}}$}}
 %\put(56,22){\makebox{$\scriptscriptstyle{H_{2}}$}}
 %\put(87,15){\makebox{$\scriptscriptstyle{H_{3}}$}}
 %
 
 \put(0,20){\line(1,-1){15}}%15,15
 \put(14,25){\vector(0,1){10}}
 \put(14,35){\vector(0,-1){10}}
 \put(10,30){\makebox{$\scriptscriptstyle{O_{1}}$}}
 \multiput(15,5)(0,1){30}{\line(0,1){0.1}}%15,35

 \put(15,25){\line(1,-1){5}}%20,30
 \multiput(20,20)(0,1){10}{\line(0,1){0.1}}%20,40
 \put(19,25){\vector(0,1){5}}
 \put(19,30){\vector(0,-1){5}}
 \put(15,27){\makebox{$\scriptscriptstyle{O_{2}}$}}
 
 %\color{black}
 %\put(20,25){\line(1,-1){5}}%25,35
 %\multiput(25,35)(0,1){10}{\line(0,1){0.1}}%25,45
 %\put(25,45){\line(1,-1){10}}%35,35
 %\put(35,35){\line(1,-1){2}}%37,33
 \color{black}
 \put(20,25){\line(1,-1){5}}%25,35
 \multiput(25,20)(0,1){10}{\line(0,1){0.1}}%25,45
 \put(25,25){\line(1,-1){10}}%35,35
 \put(24,25){\vector(0,1){5}}
 \put(24,30){\vector(0,-1){5}}
 \put(20,27){\makebox{$\scriptscriptstyle{O_{3}}$}}
 
 \multiput(35,15)(0,1){5}{\line(0,1){0.1}}
 \put(35,20){\line(1,-1){2}}%37,33
 
 %\color{black}
 
 \multiput(37,18)(0,1){5}{\line(0,1){0.1}}%37,38
 \put(37,23){\line(1,-1){4}}%41,34
 
 \put(40,25){\vector(0,1){5}}
 \put(40,30){\vector(0,-1){5}}
 \put(36,27){\makebox{$\scriptscriptstyle{O_{4}}$}}
 \multiput(41,19)(0,1){10}{\line(0,1){0.1}}%41,44
 \put(41,25){\line(1,-1){10}}%51,34
 \put(51,15){\line(1,-1){1}}%52,33
 \put(52,14){\line(1,-1){3}}%55,30
 \put(55,11){\line(1,-1){10}}%65,20
 \multiput(65,1)(0,1){10}{\line(0,1){0.1}}%65,30
 \put(65,11){\line(1,-1){2}}%67,28
 
 \put(66,25){\vector(0,1){10}}
 \put(66,35){\vector(0,-1){10}}
 \put(62,30){\makebox{$\scriptscriptstyle{O_{5}}$}}
 \multiput(67,10)(0,1){25}{\line(0,1){0.1}}%67,43
 \put(67,25){\line(1,-1){5}}%72,38
 \multiput(72,20)(0,1){6}{\line(0,1){0.1}}%72,44
 \put(72,25){\line(1,-1){6}}%78,38
 \put(78,19){\line(1,-1){10}}%88,28
 \put(88,9){\line(1,-1){8}}%96,20
 \put(96,1){\line(1,-1){1}}%101,15
 %\put(101,15){\line(1,-1){10}}%111,5
 %\multiput(111,5)(0,1){15}{\line(0,1){0.1}}%111,20
 %\put(111,20){\line(1,-1){3}}%114,17
 %\multiput(114,17)(0,1){20}{\line(0,1){0.1}}%114,37
 %\put(114,37){\line(1,-1){4}}%118,33
 %\multiput(118,33)(0,1){10}{\line(0,1){0.1}}%118,43
 %\put(118,43){\line(1,-1){10}}%128,33
 %\put(128,33){\line(1,-1){16}}%144,17
 %\put(144,17){\line(1,-1){10}}%154,7
 %\multiput(154,7)(0,1){10}{\line(0,1){0.1}}%154,17
 %\put(154,17){\line(1,-1){7}}%161,10
 %\multiput(161,10)(0,1){4}{\line(0,1){0.1}}%161,14
 %\put(161,14){\line(1,-1){4}}%165,10
 %\put(165,10){\line(1,-1){3}}%168,7
 %\put(168,7){\line(1,-1){2}}%170,5
 %\put(170,5){\line(1,-1){5}}%175,0
 %%
 \linethickness{0.2pt}
 
 }}
 \end{picture}}
 \caption{Reflected JCCP with overshoot over $t$. Independence is provided by the Markov property. }
 
 \end{figure}
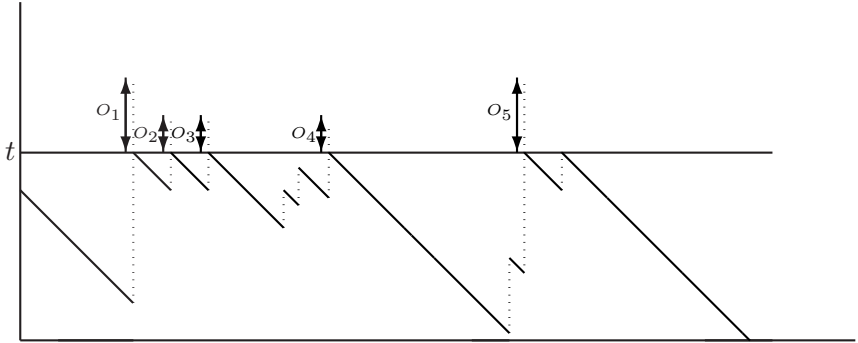

\begin{thm}
\label{thm: convPop}
We have
\[
e^{-\alpha t}N_{t}\overset{a.s.}{\underset{t\to\infty}{\longrightarrow}}\frac{1}{\psi'(\alpha)}\ \mathcal{E},\quad a.s.\ \text{and } L^{2},
\]
where $\mathcal{E}$ is the a.s.\ limit of $e^{-\alpha t}N^{\infty}_{t}$ which is an exponential random variable with parameter one under $\mathbb{P}_{\infty}$.

\end{thm}
\begin{proof}

We first look at the quantity,
\[
\mathbb{E}_{t}\left[e^{-2\alpha t}\left(N^{\infty}_{t}-\psi'(\alpha)N_{t} \right)^{2} \right].
\]
First note that $N^{\infty}_{t}$ can always be written as a sum of Bernoulli trials,
\begin{equation}
\label{eq:berntrial}
N^{\infty}_{t}=\sum_{i=1}^{N_{t}}B^{(t)}_{i},
\end{equation}
corresponding to the fact that the $i$th individual has infinite descent or not.

Now, by construction of the splitting tree, the descent of each individual alive at time $t$ can be seen as a (sub-)splitting tree where the lifetime of the root follows a particular distribution (that is the law of the residual lifetime of the corresponding individual). We denote by $O_{i}$ the residual lifetime of the $i$th individual. In particular, these subtrees are dependent only through the residual lifetimes $\left(O_{i}\right)_{1 \leq i\leq N_{t}}$ of the individuals. \textcolor{red}{ Hence, the random variables  $\left(B^{(t)}_{i} \right)_{i\geq 2}$ are independents conditionally on the family $\left(O_{i}\right)_{1 \leq i\leq N_{t}}$. In addition, the family $\left(O_{i}\right)_{1 \leq i\leq N_{t}}$ has independence properties under $\mathbb{P}_{t}$.}
\begin{lem}
 \label{lem:residual}Under $\mathbb{P}_{t}$, the family $\left(O_{i},\ i\in\llbracket1,N_{t}\rrbracket\right)$ forms a family of independent random variables, independent of $N_{t}$, and, except $O_{1}$, having the same distribution.
\end{lem}
The proof of this lemma is postponed at the end of this subsection. Hence, it follows that, under $\mathbb{P}_{t}$, the random variables $(B^{(t)}_{i})_{1 \geq i\geq N_{t}}$ are independent and identically distributed for $i\geq 2$ (in the sense of Remark \ref{rem:agoodrem}). Let us denote by $\hat{p}_{t}$ the parameter of $B_{1}^{(t)}$, and by $p_{t}$ the common parameter of the others i.i.d.\ Bernoulli random variables.
It follows from \eqref{eq:berntrial} that
\[
\mathbb{E}_{t}\left[N^{\infty}_{t}\right]=p_{t}\left(W(t)-1\right)+\hat{p}_{t}
\]
and from the Yule nature of $N^{\infty}$ under $\mathbb{P}_{\infty}$ (Proposition \ref{lem: cvYule}) that
$
\mathbb{E}_{\infty}\left[N^{\infty}_{t} \right]=e^{\alpha t}$.
%Let $O^{t}_{+}$ be a random variable whose law if the law of the overshoot above $t$ of a Lévy process with Laplace exponent $\psi$ conditioned to hit $t$ before $0$.
%Then,
%\[
%1-p_{t}=\mathbb{P}_{t}\left( B^{(T,t)}_{i}=0\right)=\mathbb{E}\mathbb{P}_{O_{+}}\left(ext \right)=\mathbb{E}e^{-\alpha O_{+}}.
%\]

Now, since
\begin{align*}
\mathbb{E}_{\infty}\left[N^{\infty}_{t}\right]=\mathbb{E}_{t}\left[N^{\infty}_{t}\right]\frac{\mathbb{P}\left(N_{t}>0\right)}{\mathbb{P}\left(\text{Non-ex}\right)},
\end{align*}
we have
\begin{align*}
e^{\alpha t}=\left(p_{t}\left(W(t)-1\right)+\hat{p}_{t}\right)\frac{\mathbb{P}\left(N_{t}>0\right)}{\mathbb{P}\left(\text{Non-ex}\right)}.
\end{align*}
We recall from \cite{L10} that,
\[
\mathbb{P}\left(\text{Non-ex} \right)=\mathbb{E}\left[e^{-\alpha V} \right],
\]
and
\[
\mathbb{P}\left(N_{t}>0\right)=\mathbb{E}\left[\frac{W(t-V)}{W(t)}\right],
\]
\textcolor{red}{where $V$ is a random variable with law $\mathbb{P}_{V}$ (i.e.\ the lifetime of a typical individual).}
It then follows, from Lesbegue Theorem that,
\begin{equation}
\label{eq:estime3}
\frac{\mathbb{P}\left(N_{t}>0\right)}{\mathbb{P}\left(\text{Non-ex}\right)}-1= \mathcal{O}\left(e^{-\beta t} \right), 
\end{equation}
with $\beta=\alpha\wedge\gamma$ where the constant $\gamma$ is given by Lemma \ref{lem: asyComp}. Hence,
\begin{equation}
\label{eq:estime1}
p_{t}e^{-\alpha t}W(t)=1+\mathcal{O}\left(e^{-\beta t}\right).
\end{equation}
Now, using \eqref{eq:berntrial}, we have
\begin{equation}
\label{eq:estime2}
\mathbb{E}_{t}\left[N^{\infty}_{t}N_{t}\right]=\mathbb{E}_{t}\left[N_{t}\left(N_{t}-1 \right)\right]p_{t}+\hat{p}_{t}\mathbb{E}_{t}N_{t}=2W(t)^{2}p_{t}+\mathcal{O}\left(e^{\alpha t}\right),
\end{equation}
where the second equality comes from the fact that $N_{t}$ is geometrically distributed with parameter $W(t)^{-1}$ under $\mathbb{P}_{t}$.

Recalling also that $N^{\infty}_{t}$ is geometrically distributed with parameter $e^{-\alpha t}$ under $\mathbb{P}_{\infty}$, it follows that
%\begin{multline*} 
%\mathbb{E}_{t}\left[\left(N^{\infty}_{t}-\psi'(\alpha)N_{t} \right)^{2} \right]=2e^{2\alpha t}\left(\frac{\mathbb{P}\left(\text{Non-ex}\right)}{\mathbb{P}\left(N_{t}>0\right)}-1\right)+2\left(e^{\alpha t}-\psi'\left(\alpha \right)W(t)\right)^{2}\\+4\psi'(\alpha)e^{\alpha t}W(t)\left(W(t)e^{-\alpha t}p_{t}-1 \right)+\mathcal{O}\left(e^{\alpha t}\right).
%\end{multline*}
\begin{equation*} 
\mathbb{E}_{t}\left[\left(N^{\infty}_{t}-\psi'(\alpha)N_{t} \right)^{2} \right]=2e^{2\alpha t}\frac{\mathbb{P}\left(\text{Non-ex}\right)}{\mathbb{P}\left(N_{t}>0\right)}-4\psi'(\alpha)W(t)^{2}p_{t}+2\psi'(\alpha)^{2}W(t)^{2}+\mathcal{O}\left(e^{\alpha t}\right).
\end{equation*}
Hence, it follows from \eqref{eq:estime1}, \eqref{eq:estime2}, \eqref{eq:estime3} and Lemma \ref{lem: asyComp}, that

\begin{equation}
\label{eq:123}
\mathbb{E}_{t}\left[e^{-2\alpha t}\left(N^{\infty}_{t}-\psi'(\alpha)N_{t} \right)^{2} \right]=\mathcal{O}\left(e^{-\beta t}\right).
\end{equation}
Let us define now, for all integer $n$, $
t_{n}=\frac{2}{\beta}\log n.
$
Then, by the previous estimation, it follows from Borel-Cantelli lemma and a Markov-type inequality that,
\begin{equation}
\label{eq:cvsubseq}
\lim\limits_{n\to\infty}e^{-\alpha t_{n}}N_{t_{n}}=\psi'(\alpha)\mathcal{E},\quad a.s.,
\end{equation}
on the survival event.
%Moreover, we have,
%\begin{align*}
%&\mathbb{P}_{t_{n}}\left(\sup_{s\in[t_{n},t_{n+1}]} \left|e^{-\alpha_{t_{n}}}N_{t_{n}}-e^{-\alpha s}N_{s} \right|>\epsilon \right)\\
%&\leq\mathbb{P}_{t_{n}}\left(\sup_{s\in[t_{n},t_{n+1}]} e^{-\alpha t_{n}}N_{t_{n}}-e^{-\alpha s}N_{s} >\epsilon\right)+\mathbb{P}_{t_{n}}\left(\inf_{s\in[t_{n},t_{n+1}]} e^{-\alpha t_{n}}N_{t_{n}}-e^{-s}N_{s} <-\epsilon\ e^{\alpha t_{n}}\right),\\
%\end{align*}
From this point, we need to control the fluctuation of $N$ between the times $\left(t_{n}\right)_{n\geq1}$. The births can be controlled by comparisons with a Yule process, but the deaths are harder to control. For this, we use that, by \eqref{eq:cvsubseq},  $e^{-\alpha t_{n+1}}N_{t_{n+1}}-e^{-\alpha t_{n}}N_{t_{n}}$ is small, for $n$ large. It then follows that, if the quantity
\[
\inf_{s\in[t_{n},t_{n+1}]} e^{-\alpha t_{n}}N_{t_{n}}-e^{-\alpha s} N_{s},
\]
takes very low negative values,
then
\[
\sup_{s\in[t_{n},t_{n+1}]} e^{-\alpha s}N_{s}-e^{-\alpha t_{n+1}}N_{t_{n+1}},
\]
must take very high positive value. More precisely,
\begin{multline*}
\mathbb{P}_{t_{n}}\left(\sup_{s\in[t_{n},t_{n+1}]} \left|e^{-\alpha_{t_{n}}}N_{t_{n}}-e^{-\alpha s}N_{s} \right|>\epsilon \right)\leq
\mathbb{P}_{t_{n}}\left(\sup_{s\in[t_{n},t_{n+1}]} e^{-\alpha s}N_{s}-e^{-\alpha t_{n}}N_{t_{n}} >\epsilon\ \right)\\+\mathbb{P}_{t_{n}}\left(e^{-\alpha t_{n}}N_{t_{n}}-e^{-\alpha t_{n+1}}N_{t_{n+1}}+\sup_{s\in[t_{n},t_{n+1}]} e^{-\alpha t_{n+1}}N_{t_{n+1}}-e^{-\alpha s}N_{s} >\epsilon\ \right)\\
\leq 
\mathbb{P}_{t_{n}}\left(\sup_{s\in[t_{n},t_{n+1}]} e^{-\alpha s}N_{s}-e^{-\alpha t_{n}}N_{t_{n}} >\epsilon\ \right)
+\mathbb{P}_{t_{n}}\left(\sup_{s\in[t_{n},t_{n+1}]} e^{-\alpha t_{n+1}}N_{t_{n+1}}-e^{-\alpha s}N_{s} >\epsilon\ \right)\\
+\mathbb{P}_{t_{n}}\left(e^{-\alpha t_{n}}N_{t_{n}}-e^{-\alpha t_{n+1}}N_{t_{n+1}} >\epsilon\ \right)
\end{multline*}
Now, there exists a Yule process $Y$ with parameter $b$ such that $Y_{0}=N_{t_{n}}$ and for all $s$ in $[0,t_{n+1}-t_{n}] $,
\begin{equation}
\label{eq:yulemaj}
 N_{t_{n}}-N_{s} \leq Y_{s-t_{n}}-Y_{0}, \quad a.s.
\end{equation}
This Yule process can be constructed from the population at time $t_{n}$ by extending the lifetimes of all individuals to infinity, and constructing births from the same Poisson process as in the splitting tree.
This leads to
\begin{multline*}
\mathbb{P}_{t_{n}}\left(\sup_{s\in[t_{n},t_{n+1}]} \left|e^{-\alpha_{t_{n}}}N_{t_{n}}-e^{-\alpha s}N_{s} \right|>\epsilon \right)\leq \mathbb{P}_{t_{n}}\left(\sup_{s\in[t_{n},t_{n+1}]} Y_{s-t_{n}}-Y_{0} >\epsilon\ e^{\alpha t_{n}}\right)\\+\mathbb{P}_{t_{n}}\left(\sup_{s\in[t_{n},t_{n+1}]} Y_{t_{n+1}-t_{n}}-Y_{s-t_{n}} >\epsilon\ e^{\alpha t_{n}}\right)+\mathbb{P}_{t_{n}}\left(e^{-\alpha t_{n}}N_{t_{n}}-e^{-\alpha t_{n+1}}N_{t_{n+1}} >\epsilon\ \right)
\\
\leq2\ \mathbb{P}_{t_{n}}\left( Y_{t_{n+1}}-Y_{t_{n}} >\epsilon\ e^{\alpha t_{n}}\right)+\mathbb{P}_{t_{n}}\left(e^{-\alpha t_{n}}N_{t_{n}}-e^{-\alpha t_{n+1}}N_{t_{n+1}} >\epsilon\ \right).
\end{multline*}
Since Markov inequalities are not precise enough to go further, we need to compute exactly the probability,
\[
\mathbb{P}_{t_{n}}\left( Y_{t_{n+1}-t_{n}}-Y_{0} >\epsilon\ e^{\alpha t_{n}}\right).
\]
From the branching and Markov properties, $Y_{t_{n+1}-t_{n}}-Y_{0}$ is a sum of a geometric number, with parameter $W(t_{n})^{-1}$, of independent and i.i.d.\ geometric random variables supported on $\mathbb{Z}_{+}$ with parameter $e^{-b\left(t_{n+1}-t_{n}\right)}$. Hence, $Y_{t_{n+1}-t_{n}}-Y_{0}$ is geometric supported on $\mathbb{Z}_{+}$ with parameter
\[
\frac{e^{-b\left(t_{n+1}-t_{n}\right)}}{W(t_{n})\left(1- e^{-b\left(t_{n+1}-t_{n}\right)}\left(1-\frac{1}{W(t_{n})} \right)\right)},
\]
and, we have
\[
\mathbb{P}_{t_{n}}\left( Y_{t_{n+1}-t_{n}}-Y_{0} \geq k\right)=\left(1-\frac{1}{W(t_{n})\left( e^{b\left(t_{n+1}-t_{n}\right)}-1\right)+1}\right)^{k}.
\]
Using
\[
W(t_{n})=\mathcal{O}\left(e^{\alpha t_{n}}\right)=\mathcal{O}\left(n^{\frac{2\alpha}{\beta}} \right),
\]
we have
\[
W(t_{n})\left(e^{b\left(t_{n+1}-t_{n}\right)}-1 \right)=\mathcal{O}\left(n^{\frac{\alpha}{2\beta}-1} \right).
\]
Finally,
\[
\mathbb{P}_{t_{n}}\left( Y_{t_{n+1}-t_{n}}-Y_{0} >\epsilon e^{\alpha t_{n}}\right)\leq\left(1-\frac{1}{1+\mathcal{C}n^{\frac{\alpha}{2\beta}-1}} \right)^{n^{\frac{\alpha}{2\beta}}},
\]
for some positive real constant $\mathcal{C}$. 
%frac{1-e^{-b\left(t_{n+1}-t_{n}\right)}}{1-e^{-b\left(t_{n+1}-t_{n}\right)}\left(1-\frac{1}{W(t_{n})}\right) }
Borel-Cantelli's Lemma then entails
\[
\lim\limits_{n\to\infty}\sup_{s\in[t_{n},t_{n+1}]} \left|e^{-\alpha_{t_{n}}}N_{t_{n}}-e^{-\alpha s}N_{s} \right|=0, \quad \text{almost surely},
\]
which ends the proof of the almost sure convergence.

\textcolor{red}{Now, for the convergence in $L^{2}$, we have that
\[
\mathbb{E}_{t}\left[\left(\psi'(\alpha)e^{-\alpha t}N_{t}-\mathcal{E}\right)^{2} \right]\leq 2\mathbb{E}_{t}\left[e^{-2\alpha t}\left(N^{\infty}_{t}-\psi'(\alpha)N_{t} \right)^{2} \right]+2\mathbb{E}_{t}\left[\left(e^{-\alpha t}N^{\infty}_{t}-\mathcal{E}\right)^{2} \right].
\]
The first term in the right hand side of the last inequality converges to $0$ according to \eqref{eq:123}. For the second term, since $N^{\infty}_{t}$ and $\mathcal{E}$ vanish on the extinction event, we have
\[
\lim\limits_{t\to\infty}\mathbb{E}_{t}\left[\left(e^{-\alpha t}N^{\infty}_{t}-\mathcal{E}\right)^{2} \right]=\lim\limits_{t\to\infty}\mathbb{E}_{\infty}\left[\left(e^{-\alpha t}N^{\infty}_{t}-\mathcal{E}\right)^{2} \right].
\]
The conclusion comes from the fact that $\left(e^{-\alpha t}N^{\infty}_{t},\ t\in\mathbb{R}_{+}\right)$ is a martingale uniformly bounded in $L^{2}$.}
\end{proof}

In the preceding proof, we postponed the demonstration of the independence of the residual lifetimes of the alive individuals at time $t$.  We give its proof now, which is quite similar to the Proposition 5.5 of \cite{L10}.

 \begin{proof}[Proof of Lemma \ref{lem:residual}]
 %Let $\left(Y^{(t)}_{s}, s\in\mathbb{R}_{+}\right)$ the JCCP of the splitting tree truncated above $t$, we know that $Y^{(t)}$ has the law of a L\'evy process
 Let $\left(Y^{(i)}\right)_{0\leq i\leq N_{t}}$ be a family of independent L\'evy processes with Laplace exponent
 \[
 \psi(x)=x-\int_{(0,\infty]}\left(1-e^{-rx}\right)\Lambda(dr), \ \ x\in\mathbb{R}_{+},
 \]
 conditioned to hit $(t,\infty)$ before $0$, for $i\in\left\{0,\dots,N_{t}-1 \right\}$, and conditioned to hit $0$ before $(t,\infty)$ for $i=N_{t}$.
 We also assume that,
 \[
 Y^{(0)}_{0}=t\wedge V,
 \]
 and
 \[
 Y^{(i)}_{0}=t,\quad i\in\left\{1,\dots,N_{t} \right\}.
 \]
 Now, denote by $\tau_{i}$ the exit time of the $i$th process of $(0,t)$ and
 \[
 T_{n}=\sum_{i=0}^{n-1}\tau_{i},\quad  n\in\left\{0,\dots,N_{t}+1 \right\}.
 \] 
 Then, the process defined for all $s\in[0,T_{N_{t}}]$ by
 \[
 Y_{s}=\sum_{i=0}^{N_{t}}Y^{(i)}_{s-T_{i}}\mathds{1}_{T_{i}\leq s< T_{i+1}},
 \]
 has the law of the contour process of a splitting tree cut under $t$. Moreover, the quantity $Y^{(i)}_{\tau_{i}}-Y^{(i)}_{\tau_{i}-}$ is the lifetime of the $i$th alive individual at time $t$.
 %We also assume that for $i\in\left\{1,\dots,N_{t}-1 \right\}$ the processes
 %On the event $\left\{N_{t}>0\right\}$ conditionally on $N_{t}$, $Y^{(t)}$ can then be constructed as a succession of independent excursion of L\'evy processes with Laplace exponent $\psi$, independent of $N_{t}$. 
 The family of residual lifetime $\left(O_{i}\right)_{1\leq i\leq N_{t}}$ has then the same distribution as the sequence of the overshoots of the contour above $u$.
 Thus, the independence of the L\'evy processes $Y^{(i)}$ ensures us that $\left(O_{i},\ i\in\llbracket2,N_{t}\rrbracket \right)$ is an i.i.d family of random variables, and that $O_{1}$ is independent of the other $O_{i}$'s.
 \end{proof}
\subsection{Convergence of the frequency spectrum}
\label{subsec:cvFreq}
We end up this section with the almost sure convergence of the frequency spectrum.
\begin{thm} We have,
\label{thm: akconv}
\[
e^{-\alpha t}\left(A\left(k,t\right) \right)_{k\geq 1}\underset{t\to\infty}{\longrightarrow}\frac{\mathcal{E}}{\psi'(\alpha)}\left(c_{k} \right)_{k\geq 1}, \quad a.s.\text{ and in }L^{2},
\]
where $\mathcal{E}$ is the same random variable as in Theorem \ref{thm: convPop}, and $c_{k}$ was defined in Proposition \ref{lem: cov}.
\end{thm}
\begin{proof}

Using \eqref{eq:momNtak} and the bound $\mathbb{E}\left[N_{a}\mathds{1}_{Z_{0}(a)=k}\right]\leq\mathbb{E}\left[N_{a}\right]$, it follows that
\[
\mathbb{E}_{t}\left[\left(c_{k}N_{t}-A(k,t) \right)^{2}\right]=2W(t)^{2}\left(\int_{t}^{\infty}\frac{\theta e^{-\theta a}}{W_{\theta}(a)^{2}}\left(1-\frac{1}{W_{\theta}(a)} \right)^{k-1}da\right)^{2}+\mathcal{O}\left(W(t)\right).
\]
Finally, it follows from Lemma \ref{lem: asyComp} that
\[
\mathbb{E}_{t}\left[e^{-2\alpha t}\left(c_{k}N_{t}-A(k,t) \right)^{2}\right]\underset{t\to\infty}{\sim}\mathcal{C}e^{-\gamma t},
\]
where $\gamma$ is equal to $\theta$ (resp. $2\alpha-\theta$) in the clonal critical and sub-critical cases (resp. supercritical case).

From this point we follow the proof of Theorem \ref{thm: convPop}, except that the Yule process used in \eqref{eq:yulemaj} must be replaced by another Yule process corresponding to the a binary fission every time an individual experiences a \emph{ birth or a mutation}, i.e. the new Yule process has parameter $b+\theta$.
Indeed, the process $A(k,t)$ can make a positive jump only in two cases: the first corresponding to the birth of an individual in a family of size $k-1$, the other one correspond to a mutation occurring on an individual in a family of size $k+1$.
\end{proof}
\label{sec:coalescent}

\section*{Acknowledgement} This work was partly founded by F\'ed\'eration Charles Hermite. The authors also wish to warmly thank anonymous referees for their many valuable comments and their careful reading.

\bibliographystyle{plain}
\bibliography{biblio.bib}
\end{document}